\documentclass[oneside,11pt,reqno]{amsart}
\usepackage{stmaryrd}
\usepackage{bbm}
\usepackage{mathrsfs}
\usepackage{enumerate}
\usepackage{latexsym,amsxtra}
\usepackage[dvips]{graphicx}
\usepackage{mathtools}
\usepackage{dsfont}
\usepackage{slashed}
\usepackage[all]{xy}
\usepackage{amscd,graphics}
\usepackage[usenames, dvipsnames]{color}
\usepackage{amsmath,amsfonts,amsthm,amssymb}
\usepackage{latexsym,amsmath}
\usepackage{graphicx,psfrag}
\usepackage{mathabx}
\usepackage{hyperref}
\usepackage{fancyhdr}

\newtheorem{thm}{Theorem}[section]
\newtheorem*{nota}{Notation}
\newtheorem*{claim}{Claim}
\newtheorem{assum}[thm]{Assumption}
\newtheorem{lem}[thm]{Lemma}
\newtheorem{cor}[thm]{Corollary}
\newtheorem{pro}[thm]{Proposition}

\theoremstyle{definition}
\newtheorem{defi}[thm]{Definition}
\newtheorem{ex}[thm]{Example}
\theoremstyle{remark}
\newtheorem{rmk}{Remark}
\title{The Seiberg-Witten equations on end-periodic manifolds and an obstruction to positive scalar curvature metrics}

\author{Jianfeng Lin}

\begin{document}
\maketitle
\begin{abstract}
By studying the Seiberg-Witten equations on end-periodic manifolds, we give an obstruction on the existence of positive scalar curvature metric on compact $4$-manifolds with the same homology as $S^{1}\times S^{3}$. This obstruction is given in terms of the relation between the Fr\o yshov invariant of the generator of $H_{3}(X;\mathds{Z})$ with the $4$-dimensional Casson invariant $\lambda_{SW}(X)$ defined in \cite{MRS}.
Along the way, we develop a framework that can be useful in further study of the Seiberg-Witten theory on general end-periodic manifolds.
\end{abstract}
\section{Introduction}
A natural question in Riemannian geometry is: When does a closed manifold $X$ admit a Riemannian metric with positive scalar curvature? (See \cite{rosenberg1994manifolds} for a survey on this problem. We call such manifolds ``psc-manifolds''.) The answer is fully understood in the following two cases:
\begin{itemize}
\item $X$ is $3$-dimensional or less \cite{perelman2003ricci};
\item $X$ is simply connected and $5$-dimensional or more \cite{gromov-Lawson,stolz1992simply}.
\end{itemize}

Now consider the case that $X$ is a $4$-dimensional psc-manifold. Then we have the following three constrains on the topology of $X$:
\begin{enumerate}[(i)]
\item Suppose $X$ is spin. Then the signature of $X$ (denoted by $\operatorname{sign}(X)$) must be zero. Similar result holds for its covering spaces \cite{hitchin1974harmonic,lichnerowicz1963spineurs};
\item Suppose $b_{3}(X)>0$. Then up to a nonzero multiple, any element of $H_{3}(X;\mathds{R})$ can be represented by an embedded, oriented psc $3$-manifold. Similar result holds for its covering spaces \cite{schoen1979structure};
\item Suppose $b_{2}^{+}(X)>1$. Then the Seiberg-Witten invariant $SW(X,\hat{\mathfrak{s}})$ must equal $0$ for any spin$^{c}$ structure $\hat{\mathfrak{s}}$. Similar result holds for its covering spaces \cite{Witten}.
\end{enumerate}

In the current paper, we consider the following case:
\begin{assum}\label{homology S3timesS1}
 $X$ is a $4$-manifold with the same homology as $S^{1}\times S^{3}$; the homology group $H_{3}(X;\mathds{Z})$ is generated by an embedded $3$-manifold $Y$ with $b_{1}(Y)=0$.
\end{assum}

For such $X$, condition (i) tells nothing interesting and condition (ii) provides a cobordism between $Y$ and a psc $3$-manifold. As for condition (iii), it can not be applied because the Seiberg-Witten invariants are not well defined (since $b_{2}^{+}(X)=0$).

The first purpose of the current paper is to obtain a new obstruction of positive scalar curvature in the direction of (iii). Recall that for $X$ satisfying Assumption \ref{homology S3timesS1}, although the original Seiberg-Witten invariant is not well defined, there are two other invariants from the Seiberg-Witten theory:
\begin{itemize}
\item The $4$-dimensional Casson-type invariant $\lambda_{SW}(X)$, defined by Mrowka-Ruberman-Saveliev \cite{MRS};
\item The Fr{\o}yshov invariant $\operatorname{h}(Y,\mathfrak{s})$, defined by Fr{\o}yshov \cite{Froyshov}, where $\mathfrak{s}$ is the unique spin structure on $Y$ that can be extended to a spin structure on $X$. (It was proved in \cite{Froyshov} that this invariant does not depend on the choice of $Y$.)
\end{itemize}
Here is the main theorem of the paper:
\begin{thm}\label{new obstruction}
Suppose $\lambda_{SW}(X)+\operatorname{h}(Y,\mathfrak{s})\neq 0$. Then $X$ admits no Riemannian metric with positive scalar curvature.
\end{thm}
\begin{rmk}
We conjecture one should be able to give an alternative proof of Theorem \ref{new obstruction} by combining Schoen-Yau's result \cite{schoen1979structure} and monopole Floer homology.\footnote{This has been verified \cite[Theorem B]{LRS} in a recent paper by the author, Ruberman and Saveliev.}
\end{rmk}

Since it was proved in \cite{MRS} that the mod-$2$ reduction of $\lambda_{SW}(X)$ is always $\rho(Y,\mathfrak{s})$ (the Rohlin invariant of $(Y,\mathfrak{s})$), we have the following corollary:
\begin{cor}\label{rohlin=froyshov}
Suppose $X$ is a homology $S^{3}\times S^{1}$ with $H_{3}(X;\mathds{Z})$ generated by an embedded rational homology sphere $Y$ satisfying $$
\operatorname{h}(Y,\mathfrak{s}) \not\equiv \rho(Y,\mathfrak{s})\ (\textrm{mod}\ 2).
$$
Then $X$ admits no Riemannian metric with positive scalar curvature.
\end{cor}
This corollary gives a large family of interesting examples of $4$-manifolds (with $b_{2}=0$) admitting no positive scalar curvature metric.
\begin{ex}
Let $X$ be obtained by furling up any homology cobordism from $Y=\Sigma(2,3,7)$ (the Brieskorn sphere) to itself. Then $X$ admits no Riemannian metric with positive scalar curvature because $\rho(Y)=1$ and $\operatorname{h}(Y)=0$.
\end{ex}

We summarize the idea in the proof of Theorem \ref{new obstruction} as follows: Let $W$ be the cobordism from $Y$ to itself obtained by cutting $X$ along $Y$. We consider the manifold
$$
Z_{+}=((-\infty,0]\times Y)\cup_{Y}W\cup_{Y}W\cup_{Y}...
$$
This non-compact manifold has two ends: one is cylindrical and the other one is periodic. (The word ``periodic'' indicates the fact that we are gluing togegher infinitely many copies of the same manifold $W$. See \cite{Taubes} for the precise definition.) For a Riemannian metric $g_{X}$ on $X$, we can construct, using a cut-off function, a metric on $Z_{+}$ that equals the a lift of $g_{X}$ over the periodic-end and restricts to the product metric on the cylindrical end. Now consider the (suitably perturbed) Seiberg-Witten equations on $Z_{+}$. More specifically, let $[\mathfrak{b}]$ be a critical point of the Chern-Simons-Dirac functional with certain absolute grading.  We consider the moduli space $\mathcal{M}([\mathfrak{b}],Z_{+})$ of gauge equivalent classes of solutions that approaches $[\mathfrak{b}]$ on the cylindrical end and has exponential decay on the periodic end. By adding end points to the moduli space $\mathcal{M}([\mathfrak{b}],Z_{+})$, which correspond to ``broken solutions'' on $Z_{+}$, we get the moduli space $\mathcal{M}^{+}([\mathfrak{b}],Z_{+})$, which is a $1$-manifold with boundary. Now we use the assumption that $g_{X}$ has positive scalar curvature. Under this assumption, we can prove that $\mathcal{M}^{+}([\mathfrak{b}],Z_{+})$ is compact. Therefore, the number of points in $\partial\mathcal{M}^{+}([\mathfrak{b}],Z_{+})$, counted with sign, should be $0$. This actually implies that a certain reducible critical point $[\mathfrak{a}_{0}]$ can not be ``killed by the boundary map'' and hence survives in the monopole Floer homology. By this argument, we show that $-2\operatorname{h}(Y,\mathfrak{s})\leq 2\lambda_{\textnormal{SW}}(X)$. By the same argument on $-X$, we can also prove $-2\operatorname{h}(Y,\mathfrak{s})\geq 2\lambda_{\textnormal{SW}}(X)$, which completes the proof of Theorem \ref{new obstruction}.

As can be seen from the above discussion, the study of Seiberg-Witten equations on end-periodic manifolds plays a central role in our argument. We note that the first application of gauge theory on end-periodic manifolds was given by Taubes \cite{Taubes} in the context of Donaldson theory, where he proved that the Euclidean space $\mathds{R}^{4}$ admits uncountable many exotic smooth structures. However, the Seiberg-Witten theory on end-periodic manifold is still not well developed. One major difficulty in this direction is finding a reasonable substitution for the assumption $\pi_{1}(W)=1$ (which was used in \cite{Taubes}) and prove the compactness theorem under this new assumption. In the current paper, we use the positive scalar curvature assumption, which tells something interesting but still not general enough. One possible substitute is the end-periodic symplectic structure assumption. This motivates the second purpose of the paper: we try to develop a framework that can be useful in further study of the Seiberg-Witten theory on general end-periodic manifolds. Actually, all the results (except Lemma \ref{orientation reversal 2}) in Section 2, Section 3 and the appendix are stated and proved without the positive scalar curvature assumption.

We note that many of the results and proofs in the current paper follow the same line as Kronheimer-Mrowka's book \cite{KM}. The idea is that: by working with suitably weighted Sobolev spaces, one can treat the non-compact manifold $$X_{+}=W\cup_{Y} W\cup_{Y}...$$ as a compact manifold whose signature equals the correction term $-w(X,0,g_{X})$ (see Subsection 2.4).

The precise statements of all the results used in the current paper will be given. However, to keep the length of the paper somehow under control, we will omit the proofs that are word by word translations from the corresponding parts of \cite{KM}. In order to help the reader to follow the argument, we will always give the precise reference of the omitted details. From now on, we will refer to \cite{KM} as \textbf{the book}.

The paper is organized as follows: In Section 2, we  briefly recall the definition of the monopole Floer homology, the Fr\o yshov invariant $\operatorname{h}(Y,\mathfrak{s})$ and the $4$-dimensional Casson invariant $\lambda_{\textnormal{SW}}(X)$. We will also review and prove some results about linear analysis on end-periodic manifolds. In Section 3, we start setting up the gauge theory on end-periodic manifolds and define the moduli spaces. In Section 4, we prove the compactness result under the positive scalar curvature assumption. In Section 5, we will put all the pieces together and finish the proof of Theorem \ref{new obstruction}. In the appendix, we prove (using Fourier-Laplace transformation) Proposition \ref{laplace equation}, which states the uniqueness and existence of the solution of the Laplace equation on end-periodic manifolds. This may be of independent interest for some readers.
\\
\\
\textbf{Acknowledgement.} The author wishes to thank
Diogo Veloso, Peter Kronheimer, Tomasz Mrowka, Ciprian Manolescu, Daniel Ruberman, Nicolai Saveliev and Richard Schoen for sharing their expertise in several inspiring discussions. The author is especially grateful to Clifford Taubes for suggesting the idea of proof of Lemma \ref{exp decay} (the key estimate in Section 4) and Terence Tao for providing an alternative proof of Lemma \ref{Solving laplace equation on covering space}. Corollary \ref{rohlin=froyshov} was also proved by Daniel Ruberman \cite{Ruberman} using Schoen-Yau's minimal surface result.

During the preparation of the current paper, the author noticed that a different version of compactness theorem for Seiberg-Witten equations over manifolds with periodic ends of positive scalar curvature was proved earlier in Diogo Veloso's thesis \cite{Diogo}. A different type of Hodge decomposition for such manifolds was also studied there.

\section{Preliminaries}
\subsection{The set up and the notations}
Let $X$ connected, oriented, smooth 4-manifold satisfying the condition $$H_{1}(X;\mathds{Z})\cong \mathds{Z},\  H_{2}(X;\mathds{Z})\cong 0.$$ In other words, $X$ is a homology $S^{1}\times S^{3}$. We further assume that $H_{3}(X;\mathds{Z})$ is generated by an embedded rational homology 3-sphere $Y$. (This is not always the case.)  We fix a homology orientation of $X$ by fixing a generator $[1]\in H_{1}(X;\mathds{Z})$. This induces an orientation on $Y$ by requiring that $[1]\cup [Y]=[X]$. Let $W$ be the cobordism from $Y$ to itself obtained from cutting $X$ open along $Y$. The infinite cyclic covering space of $X$ has a decomposition
$$
\tilde{X}=...\cup_{Y} W_{-1}\cup_{Y} W_{0}\cup_{Y} W_{1}\cup ... \ \text{with all }W_{n}\cong W.
$$
 We choose a lift of $Y$ to $\tilde{X}$ and still call it $Y$. We let
$$X_{+}=W_{0}\cup_{Y} W_{1}\cup_{Y} W_{2}\cup ... $$
be one of the two components of $\tilde{X}\setminus Y$.
\begin{nota}
In the current paper, we will use $\cup$ to denote the disjoint union and use $\cup_{Y}$ to denote the result of gluing two manifolds along their common boundary $Y$.
\end{nota}
There are two spin structures on $X$. We pick one of them and denote it by $\hat{\mathfrak{s}}$. It induces spin structures on the various manifolds we constructed so far. In particular, we have an induced spin structure on $Y$ and we denote it by $\mathfrak{s}$. It is not hard to see that $\mathfrak{s}$ does not depend on the choice of $\hat{\mathfrak{s}}$. These spin structures will be fixed through out the paper and we will suppress them from most of our notations. We denote by $S^{+}$ and $S^{-}$ the positive and negative spinor bundles over various 4-manifold. The spin connection over $4$-manifolds are all denoted by $A_{0}$. For the 3-manifold $Y$, we denote the spinor bundle by $S$ and the spin connection by $B_{0}$. In both dimensions, we write $\rho$ for the Clifford multiplication.

Other than $\tilde{X}$ and $X_{+}$, we also consider the following two (non-compact) spin $4$-manifolds
$$
M_{+}:=M\cup_{Y} X_{+}\ \text{and }Z_{+}:=Z\cup_{Y} X_{+},
$$
where $Z=(-\infty,0]\times Y$ and $M$ is a compact spin $4$-manifold bounded by $(Y,\mathfrak{s})$. By doing surgeries along loops in $M$, we can assume that 
\begin{equation}\label{equa: b1mequalszero}
H^{1}(M;\mathbb{R})=0
\end{equation}
We denote by $\bar{M}$ the orientation reversal of $M$.


Now we specify Riemannian metrics on these manifolds: Let $g_{X}$ be a metric on $X$. We consider a harmonic map
\begin{equation}\label{harmonic function}f:X\rightarrow S^{1}\cong \mathds{R}/\mathds{Z}\end{equation}
satisfying
$$
f^{*}(d\theta)=[1]\in H^{1}(X;\mathds{Z}).
$$
It was proved in \cite{ruberman2007dirac} that for a generic choice of $g_{X}$, the Dirac operator
$$\slashed{D}^{+}_{A}:L^{2}_{1}(X;S^{+})\rightarrow  L^{2}(X;S^{-}),$$
associated to the connection $A=A_{0}+i a\cdot f^{*}(d\theta)$ for any $a\in \mathds{R}$, has trivial kernel. We call such metric ``admissible metric''.
\begin{assum}\label{admissible metric}
Throughout this paper, we fix a choice of admissible metric $g_{X}$.
\end{assum}

\begin{rmk}
By the Weitzenb\"ock formula, any metric with positive scalar curvature is admissible. However, we will not impose this positive scalar curvature condition until Section 4.
\end{rmk}

Let $g_{\tilde{X}}$ be the lift of $g_{X}$ on $\tilde{X}$ and $g_{Y}$ be an arbitrary metric on $Y$. Using a cut-off function, we can construct a metric $g_{X_{+}}$ on $X_{+}$ which is isomorphic to the product metric $[0,3]\times g_{Y}$ near the boundary (with $\{0\}\times Y$ identified with $\partial X_{+}$)
 and whose restriction on $X_{+}\setminus W_{0}$ equals $g_{\tilde{X}}$. Let $g_{M}$ be a metric on $M$ isomorphic to the product metric near the boundary. By gluing $g_{M}$ and $g_{X_{+}}$ together, we get a metric $g_{M_{+}}$ on $M_{+}$. Similarly, we obtain the metric
$g_{Z^{+}}$ on $Z^{+}$ by gluing the metric $g_{X_{+}}$ together with the product metric on $Z$.

\subsection{The monopole Floer homology and the Fr\o yshov invariant }
In this subsection, we briefly review the definition of the monopole Floer homology and the Fr\o yshov invariant. For details, we refer to the book and \cite{Froyshov}.

Let $k\geq 3$ be an integer fixed throughout the paper. To begin with, we define
$$
\mathcal{A}_{k-1/2}(Y)=\{B_{0}+a| a\in L^{2}_{k-1/2}(Y;i\mathds{R})\}
$$
as the space of spin$^{\text{c}}$ connections over $Y$ of class $L^{2}_{k-1/2}$. Consider the configuration space:
$$
\mathcal{C}_{k-1/2}(Y)=\mathcal{A}_{k-1/2}(Y)\times L^{2}_{k-1/2}(Y;S).
$$
The pair $(B,\Psi)\in \mathcal{C}_{k-1/2}(Y)$ is called reducible if $\Psi=0$. Denote by $\mathcal{C}^{\text{red}}_{k-1/2}(Y)$ the space of reducible pairs. We will also consider the blown-up configuration space:
\begin{equation}
\begin{split}
\mathcal{C}^{\sigma}_{k-1/2}(Y)=\{&(B,s,\Psi)|\  B\in \mathcal{A}_{k-1/2}(Y),\\
&s\in \mathds{R}_{\geq 0} \text{ and } \Psi\in L^{2}_{k-1/2}(Y;S) \text{ satisfies }\|\Psi\|_{L^{2}}=1\}.\end{split}
\end{equation}
The gauge group
$$\mathcal{G}_{k+1/2}(Y)=\{u:Y\rightarrow S^{1}|\  \|u\|_{L^{2}_{k+1/2}}<\infty\}$$ acts on both $\mathcal{C}_{k-1/2}(Y)$ and $\mathcal{C}^{\sigma}_{k-1/2}(Y)$. Denote the quotient spaces by $\mathcal{B}_{k-1/2}(Y)$ and $\mathcal{B}^{\sigma}_{k-1/2}(Y)$ respectively. It was proved in the book that $\mathcal{C}_{k-1/2}(Y)$ and $\mathcal{B}_{k-1/2}(Y)$ are Hilbert manifolds without boundary, while
$\mathcal{C}_{k-1/2}(Y)$ and $\mathcal{B}_{k-1/2}(Y)$ are Hilbert manifolds with boundary.

We define the Chern-Simons-Dirac functional $\mathcal{L}$ (with $B_{0}$ as the preferred reference connection) on $C_{k-1/2}(Y)$ as
\begin{equation}\label{CSD}
\mathcal{L}(B,\Psi)=-\frac{1}{8}\int_{Y}(B^{t}-B^{t}_{0})\wedge (F_{B^{t}}+F_{B^{t}_{0}})+\frac{1}{2}\int_{Y}\langle \slashed{D}_{B}\Psi,\Psi\rangle\,d\text{vol},
\end{equation}
where $B^{t}$ and $B_{0}^{t}$ denote the induced connections on the determine bundle $\text{det}(S)$ and $F_{B^{t}}$, $F_{B_{0}^{t}}$ denote their curvatures. We denote by $\operatorname{grad}\mathcal{L}$ the formal gradient of $\mathcal{L}$. This is a section of the $L^{2}_{k-3/2}$-completed tangent bundle of $\mathcal{C}_{k-1/2}(Y)$. In order to get the transversality condition, we need to add a perturbation $\mathfrak{q}$ on $\operatorname{grad}\mathcal{L}$. The sum $\operatorname{grad}\mathcal{L}+\mathfrak{q}$ is gauge invariant and gives rise to a ``vector field'' $$v_{\mathfrak{q}}^{\sigma}:\mathcal{B}_{k-1/2}^{\sigma}(Y)\rightarrow \mathcal{T}_{k-3/2}(Y),$$ where $\mathcal{T}_{k-3/2}(Y)$ denotes the $L^{2}_{k-3/2}$ completion of the tangent bundle of $\mathcal{B}_{k-1/2}^{\sigma}(Y)$. (We put the quotation marks here because $v_{\mathfrak{q}}^{\sigma}$ is not a section of the actual tangent bundle). We call the perturbation $\mathfrak{q}$ admissible if all critical points of $v^{\sigma}_{\mathfrak{q}}$ are nondegenerate and the moduli spaces of flow lines connecting them are regular. (See Page 411 of the book for an exact definition.) Under this admissibility condition, the set $\mathfrak{C}$ of critical points of $v^{\sigma}_{\mathfrak{q}}$ is discrete and can be decomposed into the disjoint union of three subsets:
\begin{itemize}
\item $\mathfrak{C}^{o}$: the set of  irreducible critical points;
\item $\mathfrak{C}^{s}$: the set of reducible, boundary stable critical points (i.e., reducible critical points where $v^{\sigma}_{\mathfrak{q}}$ points outside the boundary);
\item $\mathfrak{C}^{u}$: the set of reducible, boundary unstable critical points (i.e., reducible critical points where $v^{\sigma}_{\mathfrak{q}}$ points inside the boundary).
\end{itemize}
The monopole Floer homologies $\widebar{HM}(Y,\mathfrak{s};\mathds{Q})$,
$\widecheck{HM}(Y,\mathfrak{s};\mathds{Q})$ and $\widehat{HM}(Y,\mathfrak{s};\mathds{Q})$ are defined as the homology of the chain complexes freely generated by $\mathfrak{C}^{o}$, $\mathfrak{C}^{o}\cup \mathfrak{C}^{s}$ and $\mathfrak{C}^{s}\cup \mathfrak{C}^{u}$ respectively.

Our main concern will be $\widebar{HM}(Y,\mathfrak{s};\mathds{Q})$ and $\widecheck{HM}(Y,\mathfrak{s};\mathds{Q})$. To give the precise definitions, we first recall that a two-element set $\Lambda([\mathfrak{b}])$ (called the orientation set) can be associated to each $[\mathfrak{b}]\in \mathfrak{C}$ (see Section 20.3 of the book). After making a choice of preferred element $\chi([\mathfrak{b}])\in\Lambda([\mathfrak{b}])$ for each $[\mathfrak{b}]$, we can canonically orient the moduli spaces of trajectories connecting them. Now let $C^{o}$ (resp. $C^{u}$ and $C^{s}$) be a vector space over $\mathds{Q}$ with basis $\{e_{[\mathfrak{b}]}\}$ indexed by elements $[\mathfrak{b]}$ in $\mathfrak{C}^{o}$ (resp. $\mathfrak{C}^{s}$ and  $\mathfrak{C}^{u}$). We define the linear maps
$$
\partial^{o}_{o}:C^{o}\rightarrow C^{o},\ \ \ \ \partial^{o}_{s}:C^{o}\rightarrow C^{s},
$$
$$
\partial^{u}_{o}:C^{u}\rightarrow C^{o},\ \ \ \ \partial^{u}_{s}:C^{u}\rightarrow C^{s}.
$$
by the formulae
$$
\partial^{o}_{o}e_{[\mathfrak{b}]}=\mathop{\sum}\limits_{[\mathfrak{b}']\in \mathfrak{C}^{o}} \#\breve{\mathcal{M}}([\mathfrak{b}],[\mathfrak{b}'])\cdot e_{[\mathfrak{b}']}\ \ \ \ ([\mathfrak{b}]\in \mathfrak{C}^{o})
$$
and so on, where the integer $\#\breve{\mathcal{M}}([\mathfrak{b}],[\mathfrak{b}'])$ counts (with sign) the number of points in $\breve{\mathcal{M}}([\mathfrak{b}],[\mathfrak{b}'])$ (the moduli space of Seiberg-Witten trajectories going from $[\mathfrak{b}]$ to $[\mathfrak{c}]$) that has dimension $0$.

By considering the number $\#\breve{\mathcal{M}}^{\text{red}}([\mathfrak{b}],[\mathfrak{b}'])$ instead (i.e., only counting reducible trajectories), we can similarly define the linear maps
$$
\bar{\partial}^{s}_{s}:C^{s}\rightarrow C^{s},\ \ \ \
\bar{\partial}^{s}_{u}:C^{s}\rightarrow C^{u},$$
$$
\bar{\partial}^{u}_{s}:C^{u}\rightarrow C^{s},\ \ \ \ \bar{\partial}^{u}_{u}:C^{u}\rightarrow C^{u}.
$$
(We note that $\bar{\partial}^{u}_{s}$ is different with $\partial^{u}_{s}$.)

The following definition was given as Definition 22.1.7 of the book.
\begin{defi}\label{monopole Floer}
The monopole Floer homology groups $\widebar{HM}_{*}(Y,\mathfrak{s};\mathds{Q})$ and $\widecheck{HM}_{*}(Y,\mathfrak{s};\mathds{Q})$ are defined as the homology groups of the chain complexes $\bar{C}=C^{s}\oplus C^{u}$ and  $\check{C}=C^{o}\oplus C^{s}$ with the differentials
\begin{equation}\label{differential for hm-bar}
\bar{\partial}=\left(\begin{array} {cc}
 \bar{\partial}^{s}_{s}  & \bar{\partial}^{u}_{s}  \\
 \bar{\partial}^{s}_{u}  & \bar{\partial}^{u}_{u}
\end{array}\right)\text{ and }\check{\partial}:=\left(\begin{array} {cc}
 \partial^{o}_{o}  & -\partial^{u}_{o}\bar{\partial}^{s}_{u} \\
 \partial^{o}_{s} & \bar{\partial}^{s}_{s}-\partial^{u}_{s}\bar{\partial}^{s}_{u}
\end{array}\right)
\end{equation}
respectively. There is a natural map $i_{*}:\widebar{HM}_{*}(Y,\mathfrak{s};\mathds{Q})\rightarrow \widecheck{HM}_{*}(Y,\mathfrak{s};\mathds{Q})$  induced by the chain map $i:\bar{C}\rightarrow \check{C}$ defined as
\begin{equation}\label{chain map}
\left(\begin{array} {cc}
0  & -\partial^{u}_{o}\\
 1 & -\partial^{u}_{s}
\end{array}\right).
\end{equation}
\end{defi}
To each $[\mathfrak{b}]\in \mathfrak{C}$, we can assign a rational number $\operatorname{gr}^{\mathds{Q}}([\mathfrak{b}])$ (called the absolute grading) as follows (see Definition 28.3.1 of the book): Let $\operatorname{gr}(M,[\mathfrak{b}])$ be the ``relative $M$-grading'' of $[\mathfrak{b}]$. This number describes the expected dimension of the Seiberg-Witten moduli space on the manifold $M^{*}=M\cup_{Y} ([0,+\infty)\times Y)$ with limit $[\mathfrak{b}]$. It was proved in the book that the quantity
\begin{equation}\label{absolute grading}
-\operatorname{gr}(M,[\mathfrak{b}])-b_{2}^{+}(M)-\frac{1}{4}\operatorname{sign}(M)-1
\end{equation}
does not depend on the choice of $M$ and we define it as $\operatorname{gr}^{\mathds{Q}}([\mathfrak{b}])$. This grading induces absolute gradings on $\widebar{HM}_{*}(Y,\mathfrak{s};\mathds{Q}),\  \widehat{HM}_{*}(Y,\mathfrak{s};\mathds{Q})$ and $\widecheck{HM}_{*}(Y,\mathfrak{s};\mathds{Q})$. The map $i_{*}$ in Definition \ref{monopole Floer} preserves this grading.
\begin{rmk}
In (\ref{absolute grading}), we use $\operatorname{gr}(M,[\mathfrak{b}])$ instead of $\operatorname{gr}([\mathfrak{a}_{0}],M\setminus B^{4},[\mathfrak{b}])$ as in the book. Here $[\mathfrak{a}_{0}]$ denotes the first boundary stable critical point in $\mathcal{B}^{\sigma}_{k-1/2}(S^{3})$. These two gradings satisfy the relation (see Lemma 27.4.2 of the book)
$$
\operatorname{gr}(M,[\mathfrak{b}])=\operatorname{gr}(B^{4},[\mathfrak{a}_{0}])+\operatorname{gr}([\mathfrak{a}_{0}],M\setminus B^{4},[\mathfrak{b}])=-1+\operatorname{gr}([\mathfrak{a}_{0}],M\setminus B^{4},[\mathfrak{b}]).
$$
This explains the extra term ``$-1$'' in our formula.
\end{rmk}

\begin{rmk}
In general, one needs to specify a connected component of $\mathcal{B}^{\sigma}_{k}(M)$ (the blown-up quotient configuration space of $M$) to define the relative $M$-grading. However, in our case the space $\mathcal{B}^{\sigma}_{k}(M)$ is connected since $b_{1}(Y)=0$.
\end{rmk}

\begin{defi}\cite{Froyshov}
The Fr\o yshov invariant is defined as
$$\operatorname{h}(Y,\mathfrak{s}):=-\frac{1}{2}\cdot\inf\{\operatorname{gr}^{\mathds{Q}}([\mathfrak{b}])|[\mathfrak{b}]\text{ represents a nonzero elements in }\operatorname{im}i_{*}\}.$$
\end{defi}
The following lemma was proved in \cite{Froyshov} (in a (possibly) different version of monopole Floer homology). The proof can be easily adapted to the version used in the book.
\begin{lem}\label{orientation reversal}
For any rational homology sphere $Y$ and any  spin$^{\text{c}}$ structure $\mathfrak{s}$ on $Y$, we have $\operatorname{h}(-Y,\mathfrak{s})=-\operatorname{h}(Y,\mathfrak{s})$.
\end{lem}

\begin{defi}
An admissible perturbation $\mathfrak{q}$ is called a ``nice perturbation'' if $\mathfrak{q}=0$ when restricted to ${\mathcal{C}^{\operatorname{red}}_{k-1/2}(Y)}$.
\end{defi}

\begin{rmk}\label{component of perturbation}
Since the tangent bundle of $\mathcal{C}_{k-1/2}(Y)$ is trivial with fiber $$L^{2}_{k-1/2}(Y;i\mathds{R})\oplus L^{2}_{k-1/2}(Y;S),$$ we can write the perturbation $\mathfrak{q}$ as $(\mathfrak{q}^{0},\mathfrak{q}^{1})$, where $\mathfrak{q}^{0}$ denotes the connection component and $\mathfrak{q}^{1}$ denotes the spinor component. Note that by the gauge invariance, the restriction of $\mathfrak{q}^{1}$ to $\mathcal{C}^{\operatorname{red}}_{k-1/2}(Y)$ is always $0$. Therefore, an admissible perturbation $\mathfrak{q}$ is nice if and only if $\mathfrak{q}^{0}=0$ when restricted to ${\mathcal{C}^{\operatorname{red}}_{k-1/2}(Y)}$.
\end{rmk}

Under the assumption that $\mathfrak{q}$ is nice, there is only one reducible critical point downstairs (up to gauge transformation), which is just $(B_{0},0)$. As for the critical points upstairs, the sets $\mathfrak{C}^{u}$ and $\mathfrak{C}^{s}$ can be described explicitly as follows:
Consider the self-adjoint operator
\begin{equation}\label{perturbed dirac}
\slashed{D}_{\mathfrak{q},B_{0}}:L^{2}_{k-1/2}(Y;S)\rightarrow L^{2}_{k-3/2}(Y;S)
\end{equation}
$$
\Psi\mapsto \slashed{D}_{B_{0}}\Psi+\mathcal{D}_{(B_{0},0)}\mathfrak{q}^{1}(0,\Psi) .
$$
Since $\mathfrak{q}$ is admissible, $0$ is not an eigenvalue of $\slashed{D}_{\mathfrak{q},B_{0}}$ and all eigenvalues have multiplicity $1$ (see Proposition 12.2.5 of the book). We arrange the eigenvalues $\lambda_{*}$ so that
$$
...\lambda_{-2}<\lambda_{-1}<0<\lambda_{0}<\lambda_{1}<...
$$
For each $i$, we pick an eigenvector $\psi_{i}$  with eigenvalue $\lambda_{i}$ and $\|\psi_{i}\|_{L^{2}}=1$. We let $[\mathfrak{a}_{i}]=[(B_{0},0,\psi_{i})]$.  By Proposition 10.3.1 of the book, we have
$$
\mathfrak{C}^{s}=\{[\mathfrak{a}_{i}]|\,i\geq 0\},\ \mathfrak{C}^{u}=\{[\mathfrak{a}_{i}]|\,i< 0\}.
$$
From now on, we always use $[\mathfrak{a}_{*}]$ to denote these reducible critical points. Note that $\operatorname{gr}^{\mathds{Q}}([\mathfrak{a}_{i}])-\operatorname{gr}^{\mathds{Q}}([\mathfrak{a}_{i-1}])$ equals $1$ when $i=0$ and equals $2$ otherwise.

\begin{defi}
Let $\mathfrak{\mathfrak{q}}$ be a nice pertrubation. The height of $\mathfrak{q}$ is defined as $$\operatorname{ht}(\mathfrak{q})=\operatorname{gr}^{\mathds{Q}}([\mathfrak{a}_{0}]).$$
In other words, the height is defined to be the absolute grading of the lowest boundary stable critical point.
\end{defi}
Consdier the operator
$$
D_{\mathfrak{q}}:L^{2}_{k}(M;S^{+})\rightarrow L^{2}_{k-1}(M;S^{-})\oplus (L^{2}_{k-1/2}(Y;S)\cap H_{1}^{-})
$$
$$
\Phi\mapsto (\slashed{D}^{+}_{\hat{\mathfrak{q}},A_{0}}\Phi,\pi^{-}(\Phi|_{Y}))
$$
where $\slashed{D}^{+}_{\hat{\mathfrak{q}},A_{0}}$ is a perturbed Dirac operator over $M$ which equals $\frac{d}{dt}+\slashed{D}_{\mathfrak{q},B_{0}}$ near the boundary; $H^{-}_{1}$ (resp. $H^{+}_{1}$) is the closure in $L^{2}(Y;S)$ of the eigenvectors of $\slashed{D}_{\mathfrak{q},B_{0}}$ with negative (resp. positive) eigenvalue; $\pi^{-}$ is the projection to $L^{2}_{k-1/2}(Y;S)\cap H_{1}^{-}$ with kernel $H^{+}_{1}$.
\begin{lem}\label{height as index}
For any nice perturbation $\mathfrak{q}$, we have
\begin{equation}\label{height as eta invariant}
\operatorname{ht}(\mathfrak{q})=-2\operatorname{ind}_{\mathds{C}}D_{\mathfrak{q}}-\tfrac{\operatorname{sign}(M)}{4}.
\end{equation}
\end{lem}
\begin{proof}
By the same argument as Page 508 of the book, we can identify $\operatorname{grad}(M,[\mathfrak{a}_{0}])$ with the index of the Fredholm operator (24.41) in the book. A further deformation identifies this index with the index of the operator  $D_{\mathfrak{q}}\oplus B$, where $B$ is the Fredholm operator
$$
L^{2}_{k}(M;iT^{*}M)\rightarrow L^{2}_{k-1}(M;i\mathds{R}\oplus i\wedge^{2}_{+}T^{*}M)\oplus L^{2}_{k-1/2}(Y;i\mathds{R})\oplus C^{-}
$$
$$
\alpha\mapsto (d^{*}\alpha,d^{+}\alpha,\langle \alpha,\vec{v}\rangle,\alpha^{-}).
$$
Here $C^{-}\subset (\operatorname{ker}d^{*}\cap L^{2}_{k-1/2}(Y;iT^{*}Y))$ denotes the negative eigenspace of the operator $*d$ and $\alpha^{-}\in C^{-}$ denotes projection of $\alpha|_{Y}$.  By Lemma 24.8.1 of the book, we have $\operatorname{ind}_{\mathds{R}}B=-b^{+}_{2}(M)-1$. Therefore, we get
$$
\operatorname{grad}(M,[\mathfrak{a}_{0}])=2\operatorname{ind}_{\mathds{C}}D_{q}-b_{2}^{+}(M)-1.
$$
By (\ref{absolute grading}), this implies the lemma.
\end{proof}
Now consider the following subset of $\mathds{Q}$
$$
\mathfrak{m}(Y,\mathfrak{s})=\{a\in \mathds{Q}|\, a=[-\frac{\operatorname{sign}(M)}{8}]\in \mathds{Q}/\mathds{Z}\}.
$$
\begin{rmk}$\mathfrak{m}(Y,\mathfrak{s})$ is actually determined by the Rohlin invariant $\rho(Y,\mathfrak{s})$ and hence independent with the choice of $M$.\end{rmk}

\begin{pro}\label{height of nice perturbation}
For any $e\in \mathfrak{m}(Y,\mathfrak{s})$, there exists a nice perturbation $\mathfrak{q}$ with $\tfrac{\operatorname{ht}(\mathfrak{q})}{2}=e$.
\end{pro}
\begin{proof}
Let $\{\psi_{n}|\,n\in \mathds{Z}_{\geq 0}\}$ be a complete, orthonormal set of eigenvectors of  $\slashed{D}_{B_{0}}$. Let the eigenvalue of $\psi_{n}$ be $\lambda_{n}'$. For each $n$, we consider the the function
$$
f_{n}:\mathcal{C}_{k-1/2}(Y)\rightarrow \mathds{R}
$$
$$
(B_{0}+a,\Psi)\mapsto |\langle e^{i\xi}\Psi,\psi_{n}\rangle_{L^{2}}|^{2}
$$
where $\xi:Y\rightarrow\mathds{R}$ is the unique solution of
$$
i\Delta \xi=d^{*}da,\ \int_{Y}\xi=0.
$$
One can prove that $f_{n}$ is invariant under the action of $\mathcal{G}_{k+1/2}(Y)$. We denote by $\mathfrak{q}_{n}$ the formal gradient of $f_{n}$. A simple calculation shows that
$$
\mathcal{D}_{(B_{0},0)}q_{n}^{1}(0,\Psi)=2\langle \Psi,\psi_{n}\rangle_{L^{2}}\cdot \psi_{n}.
$$
We let $\mathfrak{q}'=\mathop{\sum}\limits_{n=0}^{+\infty}c_{n}\mathfrak{q}_{n}$, where $\{c_{n}\}$ is a sequence of real numbers. We require $|c_{n}|$ decreasing to $0$ fast enough so that $\mathfrak{q}'$ is a tame-perturbation (see Definition 10.5.1 of the book). Now consider the perturbed Dirac operator $\slashed{D}_{\mathfrak{q}',B_{0}}$ (see (\ref{perturbed dirac})). Its eigenvalues are of the form $\lambda'_{n}+2c_{n}$ and the corresponding eigenvector is just $\psi_{n}$. By choosing a generic sequence $\{c_{n}\}$, we can assume
$$
\lambda'_{n}+2c_{n}\neq \lambda'_{m}+c_{m},\ \forall n\neq m\text{ and } \lambda'_{n}+2c_{n}\neq 0,\ \forall n\in\mathds{Z}_{\geq 0}.
$$
Note that the number $
-\operatorname{ind}_{\mathds{C}}D_{\mathfrak{q}'}-\tfrac{\operatorname{sign}(M)}{8}
$
always belongs to $\mathfrak{m}(Y,\mathfrak{s})$. Moreover, as we varies $\{c_{n}\}$, this number changes by the spectral flow of $\slashed{D}_{\mathfrak{q}',B_{0}}$. Therefore, by choosing suitable $\{c_{n}\}$, we may assume that
$$
e=-\operatorname{ind}_{\mathds{C}}D_{\mathfrak{q}'}-\frac{\operatorname{sign}(M)}{8}.
$$
Under this perturbation $\mathfrak{q}'$, the reducible critical points are just $[(B_{0},0,\psi_{n})]$ with $n\geq 0$. All of them are non-degenerate by \cite[Proposition 12.2.5]{KM}. Therefore, by the compactness result of the critical points, we can find $\epsilon>0$ such that the gauge invariant open subset
$$
U(\epsilon)=\{(B,\Phi)| \|\Phi\|_{L^{2}}<\epsilon\} \subset \mathcal{C}_{k-1/2}(Y)
$$
contains no irreducible critical point. Now consider the Banach space
$$\mathcal{P}(U(\epsilon)):=\{\mathfrak{q}''\in \mathcal{P}|\  \mathfrak{q}''|_{U(\epsilon)}= 0\},$$
where $\mathcal{P}$ is the large Banach space of tame perturbations constructed in Theorem 11.6.1 of the book. By repeating the proof of Theorem 15.1.1 of the book, we can find a perturbation $\mathfrak{q}''\in \mathcal{P}(U(\epsilon))$ such that the perturbation $\mathfrak{q}=\mathfrak{q}''+\mathfrak{q}'$ is admissible. Since both $\mathfrak{q}''$ and $\mathfrak{q}'$ vanishes on $\mathcal{C}_{k-1/2}^{\text{red}}(Y)$, the perturbation $\mathfrak{q}$ is nice. Moreover, since $\mathfrak{q}''$ vanishes on $U(\epsilon)$, we have $D_{\mathfrak{q}}=D_{\mathfrak{q}'}$. By Lemma \ref{height as eta invariant}, we have
$$
\frac{\operatorname{ht}(\mathfrak{q})}{2}=-\operatorname{ind}_{\mathds{C}}D_{\mathfrak{q}}-\tfrac{\operatorname{sign}(M)}{8}=-\operatorname{ind}_{\mathds{C}}D_{\mathfrak{q}'}-\tfrac{\operatorname{sign}(M)}{8}=e.
$$This finishes the proof.\end{proof}

\begin{lem}\label{alternative defi of Froyshov}
Suppose $\mathfrak{q}$ is a nice perturbation with $\operatorname{ht}(\mathfrak{q})<-2\operatorname{h}(Y,\mathfrak{s})$. Then we have
\begin{equation}\label{Froyshov 1}
\begin{split}
-2\operatorname{h}(Y,\mathfrak{s})=&\inf\{\operatorname{gr}^{\mathds{Q}}([\mathfrak{a}_{j}])|\,j\geq 0;\   \nexists\ n,m_{1},...,m_{l}\in \mathds{Z}_{\neq 0} \text{ and }[\mathfrak{b}_{1}],...,[\mathfrak{b}_{l}]\in \mathfrak{C}^{o} \text{ s.t. }\\
&\partial^{o}_{o}(m_{1}[\mathfrak{b}_{1}]+...+m_{l}[\mathfrak{b}_{l}])=0 \text{ and } \partial^{o}_{s}(m_{1}[\mathfrak{b}_{1}]+...+m_{l}[\mathfrak{b}_{l}])=n[\mathfrak{a}_{j}]\}.
\end{split}
\end{equation}
\end{lem}
\begin{proof}
For the grading reason, all the maps $\bar{\partial}^{*}_{*}$ vanish. As a result, the set
$$\{[e_{[\mathfrak{a}_{j}]}]|\,j\in \mathds{Z}\}$$ is a basis of $\widebar{HM}_{*}(Y,\mathfrak{s};\mathds{Q})$. For $j\geq 0$, the map $i_{*}$ sends $$[e_{[\mathfrak{a}_{j}]}]\in \widebar{HM}_{*}(Y,\mathfrak{s};\mathds{Q})$$ to $$[e_{[\mathfrak{a}_{j}]}]\in \widecheck{HM}_{*}(Y,\mathfrak{s};\mathds{Q}).$$ Since we have $\operatorname{ht}(\mathfrak{q})<-2\operatorname{h}(Y;\mathfrak{s})$, the set
$$
S=\{j|j\geq 0,\ [e_{[\mathfrak{a}_{j}]}]\neq 0\in \widecheck{HM}_{*}(Y,\mathfrak{s};\mathds{Q})\}
$$
does not equals $\mathds{Z}_{\geq 0}$ and we have
\begin{equation}\label{Froyshov 2}-2\operatorname{h}(Y,\mathfrak{s})=\inf\{\operatorname{gr}^{\mathds{Q}}([\mathfrak{a}_{j}])|j\in S\}.\end{equation}
Since we have
$$
\check{\partial}=\left(\begin{array} {cc}
 \partial^{o}_{o}  & 0\\
 \partial^{o}_{s} & 0
\end{array}\right).
$$ in the current case, (\ref{Froyshov 1}) and (\ref{Froyshov 2}) coincide with each other. This finishes the proof of the lemma.
\end{proof}

\subsection{Linear analysis on end-periodic manifolds}

In this subsection, we will set up the appropriate Sobolev spaces on end-periodic manifolds and review the related Fredholm theory. Our construction is inspired from \cite{Taubes} and \cite{MRS}.

Let $E$ be an end-periodic bundle (over $\tilde{X},X_{+},M_{+} \text{ or }Z_{+}$) equipped with an end-periodic metric  $|\cdot|$ and an end-periodic connection $\nabla$ (see \cite{Taubes} for definition). For any $j,p\in \mathds{Z}_{\geq 0}$,  we can define the unweighted Sobolev norm of a smooth section $s$ in the usual way:
\begin{equation}\label{unweighted sobolev space}
\|s\|_ {L^{p}_{j}}:=(\mathop{\Sigma}\limits_{i=0}^{j}\int|\nabla^{(i)}s|^{p} d\operatorname{vol})^{\frac{1}{p}}.
\end{equation}
(We can also define the $L^{p}_{j}$ norm for negative $j$ using integration.)
\begin{rmk}
Other then a trivial real or complex line bundle, which we denote by $\mathds{R},\mathds{C}$ respectively, two other types of end-periodic bundle will be considered: the spinor bundle $S^{\pm}$ (associated to spin structures) and the bundle of differential forms. Both of them have a canonical metric. As for the connection, we use the spin connection for the former and the Levi-Civita connection for the latter.
\end{rmk}

In general, the differential operators that we will consider do not have Fredholm properties under the norms defined in \ref{unweighted sobolev space}. Therefore, we need to use the weighted Sobolev norms instead. To define them, recall that we have a harmonic map $f:X\rightarrow S^{1}$ corresponding to a generator of $H^{1}(X;\mathds{Z})$. We lift $f$ to a function $\tilde{f}:\tilde{X}\rightarrow \mathds{R}$ satisfying
$$
f^{-1}([-1,1])\subset \mathop{\cup}\limits_{n=-N}^{N}W_{n}
\text{ for some }N\gg 0.$$
Now consider the following smooth cut-off functions:
\begin{itemize}
\item $\tau_{0}:\tilde{X}\rightarrow [0,+\infty)$: a function that equals $|f|$ on $\tilde{X}\setminus \mathop{\cup}\limits_{n=-N}^{N} W_{n}$;
\item $\tau_{1}:X_{+}\rightarrow [0,+\infty)$: the restriction of $\tau_{0}$;
\item $\tau_{2}:M_{+}\rightarrow [0,+\infty)$: an extension of $\tau_{1}$;
\item $\tau_{3}:Z_{+}\rightarrow [0,+\infty)$: an extension of $\tau_{1}$ with the property that $$\tau_{2}(t,y)=|t|,\ \forall (t,y)\in (-\infty,-1]\times Y.$$
\end{itemize}
\begin{defi}
For $\delta\in \mathds{R},j\in \mathds{Z},p\in \mathds{Z}_{\geq 0}$, we define the weighted Sobolev norm of a smooth section $s$ of $E$ in different ways depending on the underlying manifold:
\begin{itemize}
\item Over $X_{+}$, we set $\|s\|_ {L^{p}_{j,\delta}}=\|e^{\delta\cdot \tau_{1}}\cdot s\|_ {L^{p}_{j}}$;
\item Over $M_{+}$, we set $\|s\|_ {L^{p}_{j,\delta}}=\|e^{\delta\cdot \tau_{2}}\cdot s\|_ {L^{p}_{j}}$;
\item Over $\tilde{X}$, we set $\|s\|_ {L^{p}_{j;-\delta,\delta}}=\|e^{\delta\cdot \tau_{0}}\cdot s\|_ {L^{p}_{j}}$;
\item Over $Z_{+}$, we set $\|s\|_ {L^{p}_{j;-\delta,\delta}}=\|e^{\delta\cdot \tau_{3}}\cdot s\|_ {L^{p}_{j}}$.
\end{itemize}
(Note that we use two weight indices for manifolds $\tilde{X}$ and $Z_{+}$ because they both have two ends.) We denote the corresponding Sobolev space respectively by $$L^{2}_{j,\delta}(X_{+};E),\ L^{2}_{j,\delta}(M_{+};E),\ L^{2}_{j;-\delta,\delta}(\tilde{X};E) \text{ and } L^{2}_{j;-\delta,\ \delta}(Z_{+};E).$$
We remove $j$ from our notations when it equals $0$.
We sometimes also suppress the bundle $E$ when it is clear from the context.
\end{defi}

The following lemma is a straightforward corollary of \cite[Lemma 5.2]{Taubes}. It asserts that one can control the weighted Sobolev norm of a function using the weighted Sobolev norm of its derivative. (Although \cite{Taubes} only stated the result for smooth functions, we can prove the general case easily using standard arguments, i.e., approximating a Sobolev function by smooth functions.)

\begin{lem} \label{Taubes's lemma}
For any $\delta>0,j\geq 0$, we can find a positive constant $C$ with the following significance:
\begin{enumerate}
\item For any $u\in L^{2}_{1,\operatorname{loc}}(X_{+};\mathds{R})$ with $\|du\|_{L^{2}_{j,\delta}}<\infty$, there exists a unique number $\bar{u}\in \mathds{R}$ such that $\|u-\bar{u}\|_{L^{2}_{j+1,\delta}}<\infty$. Moreover, in this case we have $$\|u-\bar{u}\|_{L^{2}_{j+1,\delta}}\leq C\|d\bar{u}\|_{L^{2}_{j,\delta}}.$$
\item Fix a smooth function $$\tau_{4}:Z_{+}\rightarrow [0,1] \text{ with } \tau_{4}|_{Z}=0, \tau_{4}|_{W_{i}}=1\  \forall i\geq 1.$$Then for any $u\in L^{2}_{1,\operatorname{loc}}(Z_{+};\mathds{R})$ with $\|du\|_{L^{2}_{j;-\delta,\delta}}<\infty$, there exists unique numbers $\bar{u},\bar{\bar{u}}\in \mathds{R}$ such that $\|u-\bar{u}-\bar{\bar{u}}\cdot\tau_{4}\|_{L^{2}_{j+1;-\delta,\delta}}<\infty$. In this case we have $$\|u-\bar{u}-\bar{\bar{u}}\cdot\tau_{4}\|_{L^{2}_{j+1;-\delta,\delta}}\leq C\|du\|_{L^{2}_{j;-\delta,\delta}}.$$
\end{enumerate}
\end{lem}

Next, we summarize the Sobolev embedding and multiplication theorems. We focus on the manifold $X_{+}$ (although similar results holds other manifolds) because that will be our main concern. The proofs are straightforwardly adapted from the unweighted case (Theorem 13.2.1 and Theorem 13.2.2 of the book) and the cylindrical end case
(\cite[Proposition 2.9, Proposition 2.10]{francescolin}) so we omit them.
\begin{pro}\label{Sobolev embedding}
Let $E$ be an end-periodic bundle over $X_{+}$. There is a continuous inclusion
$$
L^{p}_{j,\delta}(X_{+};E)\rightarrow L^{q}_{l,\delta'}(X_{+};E)
$$
for $j\geq l,\ \delta\geq \delta'\geq 0,\  p \leq q$ and $(j-4/p)\geq (l-4/q)$. This embedding is compact when $j>l,\ \delta>\delta'$ and  $(j-4/p)> (l-4/q)$.
\end{pro}
\begin{pro}\label{Sobolev multiplication}
Let $E,F$ be two end-periodic bundles over $X_{+}$.
Suppose $\delta+\delta'\geq \delta'',\ j,l\geq m$ and $1/p+1/q\geq 1/r$, with $\delta,\delta',\delta''\geq 0$ and $p,q,r>1$. Then the multiplication
$$
L^{p}_{j,\delta}(X_{+};E)\times L^{q}_{l,\delta'}(X_{+};F)\rightarrow L^{r}_{m,\delta''}(X_{+};E\otimes F)
$$
is continuous in any of the following three cases:
\begin{enumerate}
\item \begin{enumerate}
\item $(j-4/p)+(l-4/q)\geq m-4/r,$ and
\item $j-4/p<0,$ and
\item $l-4/q<0$;
\end{enumerate}
\hspace{-4mm} or
\item \begin{enumerate}
\item $\min \{j-4/p,l-4/q\}\geq m-4/r,$ and
\item either $j-4/p>0$ or $l-n/q>0$;
\end{enumerate}
\hspace{-4mm} or
\item \begin{enumerate}
\item $\min \{j-4/p,l-4/q\}> m-4/r,$ and
\item either $j-4/p=0$ or $l-4/q=0$.
\end{enumerate}

\end{enumerate}
When the map is continuous, it is a compact operator as a function of second variable for fixed first variable provided $l>m$ and $l-4/q>m-4/r$.
\end{pro}


The following corollary will be very useful because the differential operators we are going to consider can often be composed into the sum of a first-order, linear operator with a zeroth-order, quadratic operator.
\begin{cor}
For any $j>2,\delta>0$, the multiplication map $$L^{2}_{j,\delta}(X_{+};E)\times L^{2}_{j,\delta}(X_{+};F)\rightarrow L^{2}_{j-1,\delta}(X_{+};E\otimes F)$$ is compact.
\end{cor}
\begin{proof}
By Proposition \ref{Sobolev multiplication}, this map factors through the natural inclusion $$L^{2}_{j,2\delta}(X_{+};E\otimes F)\rightarrow L^{2}_{j-1,\delta}(X_{+};E\otimes F),$$
which is compact by Proposition \ref{Sobolev embedding}.
\end{proof}
Now we start discussing the related Fredholm theory.
\begin{pro}\label{laplace equation}
There exists a small $\delta_{0}>0$ such that for any $j\in \mathds{Z}_{\geq 0}$ and $\delta\in (0,\delta_{0})$, we have the following results:
\begin{enumerate}[(i)]
\item The operator
$$
\Delta(\tilde{X};-\delta,\delta):L^{2}_{j+2;-\delta,\delta}(\tilde{X};\mathds{R})\rightarrow L^{2}_{j;-\delta,\delta}(\tilde{X};\mathds{R})$$
$$u \mapsto \Delta u
$$
is a Fredholm operator with trivial kernel and two dimensional cokernel. The same result holds for the manifold $Z_{+}$.
\item The operator
$$
\Delta(M_{+};\delta):L^{2}_{j+2,\delta}(M_{+};\mathds{R})\rightarrow L^{2}_{j,\delta}(M_{+};\mathds{R})$$
$$u\mapsto \Delta u
$$
is a Fredholm operator with trivial kernel and 1-dimensional cokernel.
\item The operator
$$
\Delta(X_{+};\delta):L^{2}_{j+2,\delta}(X_{+};\mathds{R})\rightarrow L^{2}_{j,\delta}(X_{+};\mathds{R})\oplus L^{2}_{j+1/2}(Y;\mathds{R})$$ $$u\mapsto (\Delta u,\langle du,\vec{v}\rangle)$$
is Fredholm with trivial kernel and $1$-dimensional cokernel, where $\vec{v}$ denotes the inward normal vector on the boundary.
\end{enumerate}
\end{pro}
Proposition \ref{laplace equation} will be proved in the appendix.
\begin{lem}\label{half De rham complex}
There exists a constant $\delta_{1}\in (0,\delta_{0})$ such that for any $j\in \mathds{Z}_{\geq 0}$ and $\delta\in(0,\delta_{1})$, we have the following results:
\begin{enumerate}[(i)]
\item For any $w\in L^{2}_{j;-\delta,\delta}(Z_{+};\mathds{R})$ with $
\int_{Z_{+}} w\,d\operatorname{vol}=0,
$
we can find $u\in L^{2}_{j+2,\operatorname{loc}}(Z_{+};\mathds{R})$ satisfying
$$
|du|_{L^{2}_{j+1;-\delta,\delta}}<\infty,\ \Delta u=w.
$$
\item The operator $$
D(M_{+}):L^{2}_{j+1,\delta}(M_{+};T^{*}M_{+})\rightarrow  L^{2}_{j,\delta}(M_{+};\mathds{R}\oplus \wedge_{2}^{+}T^{*}M_{+} ):\alpha\mapsto (d^{*}\alpha,d^{+}\alpha)
$$ is Fredholm with index  $-(b_{2}^{+}(M)+1)$;
\item The operator $$D(Z_{+}):L^{2}_{j+1;-\delta,\delta}(Z_{+};T^{*}Z_{+})\rightarrow L^{2}_{j;-\delta,\delta}(Z_{+};\mathds{R}\oplus \wedge^{2}_{+}T^{*}Z_{+}):
\alpha\mapsto (d^{*}\alpha,d^{+}\alpha)
$$ is Fredholm with trivial kernel and $1$-dimensional cokernel. Its image equals $$\{(w,\beta)|\ \int_{Z_{+}}w \,d\operatorname{vol}=0\}.$$
\item The operator $$D(X_{+}):L_{j+1,\delta}^{2}(X_{+};T^{*}X_{+})\rightarrow L^{2}_{j,\delta}(X_{+};\mathds{R}\oplus \wedge^{2}_{+}T^{*}X_{+})\oplus  L^{2}_{j+1/2}(Y;\mathds{R})\oplus C^{+}$$
given by
\begin{equation}\label{half De Rham with boundary}
\alpha \mapsto (d^{*}\alpha,d^{+}\alpha,\langle \alpha,\vec{v}\rangle,\pi^{+}(\alpha|_{Y}))
\end{equation}
is Fredholm with trivial kernel and one dimensional cokernel, which can be canonically identified with $\mathds{R}$. Here $C^{+}$ (resp. $C^{-}$) is the closure in $ L^{2}_{j+1/2}(Y;T^{*}Y)\cap \operatorname{ker} d^{*}$ of the space spanned by the eigenvectors of $*d$ with positive (resp. negative) eigenvalues and $$\pi^{+}:L^{2}_{j+1/2}(Y;iT^{*}Y)\rightarrow C^{+}$$ is the projection with kernel $C^{-}$.
\end{enumerate}

\end{lem}
\begin{proof}
(i) We consider two vector spaces:
$$
V_{1}=\{u\in L^{2}_{j+2,\text{loc}}(Z_{+};\mathds{R})|\ \|du\|_{L^{2}_{j+1;-\delta,\delta}}<\infty\}
$$
$$
V_{2}=\{w\in L^{2}_{j;-\delta,\delta}(Z_{+};\mathds{R})|\ \int_{Z_{+}} w\,d\text{vol}=0\}.
$$
Now assume $\delta \in (0,\delta_{0})$, where $\delta_{0}$ is the constant in Proposition \ref{laplace equation}. By Lemma \ref{Taubes's lemma}, we also have \begin{equation}\label{equivalent defi}V_{1}=L^{2}_{j+2;-\delta,\delta}(Z_{+};\mathds{R})\oplus\mathds{R}\oplus \mathds{R}\tau_{4}.
\end{equation}
Using this identification and integration by part, we can show that $\Delta u\in V_{2}$ for any $u\in V_{1}$. In other words, we have a well defined operator
$$
\Delta: V_{1}\rightarrow V_{2}.
$$
Comparing the domain and target of this operator with the one in Proposition \ref{laplace equation} (1), we see that it is a Fredholm operator with index $1$. To finish the proof, we just need to prove kernel of $\Delta$ consists only of constant functions. This is a simple consequence of the maximum principle, noticing that all functions in $V_{1}$ are bounded (because of (\ref{equivalent defi})).

(ii) Consider the operator
$$
d^{+}:L^{2}_{j+1,\delta}(M_{+};T^{*}M_{+})\rightarrow L^{2}_{j,\delta}(M_{+};\wedge_{+}^{2}T^{*}M_{+} ).
$$
Note that $H^{1}(M_{+};\mathbb{R})=0$ by our choice of $M$ (see (\ref{equa: b1mequalszero})). By \cite[Proposition 5.1]{Taubes}, when $\delta_{1}>0$ is small enough, both the kernel and the image of this operator  (which we denote by $V_{3}$ and $V_{4}$ respectively) are closed with the following properties:
\begin{equation}\label{1st homology vanish}
V_{3}\cong L^{2}_{j+2,\delta}(M_{+};\mathds{R}):du\leftrightarrow u;
\end{equation}
\begin{equation}\label{2nd homology}
\text{dim}(L^{2}_{j,\delta}(M_{+};\wedge_{2}^{+}T^{*}M_{+} )/V_{4})=b_{2}^{+}(M).
\end{equation}

By (\ref{1st homology vanish}), the operator
$$
V_{3}\rightarrow L^{2}_{j-1,\delta}(M_{+};\mathds{R}): \alpha\mapsto d^{*}\alpha.
$$
is essentially the same with the operator $\Delta(M_{+},\delta)$ in Proposition \ref{laplace equation}, which is Fredholm with index $-1$. This  implies that the operator
$$
L^{2}_{j,\delta}(M_{+};T^{*}M_{+})\rightarrow L^{2}_{j-1,\delta}(M_{+};\mathds{R})\oplus V_{4}:\alpha\mapsto (d^{*}\alpha,d^{+}\alpha)
$$
is also Fredholm with the same index. Therefore, by (\ref{2nd homology}), the operator
$$
L^{2}_{j+1,\delta}(M_{+};T^{*}M_{+})\rightarrow L^{2}_{j,\delta}(M_{+};\mathds{R})\oplus L^{2}_{j,\delta}(M_{+};\wedge^{2}_{+}T^{*}M_{+} ):\alpha\mapsto (d^{*}\alpha,d^{+}\alpha)
$$
is Fredholm with index $-(b_{2}^{+}(M)+1)$.

(iii) To apply the excision principle of the index, we consider the manifold $M_{-}=Z\cup_{Y}\bar{M}$. (Recall that $\bar{M}$ is the orentation reversal of $M$.) We choose a  function
$$\tau:M_{-}\rightarrow [0,+\infty)\text{ with }\tau(t,y)=|t|,\ \forall(t,y)\in (-\infty,-1]\times Y$$
and define the weighted Sobolev norm of a section $s$ over $M_{-}$ as
$$
\|s\|_{L^{2}_{j,-\delta}}:=\|e^{\delta\tau}s\|_{L^{2}_{j}}
.$$
By similar argument as (ii), one can show that the operator
$$
L^{2}_{j+1,-\delta}(M_{-};T^{*}M_{-})\rightarrow L^{2}_{j,-\delta}(M_{-};\mathds{R}\oplus \wedge^{2}_{+}T^{*}M_{-} ) :\alpha\mapsto (d^{*}\alpha,d^{+}\alpha)
$$
is Fredholm with index $-(b_{2}^{+}(\bar{M})+1)$. Notice that we have the decompositions
$$
M_{+}=M\cup_{Y}X_{+},\ M_{-}=Z\cup_{Y}\bar{M},\ Z_{+}= Z\cup_{Y}X_{+}.
$$
By an exision argument, we see that the operator
$$
(d^{*},d^{+}):L^{2}_{j+1;-\delta,\delta}(Z_{+},T^{*}Z_{+})\rightarrow L^{2}_{j;-\delta,\delta}(Z_{+},\mathds{R}\oplus \wedge^{2}_{+}T^{*}Z_{+})
$$
is Fredholm with index
$$
-(1+b_{2}^{+}(M))-(1+b_{2}^{+}(\bar{M}))+(1+b_{2}^{+}(M\cup_{Y}\bar{M}))=-1.
$$
Having proved this fact, we are left to show that the kernel is trivial. Suppose we have $$\alpha\in L^{2}_{j+1;-\delta,\delta}(Z_{+};T^{*}Z_{+})\text{ with }d^{*}\alpha=0,d^{+}\alpha=0.$$ Integrating by part, we get $d\alpha=0$. Since $H^{1}(Z_{+};\mathds{R})=0$, we have $\alpha=du$ for some harmonic function $u$. Notice that $\|du\|_{L^{2}_{j+1,-\delta,\delta}}<\infty$. By Lemma \ref{Taubes's lemma}, the function $u$ is bounded. By the maximal principle, $u$ is a constant, which implies $\alpha=du=0$.

(iv) Consider the operator
$$D(\bar{M}):L_{j+1}^{2}(\bar{M};T^{*}\bar{M})\rightarrow L^{2}_{j}(\bar{M};\mathds{R}\oplus \wedge^{2}_{+}T^{*}\bar{M})\oplus  L^{2}_{j+1/2}(Y;\mathds{R})\oplus C^{+}$$
defined by the same formula as (\ref{half De Rham with boundary}). By Lemma 24.8.1 of the book, $D(\bar{M})$ is a Fredholm operator with index $-b^{+}(\bar{M})-1$. We note that the boundary of $\bar{M}$ is $-Y$ while the boundary of the manifold in that Lemma is $Y$, this explains the reason we use $C^{+}$ while the book use $C^{-}$. We also note that the additional term ``$-1$'' in our index formula comes from the $1$-dimensional cokernel of the map
$$D(\bar{M}):L_{j+1}^{2}(\bar{M};T^{*}\bar{M})\rightarrow L^{2}_{j}(\bar{M};\mathds{R}\oplus i\wedge^{2}_{+}T^{*}\bar{M})\oplus  L^{2}_{j+1/2}(Y;\mathds{R})$$
$$\alpha\mapsto (d^{*}\alpha, d^{+}\alpha,\langle\alpha,\vec{v}\rangle ).$$ By an excision argument involving the operators $D(X_{+}),D(\bar{M}),D(M_{+})$ and the operator
$$d^{*}\oplus d^{+}:L^{2}_{j+1}(M\cup_{Y}\bar{M};T^{*}(M\cup_{Y}\bar{M}))\rightarrow L^{2}_{j}(M\cup_{Y}\bar{M};\mathds{R}\oplus \wedge^{2}_{+}T^{*}(M\cup_{Y}\bar{M})),$$
we can prove that  $D(X_{+})$ is Fredholm with index $-1$. Now suppose $\alpha\in\operatorname{ker}D(X_{+})$. Then by the integration by part argument on page 502 of the book, we can prove $d\alpha=0$. Since $H^{1}(X_{+};\mathds{R})=0$, we have $\alpha=df$ for some local $L^{2}_{j+1}$-function $f$. By Lemma \ref{Taubes's lemma}, we can assume $\|f\|_{L^{2}_{j+1,\delta}}<\infty$ after adding some constant function. Then $f$ satisfies
$\Delta f=0,\ \langle df,\vec{v}\rangle=0$.
By Lemma \ref{laplace equation}, we see that $f$ (hence also $\alpha$) equals $0$. We have proved that the kernel is trivial, which implies that the cokernel is $1$-dimensional. Using integration by part again, one can easily see that a necessary condition for an element $$(w_{1},\beta,w_{2},\alpha' )\in  L^{2}_{j,\delta}(X_{+};\mathds{R}\oplus \wedge^{2}_{+}T^{*}X_{+})\oplus  L^{2}_{j+1/2}(Y;\mathds{R})\oplus C^{+}$$ belonging to $\operatorname{im}D(X_{+})$ is $$\int_{X_{+}}w_{1}d\text{vol}+\int_{Y}w_{2}d\text{vol}=0.$$ Since the cokernel is $1$-dimensional, we see that this is also a sufficient condition. Moreover, we have a canonical isomorphism $$\operatorname{coker}D(X_{+})\cong \mathds{R}:\ [(w_{1},\beta,w_{2},\alpha' )]\leftrightarrow \int_{X_{+}}w_{1}d\text{vol}+\int_{Y}w_{2}d\text{vol}.$$
\end{proof}

Now we study the Fredholm properties related to the linearized Seiberg-Witten equations. Recall that we chose an ``admissible metric'' $g_{X}$ on $X$ (see Assumption \ref{admissible metric}). Under this assumption, we have the following proposition.

\begin{pro}[\cite{MRS}]\label{Dirac operator is Fredholm}There exists a number $\delta_{2}>0$ such that for any $\delta\in (-\delta_{2},\delta_{2}),j\in \mathds{Z}_{\geq 0}$, the end-periodic Dirac operator
$$
\slashed{D}^{+}_{A_{0}}:L^{2}_{j+1,\delta}(M^{+};S^{+})\rightarrow L^{2}_{j,\delta}(M^{+};S^{-})
$$
is Fredholm. Moreover, the number $$\operatorname{ind}_{\mathds{C}}(\slashed{D}^{+}_{A_{0}}(M_{+}))+\frac{\operatorname{sign}(M)}{8}$$ is an invariant of the pair $(X,g_{X})$, which we denote by $w(X,g_{X},0)$.
\end{pro}

To end this subsection, let us consider the Atiyah-Patodi-Singer boundary problem on the end-periodic manifold $X_{+}$. This will be essential in our study of local structure of the Seiberg-Witten moduli space. To simplify our notation,  we denote the following bundles over $X_{+}$
$$iT^{*}X_{+}\oplus S^{+}\text{ and }i(\mathds{R}\oplus \wedge_{+}^{2}T^{*}X_{+})\oplus S^{-}$$ respectively by $E_{1}$ and $E_{2}$. We also write $F_{1}$ for the bundle $i(\mathds{R}\oplus T^{*}Y)\oplus S$ over $Y$.

Recall that $k$ is a fixed integer greater than $2$. First consider the linear operator
\begin{equation}\label{equation: defition of D}
    D=D_{0}+K:L^{2}_{k,\delta}(X_{+};E_{1})\rightarrow L^{2}_{k-1,\delta}(X_{+};E_{2}),
\end{equation}
where $D_{0}=(d^{*},d^{+},\slashed{D}_{A_{0}})$ and $K$ is an operator that can be extended to a bounded operator
$$
K:L^{2}_{j,\delta}(X_{+};E_{1})\rightarrow L^{2}_{j,2\delta}(X_{+};E_{2})
$$
for any integer $j\in [-k,k]$. Next, we define the restriction map
$$
r:L^{2}_{k,\delta}(X_{+};E_{1})\rightarrow L^{2}_{k-1/2}(Y;F_{1})$$
$$(a,\phi)\mapsto (\langle a,\vec{v}\rangle,a|_{Y},\phi|_{Y}).$$
 Let $H_{0}^{+}$ (resp. $H_{0}^{-}$) be the closure in $L^{2}_{1/2}(Y;F_{1})$ of the span of the eigenvectors eigenvalues of operator
$$
L_{0}: C^{\infty}(Y;F_{1})\rightarrow C^{\infty}(Y;F_{1})
$$
$$
(u,\alpha,\phi)\mapsto (d^{*}\alpha,*d\alpha-du,\slashed{D}_{A_{0}}\phi).
$$
with positive (resp. non-positive) eigenvalues. We write $\Pi_{0}$ for the projection
$$
L^{2}_{1/2}(Y;F_{1})\rightarrow L^{2}_{1/2}(Y;F_{1})
$$
with image $H_{0}^{-}$ and kernel $H_{0}^{+}$. It also maps $L^{2}_{s}(Y;F_{1})$ to $L^{2}_{s}(Y;F_{1})$ for all $s$. Consider another projection
$$
\Pi: L^{2}_{1/2}(Y;F_{1})\rightarrow L^{2}_{1/2}(Y;F_{1})$$
satisfying
$$\Pi (L^{2}_{s}(Y;F_{1}))\subset L^{2}_{s}(Y;F_{1})$$ for any $s$. We say that $\Pi$ and $\Pi_{0}$ are $k$-commonmensurate if the difference $$
\Pi-\Pi_{0}: L^{2}_{j-1/2}(Y;F_{1})\rightarrow L^{2}_{j-1/2}(Y;F_{1})
$$
is a compact operator, for all $1\leq j\leq k$. We write $H^{-}$ for $\operatorname{im}(\Pi)\subset L^{2}_{1/2}(Y;F_{1})$ and $H^{+}$ for $\operatorname{im}(1-\Pi)\subset L^{2}_{1/2}(Y;F_{1}).$
\begin{pro}\label{APS}
Let $\delta_{1},\delta_{2}$ be the constant provided by
Lemma \ref{half De rham complex} and Proposition \ref{Dirac operator is Fredholm} respectively. Then for any $\delta\in (0,\min(\delta_{1},\delta_{2}))$ and any $1\leq j\leq k$, the operator
$$
D\oplus ((1-\Pi)\circ r): L^{2}_{j,\delta}(X_{+};E_{1})\rightarrow L^{2}_{j-1,\delta}(X_{+};E_{2})\oplus (H^{+}\cap L^{2}_{j-1/2}(Y;F_{1}))
$$
is Fredholm. In addition, if $u_{i}$ is a bounded sequence in $L^{2}_{j,\delta}(X_{+};E_{1})$ and $Du_{i}$ is Cauchy in $L^{2}_{j-1,\delta}(X_{+};E_{2})$, then $\Pi\circ r(u_{i})$ has a convergent subsequence in $L^{2}_{j-1/2}(Y;F_{1})$. In particular, the maps $\Pi\circ r$ and $(1-\Pi)\circ r$ restricted to the kernel of $D$, are respectively, compact and Fredholm.
\end{pro}
\begin{proof}
We consider the following two operators:
\begin{itemize}
\item The operator over $\bar{M}$
$$
(d^{*},d^{+},\slashed{D}_{A_{0}})\oplus (1-\Pi)\circ r_{\bar{M}}: L^{2}_{j}(\bar{M})\rightarrow L^{2}_{j-1}(\bar{M})\oplus   (H^{+}\cap L^{2}_{j-1/2}(Y)),$$ where $r_{\bar{M}}$ is the restriction map defined similarly as $r$;
\item The operator over $M_{+}$
$$
(d^{*},d^{+},\slashed{D}_{A_{0}}):  L^{2}_{j,\delta}(M_{+})\rightarrow L^{2}_{j-1,\delta}(M_{+}).$$
\end{itemize}
 By Proposition 17.2.5 of the book, Lemma \ref{half De rham complex} and Proposition \ref{Dirac operator is Fredholm}, both of these two operators are Fredholm. Note that they correspond to the operator $D_{0}\oplus((1-\Pi)\circ r)$ on $X_{+}$. We can prove the Fredholm property of $D_{0}\oplus((1-\Pi)\circ r)$ using standard parametrix patching argument (see Page 245 of the book). Since the embedding $L^{2}_{j,2\delta}\rightarrow L^{2}_{j-1,\delta}$ is compact, the operator $D\oplus((1-\Pi)\circ r)$ is a compact perturbation of $D_{0}\oplus((1-\Pi)\circ r)$ and we conclude that $D\oplus((1-\Pi)\circ r)$ is also Fredholm. To prove the second part of the Proposition, we multiply the sequence $\{u_{i}\}$ by a bump function $\beta$ supporting near $\partial X_{+}$ and follow the argument on Page 304 of the book.
\end{proof}
\subsection{The invariant $\lambda_{\textnormal{SW}}(X)$} Now we review the definition of $\lambda_{\textnormal{SW}}(X)$. By \cite[Lemma 2.1]{MRS}, for a generic pair $(g_{X},\beta)$ with $\beta\in L^{2}_{k+1}(X;iT^{*}X)$, the blown-up Seiberg-Witten moduli space $\mathcal{M}(X,g_{X},\beta)$ consisting of the gauge equivalence classes of the triples $$(A,s,\phi)\in \mathcal{A}_{k}(X)\times \mathds{R}_{\geq 0}\times L^{2}_{k}(X;S^{+}),\ \|\phi\|_{L^{2}}=1$$ that solve the Seiberg-Witten equation
\begin{equation}\label{blown up seiberg-witten}
\left\{\begin{array} {cc}
 F_{A}^{+}-s^{2}(\phi\phi^{*})_{0}=d^{+}\beta  \\
 \slashed{D}_{A}^{+}(X,g_{X})(\phi)=0
\end{array}\right.
\end{equation}
is an oriented manifold of dimension $0$ and contains no reducible points (i.e. triples with $s=0$). We call such $(g_{X},\beta)$ a regular pair. Now consider the end-periodic (perturbed) Dirac operator
$$
\slashed{D}^{+}_{A_{0}}(M_{+},g_{M_{+}})+\rho(\beta'): L^{2}_{1}(M_{+};S^{+})\rightarrow  L^{2}(M_{+};S^{-}).
$$
where $\beta'$ is an imaged valued one form on $M_{+}$ that equals the pull back of $\beta$ when restricted to $X_{+}$. As proved in \cite{MRS}, this operator is Fredholm and the quantity
$$
\operatorname{ind}_{\mathds{C}}(\slashed{D}^{+}_{A_{0}}(M_{+},g_{M_{+}})+\rho(\beta'))+\frac{\text{sign}(M)}{8}
$$
is an invariant of $(X,g_{X},\beta)$, which we denote by $w(X,g_{X},\beta)$.
\begin{thm}[\cite{MRS}] The number $
\#\mathcal{M}(X,g_{X},\beta)-w(X,g_{X},\beta)
$
does not depend on the choice of regular pair $(g_{X},\beta)$ and hence is an invariant of the manifold of $X$, which we define as $\lambda_{\textnormal{SW}}(X)$; morveover, the reduction of $\lambda_{\textnormal{SW}}(X)$ modulo $2$ is the Rohlin invariant of $X$.
\end{thm}
\begin{lem}\label{casson for psc}
Suppose $g_{X}$ is a metric with positive scalar curvature. Then the pair $(g_{X},0)$ is regular and $\lambda_{SW}(X)=-\omega(X,g_{X},0)$.
\end{lem}
\begin{proof}
This is a simple consequence of the Weitzenb\"ock formula.
\end{proof}

\begin{lem}\label{orientation reversal 2}
Suppose $X$ admits a metric $g_{X}$ with positive scalar curvature. Then we have
$
\lambda_{\textnormal{SW}}(X)=-\lambda_{\textnormal{SW}}(-X)
$.
\end{lem}
\begin{proof}
 By Lemma \ref{casson for psc}, we have $\lambda_{\textnormal{SW}}(X)=w(X,g_{X},0)$. Similarly, $\lambda_{\textnormal{SW}}(-X)=w(-X,g_{X},0)$. Notice that $$\text{sign}(M)+\text{sign}(\bar{M})=\text{sign}(M\cup_{Y}\bar{M})=0.$$ By an excision argument relating indices of the Dirac operator on  $M_{+}\cup \bar{M}_{+}$ (where $\bar{M}_{+}$ denotes the orientation reversal of $M_{+}$) and the Dirac operator on $(M\cup_{Y}\bar{M})\cup \tilde{X}$, we get
\begin{equation}
\begin{split}
w(X,g_{X},0)+w(-X,g_{X},0)=&\operatorname{ind}_{\mathds{C}}\slashed{D}^{+}_{A_{0}}(M_{+},g_{M_{+}})+\operatorname{ind}_{\mathds{C}}\slashed{D}^{+}_{A_{0}}(\bar{M}_{+},g_{M_{+}})\\=&\operatorname{ind}_{\mathds{C}}\slashed{D}^{+}_{A_{0}}(\tilde{X},g_{\tilde{X}}),
\end{split}
\end{equation}
where $$\slashed{D}^{+}_{A_{0}}(\tilde{X},g_{\tilde{X}}):L^{2}_{1}(\tilde{X};S^{+})\rightarrow L^{2}(\tilde{X};S^{-})$$
is the (unperturbed) Dirac operator on $\tilde{X}$. As in the proof of \cite[Proposition 5.4]{MRS}, this operator has index $0$. Therefore, we have proved the lemma.
\end{proof}
\begin{rmk}
It was conjectured in \cite{MRS} that the relation $
\lambda_{\textnormal{SW}}(X)=-\lambda_{\textnormal{SW}}(-X)$ holds for a general homology $S^{3}\times S^{1}$ (without any assumption on the metric). This conjecture is still open. \end{rmk}

\section{Gauge theory on end-periodic manifolds}
In this section, we study the gauge theory on the end-periodic manifolds. First, we will carefully set up the (blown up) configuration space, the gauge group and the moduli spaces. Once this was done correctly, the arguments in Section 24 and 25 of the book can be repeated without essential difficulty. For this reason, some proofs in this section will only be sketched and we refer to the book for complete details.

 Let $\delta$ be a positive number smaller than $\min(\delta_{1},\delta_{2})$, where $\delta_{1},\delta_{2}$ are constants provided by Lemma \ref{half De rham complex} and Proposition \ref{Dirac operator is Fredholm} respectively. We let
 $$
 \mathcal{A}_{k,\delta}(X_{+})=\{A_{0}+a| a\in L^{2}_{k,\delta}(X_{+};iT^{*}X_{+})\}
 $$
 be the space of spin$^{\text{c}}$ connections of class $L^{2}_{k,\delta}$. The configuration spaces are defined as
\begin{equation}\label{configuration space}
\begin{split}
\mathcal{C}_{k,\delta}(X_{+})=\mathcal{A}_{k,\delta}(X_{+})\times L^{2}_{k,\delta}(X_{+};S^{+});\\
\mathcal{C}_{k,\delta}^{\sigma}(X_{+})=\{(A,s,\phi)| A\in\mathcal{A}_{k,\delta}(X_{+}),\phi\in L_{k,\delta}^{2}(X_{+};S^{+}), \|\phi\|_{L^{2}}=1, s\in \mathds{R}_{\geq 0}\}.\end{split}
\end{equation}
It is easy to see that $\mathcal{C}_{k,\delta}(X_{+})$ is a Hilbert manifold without boundary, while $\mathcal{C}_{k,\delta}^{\sigma}(X_{+})$ is a Hilbert manifold with boundary. There is a map $\boldsymbol{\pi}:\mathcal{C}_{k,\delta}^{\sigma}(X_{+})\rightarrow \mathcal{C}_{k,\delta}(X_{+})$ given by
$$
\boldsymbol{\pi}(A,s,\phi)=(A,s\phi).
$$
Next, we define the gauge groups
$$
\mathcal{G}^{0}_{k+1,\delta}(X_{+})=\{u:X_{+}\rightarrow S^{1}| (1-u)\in L^{2}_{k+1,\delta} (X_{+};\mathds{C}) \};
$$
$$
\mathcal{G}_{k+1,\delta}(X_{+})=\mathcal{G}_{c}\times \mathcal{G}^{0}_{k+1,\delta}(X_{+}),
$$
where $\mathcal{G}_{c}\cong S^{1}$ denotes the group of constant gauge transformations. Note that we impose the $L^{2}_{k+1,\delta}$-topology on $\mathcal{G}^{0}_{k+1,\delta}(X_{+})$ and the product topology on $\mathcal{G}_{k+1,\delta}(X_{+})$. Using the equality $$1-uv=(1-u)+(1-v)-(1-u)(1-v)$$ together with the Sobolev multiplication theorem, one can prove that $\mathcal{G}^{0}_{k+1,\delta}$ (and hence $\mathcal{G}_{k+1,\delta}$) is a group. A standard argument (see \cite{Taubes} and \cite{Freed-Uhlenbeck} for the non-abelian case) shows that they are actually  Hilbert Lie groups. The Lie algebra of $\mathcal{G}_{k+1,\delta}$ is given by \begin{equation}\label{lie algebra}T_{e}\mathcal{G}_{k+1,\delta}(X_{+})=\mathds{R}\oplus L^{2}_{k+1,\delta}(X_{+};i\mathds{R}).\end{equation}
\begin{rmk}
Our main concern will be the group $\mathcal{G}_{k+1,\delta}(X_{+})$, while the group $\mathcal{G}^{0}_{k+1,\delta}(X_{+})$ is introduced to smooth the arguments.
\end{rmk}
The actions of $\mathcal{G}_{k+1,\delta}(X_{+})$  on  $\mathcal{C}_{k,\delta}(X_{+})$ and $\mathcal{C}^{\sigma}_{k,\delta}(X_{+})$ are respectively given by
$$
u\cdot (A,\Phi)=(A-u^{-1}du,u\Phi)$$ and $$u\cdot (A,s,\phi)=(A-u^{-1}du,s,u\phi).
$$
Note that the latter action is free. We denote the quotient spaces by $\mathcal{B}_{k,\delta}(X_{+})$ and $\mathcal{B}^{\sigma}_{k,\delta}(X_{+})$ respectively.

\begin{lem}
$\mathcal{B}^{\sigma}_{k,\delta}(X_{+})$ is Hausdorff.
\end{lem}
\begin{proof}
By standard argumet, the proof is reduced to the following claim:

\begin{claim}: Suppose we have sequences $\{u_{n}\}\subset \mathcal{C}^{\sigma}_{k,\delta}(X_{+}),\ \{g_{n}\}\subset \mathcal{G}_{k+1,\delta}(X_{+})$ such that $u_{n}\rightarrow u_{\infty}\text{ and } g_{n}u_{n}\rightarrow v_{\infty}$
for some $u_{\infty},v_{\infty}\in \mathcal{C}^{\sigma}_{k,\delta}(X_{+})$. Then we can find $g_{\infty}\in \mathcal{G}_{k+1,\delta}(X_{+})$ such that $g_{\infty}u_{\infty}=v_{\infty}$.\end{claim}
To prove the claim, we let $u_{n}=(A_{n},s_{n},\phi_{n})$. Then both $A_{n}$ and $A_{n}-g_{n}^{-1}dg_{n}$ converges in $L^{2}_{k+1,\delta}$ norm, which implies that the sequence  $\{g_{n}^{-1}dg_{n}\}$ is Cauchy in $L^{2}_{k,\delta}(X_{+};i\mathds{R})$. Let $g_{n}=e^{\xi_{n}}$. Then $\{d\xi_{n}\}$ is Cauchy in $L^{2}_{k,\delta}(X_{+};i\mathds{R})$. By Lemma \ref{Taubes's lemma}, we can find numbers $\bar{\xi}_{n}\in i\mathds{R}$ such that $\{\xi_{n}-\bar{\xi}_{n}\}$ is a Cauchy sequence in $L^{2}_{k+1,\delta}(X_{+};i\mathds{R})$. Using the fact that the exponential map
$$
 L^{2}_{k+1,\delta}(X_{+};iT^{*}X_{+})\rightarrow \mathcal{G}^{0}_{k+1,\delta}(X_{+}):\ \xi\mapsto e^{\xi}
$$
is well defined and continuous (which is a consequence of the Sobolev multiplication theorem). We see that $\{e^{\xi_{n}-\bar{\xi}_{n}}\}$ is a Cauchy sequence in $\mathcal{G}^{0}_{k+1,\delta}(X_{+})$.

On the other hand, by replacing $\xi_{n}$ with $\xi_{n}-2m_{n}\pi i$ for $m_{n}\in \mathds{Z}$. We can assume $\bar{\xi}_{n}\in [0,2\pi)$. After passing to a subsequence, we may assume $\bar{\xi}_{n}$ converges to some number $\bar{\xi}_{\infty}$, which implies $e^{\bar{\xi}_{n}}$ converges to $e^{\bar{\xi}_{\infty}}$ as elements of $\mathcal{G}_{c}$.

Now we see that $g_{n}=e^{\bar{\xi}_{n}}\cdot e^{\xi_{n}-\bar{\xi}_{n}}$ has a subsequencial limit $g_{\infty}$ in $\mathcal{G}_{k+1,\delta}(X_{+})$. Since the action of $\mathcal{G}_{k+1,\delta}(X_{+})$ is continuous, we get $g_{\infty}\cdot u_{\infty}=v_{\infty}$.\end{proof}

Next, we define the local slice $S^{\sigma}_{k,\delta,\gamma}$ of the gauge action at $\gamma=(A_{0},s_{0},\phi_{0})\in \mathcal{C}^{\sigma}_{k,\delta}(X_{+})$. By taking derivative on gauge group action, we get a map
$$d^{\sigma}_{\gamma}:T_{e}\mathcal{G}_{k+1,\delta}(X_{+})\rightarrow T_{\gamma}\mathcal{C}^{\sigma}_{k,\delta}(X_{+})$$
$$
\xi\mapsto(-d\xi,0,\xi\phi_{0}).
$$
We denote the image of $d^{\sigma}_{\gamma}$ by $\mathcal{J}^{\sigma}_{k,\delta,\gamma}$, which is the tangent space of the gauge orbit. To define its complement, we consider the subspace $\mathcal{K}^{\sigma}_{k,\delta,\gamma}\subset T_{\gamma}C^{\sigma}_{k,\delta}(X_{+})$ as the kernel of the operator (c.f. formula (9.12) of the book)
\begin{equation}\label{d-sigma-flat}\begin{split}
d^{\sigma,\flat}_{\gamma}:L^{2}_{k,\delta}(X_{+};i\mathds{R})\oplus \mathds{R}\oplus L^{2}_{k,\delta}(X_{+};S^{+})\rightarrow L^{2}_{k-1/2}(Y;i\mathds{R})\oplus L^{2}_{k-1,\delta}(X_{+};i\mathds{R})\\
(a,s,\phi)\mapsto (\langle a, \vec{v} \rangle, -d^{*}a+is_{0}^{2}\text{Re}\langle i\phi_{0},\phi\rangle +i|\phi_{0}|^{2}\cdot \int_{X_{+}} \text{Re}\langle i\phi_{0},\phi\rangle\,d\text{vol})
\end{split}\end{equation}
\begin{rmk}To motivate this construction, we note that when $s_{0}>0$,  $\mathcal{K}^{\sigma}_{k,\delta,\gamma}$ is obtained by lifting the $L^{2}$-orthogonal complement of the tangent space of the gauge orbit (through $\boldsymbol{\pi}(\gamma)$) in $\mathcal{C}_{k,\delta}(X_{+})$.
\end{rmk}
\begin{rmk}
We also note that in the book, the integral in the formula corresponding to (\ref{d-sigma-flat}) is divided by the total volume of the $4$-manifold. However, this difference is not essential because the kernel is not affected.
\end{rmk}
\begin{lem}\label{decomposition of the tangent space}
For any $\gamma$, we have a decomposition $T_{\gamma}\mathcal{C}^{\sigma}_{k,\delta}(X_{+})=\mathcal{J}^{\sigma}_{k,\delta,\gamma}\oplus \mathcal{K}^{\sigma}_{k,\delta,\gamma}$.
\end{lem}
\begin{proof}
We want to show that for any $(a,s,\phi)\in T_{\gamma}\mathcal{C}^{\sigma}_{k,\delta}(X_{+})$, there exists a unique $\xi\in T_{e}\mathcal{G}_{k+1,\delta}(X_{+})$ such that
$$(a,s,\phi)-d^{\sigma}_{\gamma}\xi\in \mathcal{K}^{\sigma}_{k,\delta,\gamma}.$$
This is equivalent to the condition
\begin{equation}\label{perturbed laplace}
D\xi=(\langle a, \vec{v}\rangle,-is_{0}^{2}\operatorname{Re}\langle i\phi_{0},\phi\rangle -i|\phi_{0}|^{2}\int_{X_{+}} \operatorname{Re}\langle i\phi_{0},\phi\rangle d\text{vol} +d^{*}a)
\end{equation}
where the operator  $$D:T_{e}\mathcal{G}_{k+1,\delta}(X_{+})\rightarrow L^{2}_{k-1/2}(Y;i\mathds{R})\oplus L^{2}_{k-1,\delta}(X_{+};i\mathds{R})$$
is given by
$$
\xi\mapsto (\langle d\xi, \vec{v}\rangle, \Delta \xi +s_{0}^{2}|\phi_{0}|^{2}\xi +i |\phi_{0}|^{2} \int_{X_{+}} (-i\xi)|\phi_{0}|^{2}d\text{vol})
$$
Notice that the map
$$
\xi \mapsto s_{0}^{2}|\phi_{0}|^{2}\xi +i |\phi_{0}|^{2} \int_{X_{+}} (-i\xi)|\phi_{0}|^{2}d\text{vol}
$$
actually factors through the space $L^{2}_{k,2\delta}(X;i\mathds{R})$. Therefore, the operator $D$ is a compact perturbation of the operator $D'$ given by
$$
\xi\mapsto   (\langle d\xi, \vec{v}\rangle, \Delta \xi).
$$
The index of $D'$ (hence $D$) equals $0$ by Proposition \ref{laplace equation} (iii). Here the index is increased by $1$ because we have an additional summand $\mathds{R}$ in the domain (see (\ref{lie algebra})). As in the proof of Proposition 9.3.5 of the book, we can show that $D$ has trivial kernel using integration by part.  Therefore, $D$ is an isomorphism and (\ref{perturbed laplace}) has a unique solution. \end{proof}
\begin{rmk}
The integration by part argument over the noncompact manifold $X_{+}$ is justified by the following fact (which can be proved using bump function): For any $\delta>0$ and $\theta\in L^{2}_{k,\delta}(X_{+};\wedge^{3}T^{*}X_{+})$, we have
$$
\int_{X^{+}}d\theta=\int_ {\partial{X_{+}}}\theta.
$$
\end{rmk}
We define the local slice $\mathcal{S}^{\sigma}_{k,\delta,\gamma}\subset \mathcal{C}^{\sigma}_{k,\delta}(X_{+})$ (at $\gamma$) as the set of points $(A,s,\phi)$ satisfying
$$
d^{\sigma,\flat}_{\gamma}(A-A_{0},s,\phi)=0
$$
By Lemma 9.3.2 of the book, Lemma \ref{decomposition of the tangent space} has the following corollary.
\begin{cor}
$\mathcal{B}^{\sigma}_{k,\delta}(X_{+})$ is a Hilbert manifold with boundary. For any $\gamma\in C^{\sigma}_{k,\delta}(X_{+})$, there is an open neighborhood of $\gamma$ in the slice
$$
U\subset \mathcal{S}^{\sigma}_{k,\delta,\gamma}
$$
such that $U$ is a diffeomorphism onto its image under the natural projection from  $\mathcal{C}^{\sigma}_{k,\delta}(X_{+})$ to  $\mathcal{B}^{\sigma}_{k,\delta}(X_{+})$, which is an open neighborhood of $[\gamma]$ in $\mathcal{B}^{\sigma}_{k,\delta}(X_{+})$.
\end{cor}

Now we study the Seiberg-Witten equations on the manifold $X_{+}$. Let $\mathcal{V}^{\sigma}_{k,\delta}(X_{+})$ be the trivial bundle $\mathcal{C}^{\sigma}_{k,\delta}(X_{+})$ with fiber $L^{2}_{k-1,\delta}(i\mathfrak{su}(S^{+})\oplus S^{-})$, where $\mathfrak{su}(S^{+})$ denotes the bundle of skew-hermitian, trace-$0$ automorphisms on $S^{+}$. We consider a smooth section $$\mathfrak{F}^{\sigma}:C^{\sigma}_{k,\delta}(X_{+})\rightarrow \mathcal{V}^{\sigma}_{k,\delta}(X_{+})$$ given by
$$
\mathfrak{F}^{\sigma}(A,s,\phi)=(\frac{1}{2}\rho(F^{+}_{A^{t}})-s^{2}(\phi\phi^{*})_{0},\slashed{D}_{A}^{+}\phi)
$$
The zero locus of $\mathfrak{F}^{\sigma}$ describes the solution of the (blown-up) Seiberg-Witten equations.

To obtain the transversality condition, we introduce a perturbation on $\mathfrak{F}^{\sigma}$. This was done in the same way as the book: Recall that the boundary $\partial X_{+}$ has a neighborhood $N$ which is isomorphic to $[0,3]\times Y$ (with $\{0\}\times Y$ identified with $\partial X_{+}$). Pick two $3$-dimensional tame perturbations $\mathfrak{q}$ and $\mathfrak{p}_{0}$. We impose the following assumption on $\mathfrak{q}$:
\begin{assum}\label{3 dimensional perturbation}
$\mathfrak{q}$ is a nice perturbation with $\operatorname{ht}(\mathfrak{q})=-2w(X,g_{X},0)$. Such perturbation exists by Proposition \ref{height of nice perturbation}.
\end{assum}
These two perturbations induce, in a canonical way, 4-dimensional perturbations $\hat{\mathfrak{q}}^{\sigma},\hat{\mathfrak{p}}_{0}^{\sigma}$ on $N$ (see Page 153 and 155 of the book). Pick a cut-off function  $\beta$ that equals $1$ near $\{0\}\times Y$ and equals $0$ near $\{3\}\times Y$ and a bump function $\beta_{0}$ supported in $(0,-3)\times Y$. Then the sum
\begin{equation}\label{mixed perturbation}
\hat{\mathfrak{p}}^{\sigma}=\beta\cdot \hat{\mathfrak{q}}^{\sigma}+\beta_{0}\cdot \hat{\mathfrak{p}}_{0}^{\sigma}
\end{equation} is a section of $\mathcal{V}^{\sigma}_{k,\delta}(X_{+})$ with the property that: $\hat{\mathfrak{p}}^{\sigma}(A,s,\phi)\in L^{2}_{k-1,\delta}(i\mathfrak{su}(S^{+})\oplus S^{-})$ is supported in $N$ and only depends on $(A|_{N},s,\phi|_{N})$.

We denote by $\mathfrak{p}$ the $4$-dimensional perturbation given by the section $\hat{\mathfrak{p}}^{\sigma}$. Let $\mathfrak{F}^{\sigma}_{\mathfrak{p}}=\mathfrak{F}^{\sigma}+\hat{\mathfrak{p}}^{\sigma}$.
We can define the moduli spaces $$\mathcal{M}(X_{+})=\{(A,s,\phi)|\mathfrak{F}^{\sigma}_{\mathfrak{p}}(A,s,\phi)=0\}/\mathcal{G}_{k+1,\delta}(X_{+})\subset \mathcal{B}_{k,\delta}^{\sigma}(X_{+})$$
$$\mathcal{M}^{\text{red}}(X_{+})=\{[(A,s,\phi)]\in \mathcal{M}(X_{+})|\ s=0 \}$$
as the set of gauge equivalent classes of the solutions of the perturbed Seiberg-Witten equations. ( For simplicity, we do not include $\mathfrak{p}$ in our notations of moduli spaces.)
\begin{lem}
For any choice of perturbations $\mathfrak{q},\mathfrak{p}_{0}$, the moduli space $\mathcal{M}(X_{+})$ is always a Hilbert manifold with boundary $\mathcal{M}^{\operatorname{red}}(X_{+})$.
\end{lem}
\begin{proof}
The proof is essentially identical with Proposition 24.3.1 in the book. Just replace the manifold $X$ there with $X_{+}$ and use weighted Sobolev space through out the argument.
\end{proof}
Because of the unique continuation theorem (see Proposition 10,8.1 of the book), we have $\phi|_{\partial X_{+}}\neq 0$ for any $[(A,s,\phi)]\in \mathcal{M}(X_{+})$. Therefore, we have a well defined map
\begin{equation}\label{restriction from the right}
R_{-}:\mathcal{M}(X_{+})\rightarrow \mathcal{B}_{k-1/2}^{\sigma}(Y)
\end{equation}
given by
$$
(A,s,\phi)\mapsto (A|_{\partial X_{+}},s\|\phi|_{\partial X_{+}}\|_{L^{2}},\frac{\phi|_{\partial X_{+}}}{\|\phi|_{\partial X_{+}}\|_{L^{2}}}).
$$

Now we attach the cylindrical end $(-\infty,0]\times Y$ on $X_{+}$ and consider the Seiberg-Witten equations on the manifold $Z_{+}$. We define the configuration space  as
$$\mathcal{C}_{k;\text{loc},\delta}(Z_{+})=\{(A_{0}+a,\Phi)| (a,\Phi)\in L^{2}_{k,\text{loc}}(Z_{+};iT^{*}Z_{+}\oplus S^{+}),\ \|(a,\Phi)|_{X_{+}}\|_{L^{2}_{k,\delta}}<\infty\}$$
and gauge group as
$$
\mathcal{G}_{k+1;\text{loc},\delta}(Z_{+})=\{u:Z_{+}\rightarrow S^{1}|\ u\in L^{2}_{k+1,\text{loc}}(Z_{+};\mathds{C}),\ u|_{X_{+}}\in \mathcal{G}_{k+1,\delta}(X_{+})\}.
$$
Note that in the above definitons, we only impose the exponential decay condition over the periodic end. As before, the action of $\mathcal{G}_{k+1;\text{loc},\delta}(Z_{+})$ on $\mathcal{C}_{k;\text{loc},\delta}(Z_{+})$ is not free. Therefore, we need to blow up the configuration space. Since  $\mathcal{C}_{k;\text{loc},\delta}(Z_{+})$ is not a Banach manifold now, the blown-up configuration space should be defined in the following manner: Let $\mathds{S}$ be the topological quotient of the space
$$
\{\Phi\in L^{2}_{k,\text{loc}}(Z_{+};S^{+})|\|\Phi|_{X_{+}}\|_{L^{2}_{k,\delta}}<\infty\}\setminus \{0\}
$$
by the action of $\mathds{R}_{>0}$. The blown-up configuration configuration space is defined as
$$
\mathcal{C}^{\sigma}_{k;\text{loc},\delta}(Z_{+})=\{(A,\Phi,\phi)|(A,\Phi)\in\mathcal{C}_{k;\text{loc},\delta}(Z_{+}),\ \phi \in \mathds{S},\ \Phi\in \mathds{R}_{\geq 0}\phi\}.
$$
Now we define the blown-up quotient configuration space as $$\mathcal{B}^{\sigma}_{k;\text{loc},\delta}(Z_{+})=\mathcal{C}^{\sigma}_{k;\text{loc},\delta}(Z_{+})/\mathcal{G}_{k+1;\text{loc},\delta}(Z_{+}).$$

The bundle $\mathcal{V}^{\sigma}_{k;\text{loc},\delta}(Z_{+})$ and its section $\mathfrak{F}^{\sigma}_{\mathfrak{p}}(Z_{+})$ are defined similarly as the book. The section $\mathfrak{F}^{\sigma}_{\mathfrak{p}}(Z_{+})$ is invariant under the action of $\mathcal{G}_{k+1;\text{loc},\delta}(Z_{+})$. We omit the detail here because the specific definition is not important for us. Just keep in mind that the perturbation equals $\hat{\mathfrak{q}}^{\sigma}$ over the cylindrical end $Z$, equals $\hat{\mathfrak{p}}$ over $[0,3]\times Y$ and equals 0 on $Z_{+}\setminus (-\infty,3]\times Y$. We call $(A,\phi,\Phi)$ a ``$Z_{+}$-trajectory'' if $
 \mathfrak{F}^{\sigma}_{\mathfrak{p}}(Z_{+})(A,\phi,\Phi)=0$. This is equivalent to the condition that $(A,\Phi,\phi)$ satisfies the blown-up perturbed Seiberg-Witten equations
$$
\left\{\begin{array} {cc}
 F_{A}^{+}-(\Phi\Phi^{*})_{0}=\hat{\mathfrak{p}}^{0,\sigma}_{Z_{+}}(A,\Phi)  \\
 \slashed{D}_{A}^{+}\phi=\hat{\mathfrak{p}}^{1,\sigma}_{Z_{+}}(A,\phi)
\end{array}\right.
$$
where $\hat{\mathfrak{p}}^{0,\sigma}_{Z_{+}}(A,\Phi)$ and $\hat{\mathfrak{p}}^{1,\sigma}_{Z_{+}}(A,\phi)$ are certain perturbation terms supported on $(-\infty,3]\times Y$. The second equation should be thought as a homogeneous equation in $\phi$, i.e., both sides of the equation will be rescaled by the same factor as we change the representative of $\phi$. By the unique continuation theorem, we have $\phi|_{\{t\}\times Y}\neq 0$ for any $t\leq 0$. As a result, the triple
$
(A|_{\{t\}\times Y},\|\Phi_{t\times Y }\|,\tfrac{\phi}{\|\phi|_{\{t\}\times Y }\|_{L^{2}}})
$
gives a point of $\mathcal{C}^{\sigma}_{k-1/2}(Y)$, which we define to be the restriction $(A,\Phi,\phi)|_{\{t\}\times Y}$. By restricting to $(-\infty,0]\times Y$, a gauge equivalent class $[(A,\Phi,\phi)]\in \mathcal{B}^{\sigma}_{k;\text{loc},\delta}(Z_{+})$ of $Z_{+}$-trajectory gives a path  $(-\infty,0]\rightarrow B^{\sigma}_{k-1/2}(Y)$.

Let $[\mathfrak{b}]\in \mathcal{B}^{\sigma}_{k-1/2}(Y)$ be a critical point of $\mathfrak{F}^{\sigma}_{\mathfrak{q}}(Y)$. We consider the moduli space $$\mathcal{M}([\mathfrak{b}],Z_{+})=\{[\gamma]\in \mathcal{B}^{\sigma}_{k;\text{loc},\delta}(Z_{+})|\  \mathfrak{F}^{\sigma}_{\mathfrak{p}}(Z_{+})(\gamma)=0,\ \mathop{\lim}\limits_{t\rightarrow -\infty}[\gamma|_{\{t\}\times Y}]=[\mathfrak{b}]\}.$$
It consists of $Z_{+}$-trajectories that are asymptotic to $[\mathfrak{b}]$ over the cylindrical end. By restricting to the submanifolds $Z$ and $X_{+}$, we get a map
\begin{equation}\label{restriction}
\rho:\mathcal{M}([\mathfrak{b}],Z_{+})\rightarrow \mathcal{M}([\mathfrak{b}],Z)\times \mathcal{M}(X_{+}).
\end{equation}
Here $\mathcal{M}([\mathfrak{b}],Z)$ denotes moduli space of Seiberg-Witten half-trajectories with limit $[\mathfrak{b}]$. In other words, $\mathcal{M}([\mathfrak{b}],Z)$ consists of gauge equivalent classes of paths
$$
\gamma:(-\infty,0]\rightarrow \mathcal{C}^{\sigma}_{k-1/2}(Y)\text{ with }\frac{d}{dt}\gamma(t)+\mathfrak{F}^{\sigma}_{\mathfrak{q}}(Y)(\gamma(t))=0,\ \mathop{\lim}_{t\rightarrow -\infty}\gamma(t)=\mathfrak{b}.
$$
Just like $\mathcal{M}(X_{+})$, the moduli space $\mathcal{M}([\mathfrak{b}],Z)$ is always a Hilbert manifold with boundary $\mathcal{M}^{\text{red}}([\mathfrak{b}],Z)$ (the moduli space of reducible half-trajectories) for arbitary perturbation. Note that we have a well defined restriction map\begin{equation}\label{restriction from the left} R_{+}: \mathcal{M}([\mathfrak{b}],Z)\rightarrow \mathcal{B}^{\sigma}_{k-1/2}(Y) \text{ given by }[\gamma]\mapsto [\gamma(0)].\end{equation}

\begin{lem}\label{fiber sum}
The map $\rho$ is a homeomorphism from $\mathcal{M}(Z_{+},[\mathfrak{b}])$ to its image, which equals the fiber product $\operatorname{Fib}(R_{-},R_{+})$. (The maps $R_{\pm}$ are defined in (\ref{restriction from the right}) and (\ref{restriction from the left}) respectively.)
\end{lem}
\begin{proof}This lemma is essentially a result on gluing two monopoles (or gauge transformations) on $Z$ and $X_{+}$, along the common boundary $Y$. In particular, all the analysis are carried out in a collar neighborhood of $Y$. With this in mind, the proof of this lemma is identical with Lemma 24.2.2 in the book, which treats the case when $X_{+}$ is a compact manifold with boundary.
\end{proof}

Now we start discussing the regularity of the moduli spaces. Recall that for any point $[\mathfrak{c}]\in \mathcal{B}^{\sigma}_{k-1/2}(Y)$, we have a decomposition
$$
T_{[\mathfrak{c}]}\mathcal{B}^{\sigma}_{k-1/2}(Y)\cong \mathcal{K}^{+}_{\mathfrak{c}}\oplus \mathcal{K}^{-}_{\mathfrak{c}}
$$
given by the spectral decomposition of the Hessian operator $\text{Hess}^{\sigma}_{\mathfrak{q}}(\mathfrak{c})$ (see Page 313 of the book).
\begin{lem}\label{APS for moduli space}
For any $([\gamma_{1}],[\gamma_{2}])\in \operatorname{Fib}(R_{+},R_{-})$. Let $[\mathfrak{c}]$ be the common restriction of $[\gamma_{j}]\ (j=1,2)$ on the boundary $Y$.  Denote by $\pi$ the projection from $T_{[c]}\mathcal{B}^{\sigma}_{k-1/2}(Y)$ to $\mathcal{K}^{-}_{\mathfrak{c}}$ with kernel $\mathcal{K}^{+}_{\mathfrak{c}}$. Then we have the following results.
\begin{enumerate}[(i)]
\item The linear operators $$(1-\pi)\circ\mathcal{D}R_{+}:T_{[\gamma_{1}]}\mathcal{M}([\mathfrak{b}],Z)\rightarrow \mathcal{K}_{\mathfrak{c}}^{+} \text{ and }\pi\circ\mathcal{D}R_{+}:T_{[\gamma_{1}]}\mathcal{M}([\mathfrak{b}],Z)\rightarrow \mathcal{K}_{\mathfrak{c}}^{-}$$
are respectively compact and Fredholm.
\item The linear operators $$(1-\pi)\circ\mathcal{D}R_{-}:T_{[\gamma_{2}]}\mathcal{M}(X_{+})\rightarrow \mathcal{K}_{\mathfrak{c}}^{+} \text{ and }\pi\circ\mathcal{D}R_{-}:T_{[\gamma_{2}]}\mathcal{M}(X_{+})\rightarrow \mathcal{K}_{\mathfrak{c}}^{-}$$
are respectively Fredholm and compact.
\item The linear operator $$\mathcal{D}R_{+}+\mathcal{D}R_{-}:T_{[\gamma_{1}]}\mathcal{M}([\mathfrak{b}],Z)\oplus T_{[\gamma_{2}]}\mathcal{M}(X_{+})\rightarrow T_{[\mathfrak{c}]}\mathcal{B}^{\sigma}_{k-1/2}(Y)$$
is Fredholm.
\end{enumerate}
\end{lem}
\begin{proof}
(i) was proved in Theorem 17.3.2 of the book. We sketch it here: Proposition 17.2.5 of the book states the following: Let $D,r,\Pi$ be as defined in Subsection 2.3, but over the cylindrical-end manifold $Z$ instead of the end-periodic manifold $X_{+}$ (see (\ref{equation: defition of D})). Then $(1-\Pi)\circ r$ and $\Pi\circ r$, restricted to the kernel of $D$, are respectively compact and Fredholm. Note that $\ker D$ and $\Pi$ are ``extended version'' of $T_{[\gamma_{1}]}\mathcal{M}([\mathfrak{b}],Z)$ and $\pi$, respectively. Therefore, (i) is a directly consequence of this proposition.

To prove (ii), we use Proposition \ref{APS} we proved, in place of Proposition 17.2.5 in the book. Then all the arguments there can be repeated with no essential change. (iii) is directly implied by (i) and (ii).
\end{proof}

The following definition is parallel to Definition 24.4.2 of the book.

\begin{defi}
Let $[\gamma]\in \mathcal{M}([\mathfrak{b}],Z_{+})$. If $[\gamma]$ is irreducible, we say the moduli space $\mathcal{M}([\mathfrak{b}],Z_{+})$ is regular at $[\gamma]$ if the maps of Hilbert manifolds
$$ R_{+}:\mathcal{M}([\mathfrak{b}],Z)\rightarrow  \mathcal{B}^{\sigma}_{k-1/2}(Y)\text{ and } R_{-}:\mathcal{M}(X_{+})\rightarrow  \mathcal{B}^{\sigma}_{k-1/2}(Y)$$are transverse at $[\gamma]$. If $[\gamma]$ is reducible, we say the moduli space $\mathcal{M}(Z_{:};[c])$ is regular at $[\gamma]$ if the maps of Hilbert manifolds
$$ R_{+}:\mathcal{M}^{\operatorname{red}}([\mathfrak{b}],Z)\rightarrow  \partial\mathcal{B}^{\sigma}_{k-1/2}(Y)\text{ and } R_{-}:\mathcal{M}^{\operatorname{red}}(X_{+})\rightarrow  \partial\mathcal{B}^{\sigma}_{k-1/2}(Y)$$are transverse at $\rho([\gamma])$ (see (\ref{restriction})). We say the moduli space is regular if it is regular at all points.
\end{defi}
Recall that the perturbation $\mathfrak{p}$ on $Z_{+}$ is determined a pair of $3$-dimensional perturbations $(\mathfrak{q},\mathfrak{p}_{0})$ (see (\ref{mixed perturbation})), where $\mathfrak{q}$ is a nice perturbation that is fixed throughout our argument (see Assumption \ref{3 dimensional perturbation}). We want to obtain the transversality condition by varying the second perturbation $\mathfrak{p}_{0}$. To do this, let $\mathcal{P}(Y)$ be the large Banach space of $3$-dimensional tame perturbations provided by Theorem 11.6.1 of the book. We have the following result.
\begin{pro}\label{transversality}
There exists a residual subset $U_{1}$ of  $\mathcal{P}(Y)$ such that for any $\mathfrak{p}_{0}\in U_{1}$, the moduli space $\mathcal{M}([\mathfrak{b}],Z_{+})$ corresponding to  $(\mathfrak{q},\mathfrak{p}_{0})$ is regular for any critical point $[\mathfrak{b}]\in\mathfrak{C}$.
\end{pro}
\begin{proof}
The proof follows the standard argument as in the proof of Proposition 24.4.7 of the book: We consider parametrized moduli space
$$\mathfrak{M}(X_{+})\subset \mathcal{B}^{\sigma}_{k,\delta}(X_{+})\times \mathcal{P}(Y)$$
$$
\mathfrak{M}(X_{+})=\{(A,s,\phi,\mathfrak{p}_{0})|\,\mathfrak{F}^{\sigma}_{\mathfrak{p}}=0\}/\mathcal{G}_{k+1,\delta}(X_{+}).
$$
For any $[\mathfrak{b}]\in \mathfrak{C}$, we can prove that the map
$$
R_{+}\times \mathfrak{R}_{-}: \mathcal{M}([\mathfrak{b}],Z)\times  \mathfrak{M}(X_{+})\rightarrow \mathcal{B}_{k-1/2}^{\sigma}(Y)\times  \mathcal{B}_{k-1/2}^{\sigma}(Y)
$$
is transverse to the diagonal by the same argument as the book. Here the map $\mathfrak{R}_{-}$ is defined similarly with $R_{-}$ (but with larger domain). Now we apply the Sard-Smale lemma (Lemma 12.5.1 of the book) to finish the proof. We note that Lemma \ref{APS for moduli space} (iii) is used essentially in this last step.
\end{proof}

The proof of the following proposition is by standard transversility argument and we omit it. (Compare Proposition 24.4.3 of the book.)
\begin{pro}\label{moduli space is a manifold}
Suppose the moduli space $\mathcal{M}([\mathfrak{b}],Z_{+})$ is regular and non-empty. Then the moduli space is
\begin{enumerate}[(i)]
\item a smooth manifold consisting only of irreducibles, if $[\mathfrak{b}]$ is irreducible;
\item a smooth manifold consisting only of reducibles, if $[\mathfrak{b}]$ is reducible and boundary-stable;
\item a smooth manifold with (possibly empty) boundary, if $[\mathfrak{b}]$ is reducible and boundary-unstable.
\end{enumerate}
In the last case, the boundary consists of the reducible elements of the moduli space (i.e., we have $\partial\mathcal{M}([\mathfrak{b}],Z_{+})=\mathcal{M}^{\text{red}}([\mathfrak{b}],Z_{+})$).
\end{pro}
Recall that we associated a rational number  $\operatorname{gr}^{\mathds{Q}}([\mathfrak{b}])$ to each critical point $[\mathfrak{b}]$. We have the following result.
\begin{pro}\label{expected dimension}
Suppose the moduli space $\mathcal{M}([\mathfrak{b}],Z_{+})$ is regular. Then the moduli space is
\begin{enumerate}[(i)]
\item the empty set, if $\operatorname{gr}^{\mathds{Q}}([\mathfrak{b}])+2w(X,g_{X},0)<0$;
\item a manifold with dimension  $\operatorname{gr}^{\mathds{Q}}([\mathfrak{b}])+2w(X,g_{X},0)$, if  $\operatorname{gr}^{\mathds{Q}}([\mathfrak{b}])+2w(X,g_{X},0)\geq 0$.
\end{enumerate}
\end{pro}
\begin{proof}
We just need to show that the expected dimension of $\mathcal{M}([\mathfrak{b}],Z_{+})$ (which we denote by $\operatorname{gr}(Z_{+};[\mathfrak{b}])$) can be expressed as
 $$\operatorname{gr}(Z_{+};[\mathfrak{b}])=\operatorname{gr}^{\mathds{Q}}([\mathfrak{b}])+2w(X,g_{X},0).$$ This can be done by direct computation. But we follow an alternative argument here. Recall that $M$ is a spin manifold with bounded by $(Y,\mathfrak{s})$ with $b_{1}(M)=0$. We let $M^{*}=M\cup _{Y}([0,+\infty)\times Y)$.  As discussed before, the $M$-grading of $[\mathfrak{b}]$ (which we denoted by $\operatorname{gr}(M;[\mathfrak{b}])$) equals the expected dimension of the moduli space consisting of solutions on $M^{*}$ that are asymptotic to $[\mathfrak{b}]$. Since one can deform the linearized Seiberg-Witten equations over the manifold $M^{*}\cup Z_{+}$ first to the corresponding equations over the manifold $$M\cup_{Y}([-T,T]\times Y)\cup_{Y}X_{+}\text{ with }T\gg 0$$
 and then to the manifold $M_{+}$. We see that the grading is additive in the sense that the sum $\operatorname{gr}(M;[\mathfrak{b}])+\operatorname{gr}(Z_{+};[\mathfrak{b}])$ equals the expected dimension $\mathcal{M}(M_{+})$, the moduli space consisting of gauge equivalent classes of solutions over $M_{+}$ that decay exponentially on the periodic end. The linear operator that determines the local structure of $\mathcal{M}(M_{+})$ is a compact perturbation of the operator
 $$
 L^{2}_{k,\delta}(M_{+};iT^{*}M_{+}\oplus S^{+})\rightarrow L^{2}_{k-1,\delta}(M_{+};i\mathds{R}\oplus i\wedge^{2}_{+}T^{*}M_{+}\oplus S^{-})
 $$
 $$
 (a,\Phi)\mapsto (d^{*}a,d^{+}a,\slashed{D}_{A_{0}}\Phi).
 $$
 By Lemma \ref{half De rham complex} and Proposition \ref{Dirac operator is Fredholm}, the (real) index of this operator equals
 $$
 -\frac{\text{sign}(M)}{4}+2w(X,g_{X},0)+b^{+}_{2}(M)-1.
 $$
 By (\ref{absolute grading}), this implies
 \begin{equation}\begin{split}\operatorname{gr}(Z_{+};[\mathfrak{b}])&=-\frac{\text{sign}(M)}{4}+2w(X,g_{X},0)+b^{+}_{2}(M)-1-\operatorname{gr}(M;[\mathfrak{b}])\\&=\operatorname{gr}^{\mathds{Q}}([\mathfrak{b}])+2w(X,g_{X},0).\end{split}\end{equation}
 \end{proof}

Recall that we denote the lowest boundary stable reducible critical point by $[\mathfrak{a}_{0}]$. Recall that the absolute grading  $[\mathfrak{a}_{0}]$ equals the height of the nice perturbation $\mathfrak{q}$, which has been chosen to be $-2w(X,g_{X},0)$ (see Assumption \ref{3 dimensional perturbation}). By Proposition \ref{moduli space is a manifold} and Proposition \ref{expected dimension}, for any $\mathfrak{p}\in U_{1}$ (the residue set provided by Lemma \ref{transversality}), the moduli space $\mathcal{M}([\mathfrak{a}_{0}],[Z_{+}])$ consists of discrete elements, all of which are reducible because $[\mathfrak{a}_{0}]$ is boundary stable. The moduli spaces $\mathcal{M}([\mathfrak{a}_{i}],[Z_{+}])\ (i< 0)$ are all empty.
\begin{pro}\label{only one reducible}
There exists an open neighborhood $U_{2}\subset \mathcal{P}(Y)$ of $0$ such that for any $\mathfrak{p}_{0}\in U_{2}$, the moduli space $\mathcal{M}([\mathfrak{a}_{0}],[Z_{+}])$ corresponding to  $(\mathfrak{q},\mathfrak{p}_{0})$ contains a single point.
\end{pro}
\begin{proof}
Since the moduli space only consists of reducibles, we do not need to consider the nice perturbation $\mathfrak{q}$ since it vanishes on the reducibles. Moreover, we can describe the moduli space explicitly: each gauge equivalent class of solutions of the downstairs equation
\begin{equation}\label{downstairs reducible equation}
d^{+}a-\beta_{0}\cdot \rho^{-1}(\mathfrak{\hat{p}}^{0}_{0}(A_{0}+a,0))=0,\ a\in L^{2}_{k+1;\text{loc},\delta}(Z_{+};i\mathds{R})
\end{equation}
contributes a copy of $\mathds{CP}^{n-1}\setminus \mathds{CP}^{n-2}$ in $\mathcal{M}([\mathfrak{a}_{0}],[Z_{+}])$, with $n$ being the index of the Dirac operator $\slashed{D}_{A_{0}+a}$. See page 567 of the book. (Here $\beta_{0}$ is the bump function in (\ref{mixed perturbation}) and $\hat{\mathfrak{p}}^{0}_{0}$ is a component of the $4$-dimensional, downstairs perturbation $\hat{\mathfrak{p}}_{0}$ induced by the $3$-dimensional perturbation $\mathfrak{p}_{0}$.) In our situation, since the dimension of $\mathcal{M}([\mathfrak{a}_{0}],[Z_{+}])$ equals zero, we have $n=1$.
We want to show that when $\mathfrak{p}_{0}$ (hence $\hat{\mathfrak{p}}_{0}^{0}$) is small enough, (\ref{downstairs reducible equation}) has exactly one solution up to gauge equivalence. By the exponential decay result Theorem 13.3.5 of the book (applied to $a|_{Z}$) and Lemma \ref{half De rham complex} (i), we see that each equivalent class contains a unique representative satisfying
$$
\|a\|_{L^{2}_{k;-\delta,\delta}}<\infty,\ d^{*}a=0.
$$
In other words, we just need to prove (\ref{downstairs reducible equation}) has a unique solution satisfying the above gauge fixing condition when the perturbation is small. To do this, we consider the map
$$
\mathfrak{P}:\mathcal{P}(Y)\times L^{2}_{k;-\delta,\delta}(Z_{+};iT^{*}Z_{+})\rightarrow V\oplus L^{2}_{k;-\delta,\delta}(Z_{+}; i\wedge^{2}_{+}T^{*}Z_{+}),
$$
where $V=\{\xi\in L^{2}_{k,-\delta,\delta}(Z_{+};i\mathds{R})|\int_{Z_{+}}\xi d\text{vol}=0\}$, given by
$$
(\mathfrak{p}_{0},a)\mapsto (d^{*}a,d^{+}a-\beta_{0}\cdot \rho^{-1}(\mathfrak{\hat{p}}_{0}^{0}(A_{0}+a,0))).
$$
By Lemma \ref{half De rham complex} (iii), the restriction of $\mathfrak{P}$ to $\{0\}\times L^{2}_{k;-\delta,\delta}(Z_{+};iT^{*}Z_{+})$ is a (linear) isomorphism.  Therefore, by the implicit function theorem, there exists a neighborhood $U$ of $0\in  L^{2}_{k;-\delta,\delta}(Z_{+};iT^{*}Z_{+})$ and a neighborhood $U'$ of $0\in \mathcal{P}(Y)$ with the property that: for any $\mathfrak{p}_{0}\in U'$, there exists a unique solution of the equation $\mathfrak{P}(\mathfrak{p}_{0},a)=0$ with $a\in U$. Now we claim that we can find another neighborhood $U''$ of $0\in \mathcal{P}(Y)$
such that for any $\mathfrak{p}_{0}\in U''$,  $\mathfrak{P}(\mathfrak{p}_{0},a)=0$ implies $a\in U$. This will finish the proof because we can set $U_{2}=U'\cap U''$. Now we prove our claim by contradiction. Suppose there exist $\mathfrak{p}_{0,n}\rightarrow 0$ and $a_{n}\notin U$ such that $\mathfrak{P}(\mathfrak{p}_{0,n},a_{n})=0$ for each $n$. Integrating by part on $(-\infty,-0]\times Y$ and $X_{+}\setminus [3,+\infty)$ respectively, we see that
$$
\operatorname{CSD}((A_{0}+a_{n})|_{Y\times\{0\}},0)<0,\ \operatorname{CSD}((A_{0}+a_{n})|_{Y\times\{3\}},0)>0.
$$
Using these energy estimates, one can easily adapt the proof of Theorem 10.7.1 of the book (from the single perturbation case to the case of a convergent sequence of perturbations) and prove that: after passing to a subsequence and applying suitable gauge transformations $u_{n}$, the sequence $
u_{n}\cdot((A_{0}+a_{n})|_{Y\times[1,2]},0)
$ converges smoothly. Notice that the gauge invariant term $\beta_{0}\cdot \rho^{-1}(\mathfrak{\hat{p}}_{0,n}^{0}(A_{0}+a_{n},0))$ is supported on $Y\times [1,2]$ and only depends on  $(A_{0}+a_{n})|_{Y\times[1,2]}$ (because the bump function $\beta_{0}$ is supported on $[1,2]\times Y$). We see that $$\|(d^{*}a_{n},d^{+}a)\|_{L^{2}_{k-1;-\delta,\delta}} =\|\beta_{0}\cdot \rho^{-1}(\mathfrak{\hat{p}}_{0,n}^{0}(A_{0}+a_{n},0))\|_{L^{2}_{k-1;-\delta,\delta}}\rightarrow 0\text{ as }n\rightarrow\infty$$ since $\mathfrak{p}_{0,n}\rightarrow 0$. By Lemma \ref{half De rham complex} (iii) again, we get $\|a_{n}\|_{L^{2}_{k;-\delta,\delta}}\rightarrow 0$. This contradicts with our assumption $a_{n}\notin U$ and completes our proof.\end{proof}
\begin{assum}\label{4 dim perturbation}
From now on, we fix a choice of perturbation $\mathfrak{p}_{0}\in U_{1}\cap U_{2}$, where $U_{1}$, $U_{2}$ are subsets of $\mathcal{P}(Y)$ provided by Proposition \ref{transversality} and Proposition \ref{only one reducible} respectively.
\end{assum}
As in the cylindrical case, a sequence of $Z_{+}$-trajectories (even with unifomly bounded energy) can converge to a broken trajectory. For this reason, we have to introduce the moduli space of broken trajectories before discussing the compactness property. Although our construction can be generalized to moduli space of higher dimension without essential difficulty, we focus on $1$-dimensional moduli spaces for simplicity. This will be enough for our application.

We start with recalling the ``$\tau$-module'' for blow up. (See Section 6.3 of the book for details.) Let $I\subset \mathds{R}$ be an interval. Denote the product manifold $I\times Y$ by $Z_{I}$. There are two cases:
\begin{itemize}
\item Suppose $I$ is compact, we define the configuration space \begin{equation}\label{tau-module}\begin{split}
\mathcal{C}^{\tau}_{k}(Z_{I})=&\{(A_{0}+a,s,\phi)|(a,\phi)\in L^{2}_{k}(Z_{I};iT^{*}Z_{I}\oplus S^{+}),\ s\in L^{2}_{k}(I;\mathds{R})\\& \text{satisfies }s(t)\geq 0,\ \|\phi|_{Y\times \{t\}}\|_{L^{2}(Y)}=1\text{ for any }t\in I \}\end{split}\end{equation}
The gauge group $
\mathcal{G}_{k+1}(Z_{I})$ acts on $\mathcal{C}^{\tau}_{k}(Z_{I})$ as
$$
u\cdot(A_{0}+a,s,\phi)= (A_{0}+a-u^{-1}du,s, u\phi).
$$ We denote the quotient space by $\mathcal{B}^{\tau}_{k}(Z_{I})$.
\item Suppose $I$ is non-compact, we define $\mathcal{C}^{\tau}_{k,\text{loc}}(Z_{I})$ by replacing  $L^{2}_{k}$ with $L^{2}_{k,\text{loc}}$ in (\ref{tau-module}). We let $\mathcal{B}^{\tau}_{k,\text{loc}}(Z_{I})=\mathcal{C}^{\tau}_{k,\text{loc}}(Z_{I})/\mathcal{G}_{k+1,\text{loc}}(Z_{I})$.
\end{itemize}
In both cases, we impose the quotient topology on the quotient configuration space. For any $[\mathfrak{b}],[b']\in \mathfrak{C}$, the moduli space $\mathcal{M}([\mathfrak{b}],[b'])$ is a subset of $\mathcal{B}^{\tau}_{k,\text{loc}}(Z_{(-\infty,+\infty)})$ and consists of the non-constant Seiberg-Witten trajectories going from $[\mathfrak{b}]$ to $[b']$. We let $\breve{\mathcal{M}}([\mathfrak{b}],[b'])=\mathcal{M}([\mathfrak{b}],[b'])/\mathds{R}$, where $\mathds{R}$ acts as translation (reparametrization).

Now we define the moduli space of broken trajectories. Let $[\mathfrak{b}_{0}]$ be a critical point with $\operatorname{gr}^{\mathds{Q}}([b_{0}])=-2w(X,g_{X},0)+1$. By our assumption about $\operatorname{ht}(\mathfrak{q})$, $[\mathfrak{b}_{0}]$ must be irreducible. We consider the set
$$
\mathcal{M}^{+}([\mathfrak{b}_{0}],Z_{+})=\mathcal{M}([\mathfrak{b}_{0}],Z_{+})\cup(\mathop{\cup}\limits_{[\mathfrak{b}]\in \mathfrak{C}}\breve{\mathcal{M}}([\mathfrak{b}_{0}],[\mathfrak{b}])\times \mathcal{M}([\mathfrak{b}],Z_{+})).
$$
By our regularity assumption,  $\mathcal{M}([\mathfrak{b}_{0}],Z_{+})$ is a $1$-dimensional manifold (without boundary). The set $\breve{\mathcal{M}}([\mathfrak{b}_{0}],[\mathfrak{b}])\times \mathcal{M}([\mathfrak{b}],Z_{+})$ is nonempty only if $\operatorname{gr}^{\mathds{Q}}([\mathfrak{b}])=2w(X,g_{X},0)$, in which case it is a discrete set.

To define the topology on $\mathcal{M}^{+}([\mathfrak{b}_{0}],Z_{+})$, we need to specify a neighborhood base for each point. For those points in $\mathcal{M}([\mathfrak{b}_{0}],Z_{+})$, we just use their neighborhood basis inside $\mathcal{M}([\mathfrak{b}_{0}],Z_{+})$. For a broken trajectory $([\gamma_{-1}],[\gamma_{0}])\in \breve{\mathcal{M}}([\mathfrak{b}_{0}],[\mathfrak{b}])\times \mathcal{M}([\mathfrak{b}],Z_{+})$, we let $[\gamma_{-1}]$ be represented by a parametrized trajectory
$$
\gamma_{-1}\in \mathcal{M}([\mathfrak{b}_{0}],[\mathfrak{b}]).
$$
Let $U_{0}$ be a neighborhood of $[\gamma_{0}]$ in $\mathcal{B}^{\sigma}_{k,\operatorname{loc},\delta}(Z_{+})$ and let $I\subset\mathds{R}$ be a compact interval and $U_{-1}\subset  \mathcal{B}^{\tau}_{k}(Z_{I})$ be a neighborhood of $[\gamma_{-1}|_{I}]$. For any $T\in \mathds{R}_{>0}$ with the property that $I-T$ (the translation of $I$ by $-T$) is contained in $\mathds{R}_{\leq 0}$, we define $\Omega(U_{-1},U_{0},T)$ to be the subset of  $\mathcal{M}^{+}([\mathfrak{b}_{0}],Z_{+})$ consisting of the broken $Z_{+}$-trajectory $([\gamma_{-1}],[\gamma_{0}])$ and (unbroken) $Z_{+}$-trajectories $[\gamma] \in \mathcal{M}([\mathfrak{b}_{0}],Z_{+})$ satisfying the following conditions:
\begin{itemize}
\item $[\gamma]\in U_{0};$
\item There exists  $T_{-1}>T$ such that $[\tau_{T_{-1}}(\gamma|_{I-T_{-1}})]\in U_{-1}$, where $\tau_{T_{-1}}(\gamma|_{I-T_{-1}})$ denotes the translation of $\gamma|_{I-T_{-1}}$ by $T_{-1}$ (in the positive direction).
\end{itemize}
We put the sets of the form $\Omega(U_{-1},U_{0},T)$ form a neighborhood basis for $([\gamma_{-1}],[\gamma_{0}])$.
With the topology on $\mathcal{M}^{+}([\mathfrak{b}_{0}],Z_{+})$ defined, we have the following gluing theorem, whose proof is a word by word translation from the proof of Theorem 24.7.2 in the book and we omit.

\begin{thm}\label{gluing}
For each broken $Z_{+}$-trajectory $([\gamma_{-1}],[\gamma_{0}])\in \mathcal{M}^{+}([\mathfrak{b}_{0}],Z_{+})$, we can find its open neighborhood $U$ with $U\setminus ([\gamma_{-1}],[\gamma_{0}])\subset\mathcal{M}([\mathfrak{b}_{0}],Z_{+})$ and a homeomorphism $f: (0,1]\times ([\gamma_{-1}],[\gamma_{0}]) \rightarrow U$ that sends $\{1\}\times ([\gamma_{-1}],[\gamma_{0}]) $ to $([\gamma_{-1}],[\gamma_{0}])\in U$.
\begin{rmk}
Theorem 24.7.2 in the book actually contains the two parts: the boundary obstructed case and the boundary unobstructed case. The second case is much easier than the first case.  Theorem \ref{gluing} here corresponds to the second case with the additional assumption that the moduli space is $1$-dimensional and the boundary of the $4$-manifold is connected. This further simplifies the statement of the result.
\end{rmk}
\end{thm}

Now we consider the orientation of the moduli spaces. As mentioned in Subsection 2.2, a choice of $\chi([\mathfrak{b}])$ in the orientation set $\Lambda([\mathfrak{b}])$ for each $[\mathfrak{b}]$ canonically induces an orientation of the moduli space $\breve{\mathcal{M}}([\mathfrak{b}],[\mathfrak{b}'])$ for any critical points $[\mathfrak{b}],[\mathfrak{b}']$. It was also proved in Threorem 24.8.3 of the book that a choice of $\chi([\mathfrak{b}])$ and a homology orientation of $M$ determines an orientation of $\mathcal{M}(M^{*},[\mathfrak{b}])$ (the moduli space of gauge equivalent classes consisting of solutions on $M^{*}=M\cup_{Y}[0,+\infty)\times Y$ that are asymptotic to $[\mathfrak{b}]$). By replacing the compact manifold $M$ with the non-compact manifold $X_{+}$ and working with the weighted Sobolev spaces instead of the unweighted ones, one can repeat the argument there and prove the following similar result. Note that we do not need any homology orientation of $X_{+}$. This is essentially because of Lemma \ref{half De rham complex} (iv) (compare Lemma 24.8.1 of the book). An alternative viewpoint is that $H^{1}(X_{+};\mathds{R})=H^{2}(X_{+};\mathds{R})=0$.
\begin{thm}\label{orientation}
A choice of $\{\chi([\mathfrak{b}])|\,[\mathfrak{b}]\in\mathfrak{C}\}$ canonically induces an orientation on the moduli space $\mathcal{M}([\mathfrak{b}],Z_{+})$ for any critical point $[\mathfrak{b}]$. These orientations are compatible with the gluing map in the following sense: the map $f$ provided by Theorem \ref{gluing} is orientation preserving when restricted to $(0,1)\times ([\gamma_{-1}],[\gamma_{0}])$, if we orient the moduli spaces $\breve{\mathcal{M}}([\mathfrak{b}_{0}],[\mathfrak{b}])$, $\mathcal{M}([\mathfrak{b}],Z_{+}))$ and $\mathcal{M}([\mathfrak{b}_{0}],Z_{+}))$ by the same choice $\{\chi([\mathfrak{b}])|\,[\mathfrak{b}]\in\mathfrak{C}\}$ and use the positive orientation on the interval $(0,1)$.
\end{thm}
\section{Compactness}
In the current and the next section, we impose the following assumption:
\begin{assum}
The scalar curvature $\operatorname{scal}$ of $g_{X}$ to be everywhere positive. In other words, we have
$$
s_{0}=\mathop{\operatorname{inf}}\limits_{x\in X}\operatorname{scal}(x)>0.
$$
\end{assum}
This assumption implies that the restriction of $g_{Z_{+}}$ on $\mathop{\cup}\limits_{n\geq 1}W_{n}$, which is a lift of $g_{X}$, has uniformly positive scalar curvature. Under this assumption, we will prove the following compactness theorem:
\begin{thm}\label{compactness}
For any $[\mathfrak{b}_{0}]\in \mathfrak{C}$ with $\operatorname{gr}^{\mathds{Q}}([\mathfrak{b}_{0}])=-2w(X,g_{X},0)+1$, the moduli space $\mathcal{M}^{+}([\mathfrak{b}_{0}],Z_{+})$ is compact.
\end{thm}(Again, the result can be generalized to arbitrary $[\mathfrak{b}_{0}]$. But we focus on the current case because that is all we need.)
\subsection{The topological energy $\mathcal{E}^{\text{top}}$ and the quantity $\Lambda_{\mathfrak{q}}$} We start with some standard definitions in the book, which will be useful in our proof of compactness theorem. Let $\hat{X}$ be a general spin$^{c}$ $4$-manifold and $(A,\Phi)$ be a point of the configuration space (i.e., $A$ is a spin$^{c}$ connection and $\Phi$ is a positive spinor over $\hat{X}$). Its topological energy is defined as
\begin{equation}\label{topolocial energy}
\mathcal{E}^{\text{top}}(A,\Phi)=\frac{1}{4}\int_{\hat{X}}F_{A^{t}}\wedge F_{A^{t}}-\int_{\partial\hat{X}}\langle \Phi|_{\partial\hat{X}},\slashed{D}_{B}(\Phi|_{\partial\hat{X}})\rangle d\text{vol}+\int_{\partial\hat{X}}(H/2)|\Phi|^{2}d\text{vol}
\end{equation}
where $B=A|_{\partial\hat{X}}$ and $H$ denotes the mean curvature of the boundary, which will be vanishing if we use the product metric near the boundary.
Note that in our situation, the integrals in (\ref{topolocial energy}) are always convergent (even if $\hat{X}$ is not compact) because $F_{A^{t}}$ decays exponentially over the end of $\hat{X}$.
 We also talk about the topological energy of a point in the blown-up configuration space (i.e., a triple $(A,e,\phi)$ with $e\geq 0$ and $|\phi|_{L^{2}}=1$). In this case, we define $\mathcal{E}^{\text{top}}(A,s,\phi)$ to be  $\mathcal{E}^{\text{top}}(\boldsymbol\pi(A,s,\phi))$ where  $$\boldsymbol{\pi}(A,s,\phi)=(A,s\phi)$$
as before. Since the topological energy is invariant under gauge transformation, it also makes sense to talk about the topological energy of a gauge equivalent class.

Now we return to our end-periodic manifold $X_{+}$. Recall that $\mathfrak{q}$ is a nice perturbation (of height $-2w(X,g_{X},0)$). After choosing a gauge invariant function
\begin{equation}\label{perturbation}
v:\mathcal{C}_{k-1/2}(Y)\rightarrow \mathds{R}.
\end{equation} whose formal gradient equals $\mathfrak{q}$. We can define the perturbed topological energy of a point $\gamma\in \mathcal{C}^{\sigma}_{k,\text{loc}}(X_{+})$  as
$$
\mathcal{E}^{\text{top}}_{\mathfrak{q}}(\gamma)=\mathcal{E}^{\text{top}}(\gamma)-2v(\boldsymbol\pi(\gamma)|_{Y}).$$

Let $\epsilon$ be a number lying in $(0,\frac{1}{2})$. We consider two other manifolds:
$$
X^{'}_{+}=X_{+}\setminus ([0,2\epsilon)\times Y),\
X^{''}_{+}=X_{+}\setminus ([0,\epsilon)\times Y)
$$ We can define the blown-up configuration space $\mathcal{C}^{\sigma}_{k,\delta}(X_{+}^{'})$ similarly as $\mathcal{C}^{\sigma}_{k,\delta}(X_{+})$. There is a partially defined restriction map $$\mathcal{C}^{\sigma}_{k,\delta}(X_{+})\dashrightarrow \mathcal{C}^{\sigma}_{k,\delta}(X^{'}_{+})$$
$$
(A,s,\phi)\rightarrow (A|_{X^{'}_{+}},s\|\phi|_{X^{'}_{+}}\|_{L^{2}},\frac{\phi|_{X^{'}_{+}}}{\|\phi|_{X^{'}_{+}}\|_{L^{2}}})
$$
whose domain contains triples $(A,s,\phi)$ with $\phi|_{X^{'}_{+}}\neq 0.$ We denote by $(A,s,\phi)|_{X'_{+}}$ the image of $(A,s,\phi)$ under this map. Under the assumption $\phi|_{Y\times \{\epsilon\}}\neq 0$, we can define $(A,s,\phi)|_{Y\times \{\epsilon\}}\in \mathcal{C}^{\sigma}_{k-1/2}(Y)$ in a similar vein. Note that since we are considering the solution of the perturbed Seiberg-Witten equations, these conditions are always satisfied by the unique continuation theorem.

Other than the (perturbed) topological energy, there is another quantity that will be useful when dealing with the blown-up configuration space. Let $(B,r,\psi)$ be a point of $C^{\sigma}_{k-1/2}(Y)$. We define the quantity  $$\Lambda_{\mathfrak{q}}(B,r,\psi)=\operatorname{Re}\langle \psi,\slashed{D}_{B}\psi+\tilde{\mathfrak{q}}^{1}(B,r,\psi)\rangle_{L^{2}}$$
where $\tilde{\mathfrak{q}}^{1}(B,r,\psi)$ is defined as (see Remark \ref{component of perturbation})
$$
\tilde{\mathfrak{q}}^{1}(B,r,\psi)=\int_{0}^{1}\mathcal{D}_{(B,sr\psi)}\mathfrak{q}^{1}(0,\psi)ds.
$$
(Recall that $\mathfrak{q}^{1}$ denotes the spinor component of the perturbation $\mathfrak{q}$.)

\subsection{Exponential decay}
In this subsection, we prove exponential decay results for solutions on the manifold $X_{+}=W_{0}\cup_{Y}W_{1}\cup_{Y}...$. To simplify the notation, we denote by $W_{n,n'}$ the manifold
$$
W_{n}\cup_{Y}W_{n+1}\cup_{Y}...\cup_{Y}W_{n'}\subset X_{+}
,$$
and write $\|\cdot\|_{L^{2}_{j}(W_{n,n'})}$ for the $L^{2}_{j}$ norm of the restriction to $W_{n,n'}$. We will use similar notation for other manifolds.

Let us start with the following lemma, which was communicated to the author by Clifford Taubes.

\begin{lem}\label{exp decay}
There exists uniform constants $C,\delta_{3}>0$ with the following significance: for any $\delta\in (0,\delta_{3})$ and any solution $\gamma=(A,s,\phi)\in \mathcal{C}^{\sigma}_{k,\delta}(X_{+})$ of the equation  $\mathfrak{F}^{\sigma}_{\mathfrak{p}}(\gamma)=0$, we have
$$
\|\phi\|_{L^{2}(W_{n})}\leq Ce^{-\delta_{3}n},\ \ \forall n\geq 0.
$$
\end{lem}
\begin{proof}
We first consider $W_{n}$ for $n\geq 1$. Over these manifolds, the perturbation $\mathfrak{p}$ equals $0$ and hence we have
\begin{equation}\begin{split}
\rho(F^{+}_{A^{t}})-2s^{2}(\phi\phi^{*})_{0}=0\\\slashed{D}^{+}_{A}\phi=0.\end{split}
\end{equation}
We choose an integer $N$ large enough such that there exists a bump function
$$
\tau:W_{1,3N}\rightarrow[0,1]
$$
with the following properties: i) $\tau$ is supported on $W_{2,3N-1}$; ii) $\tau$ equals $1$ when restricted to $W_{N+1,2N}$; iii) $|d\tau(x)|^{2}<s_{0}/10$ for any $x\in W_{1,3N}$. By the covering tranformations, $\tau$ induces a bump function on
$$\tau_{m}:W_{mN+1,(m+3)N}\rightarrow[0,1]\ \ (m\geq 0).$$
Let $\phi_{m}=\tau_{m}\cdot \phi$. Then $
\slashed{D}^{+}_{A}\phi_{m}=\rho(d\tau_{m})\phi
$.
Notice that $d\tau_{m}$ is supported on $$W_{mN+1,mN+N}\cup W_{mN+2N+1,mN+3N}.$$
Therefore, for any $m\geq 1$, we have
\begin{equation}\label{31}
\begin{split}\|\slashed{D}^{+}_{A}\phi_{m}\|^{2}_{L^{2}(X_{+})}&\leq \|d\tau_{m}\|^{2}_{C^{0}}\cdot (\|\phi\|^{2}_{L^{2}(W_{mN+1,mN+N})}+\|\phi\|^{2}_{L^{2}(W_{mN+2N+1,mN+3N})})\\&\leq \frac{s_{0}}{10}(\|\phi_{m-1}\|^{2}_{L^{2}(X_{+})}+\|\phi_{m+1}\|^{2}_{L^{2}(X_{+})}).\end{split}\end{equation}
On the other hand, since $\phi_{m}$ is supported on $W_{mN+1,mN+3N-1}$, by the Weitzenb\"ock formula, we have
\begin{equation}\label{32}\begin{split}
\|\slashed{D}_{A}\phi_{m}\|^{2}_{L^{2}(X_{+})}&=\int_{X_{+}}\langle \slashed{D}_{A}^{-}\slashed{D}^{+}_{A}\phi_{m},\phi_{m}\rangle\\&=\int_{X_{+}}\langle (\nabla_{A}^{*}\nabla_{A}+\frac{1}{2}\rho(F_{A^{t}}^{+})+\frac{\operatorname{scal}(x)}{4})\phi_{m},\phi_{m}\rangle \\&= \int_{X_{+}}\langle \nabla_{A}^{*}\nabla_{A}\phi_{m},\phi_{m}\rangle +2s^{2}\int_{X_{+}}\tau_{m}^{2}\langle(\phi\phi^{*})_{0}\phi,\phi\rangle+\int_{X_{+}}\frac{\operatorname{scal(x)}}{4}\langle\phi_{m},\phi_{m}\rangle\\
&\geq\frac{s_{0}}{4}\|\phi_{m}\|_{L^{2}(X_{+})}^{2}
\end{split}\end{equation}
Let $a_{m}=\|\phi_{m}\|^{2}_{L^{2}(X_{+})}$. By (\ref{31}) and (\ref{32}), we have
$$
5a_{m}\leq 2(a_{m-1}+a_{m+1}),\ \forall m\geq 1,
$$
which is the same as
$$
2(2a_{m}-a_{m-1})\leq 2a_{m+1}-a_{m},\ \forall m\geq 1.
$$
Notice that $a_{m}\leq\|\phi\|^{2}_{L^{2}(X_{+})}=1$ for any $m$. We must have $2a_{m}-a_{m-1}\leq 0$ for any $m\geq 1$ because otherwise $2a_{m}-a_{m-1}$ (and hence $a_{m}$) will increase exponentially. Therefore, we get
$a_{m}\leq 2^{-m}a_{0}\leq 2^{-m}$ for any $m\geq 0$. For any $n\geq N$, we have
$$
\|\phi\|^{2}_{L^{2}(W_{n})}\leq a_{\lfloor n/N \rfloor}\leq 2\cdot (2^{1/N})^{-n}.
$$
Since $\|\phi\|^{2}_{L^{2}(W_{n})}\leq \|\phi\|^{2}_{L^{2}(X_{+})}=1$ for any $n\geq 0$. We can set $C=2^{1/2}$ and $\delta_{3}=(\ln 2)/2N$.\end{proof}

Next, we prove an exponential decay result for energy of solutions. We start with the following lemma. Recall that $A_{0}$ is the flat base connection on $X_{+}$.
\begin{lem}\label{energy controls norm}
There exists uniform constants $\epsilon_{0},D>0$ with the following significance: for any $m\geq 1$ and any solution $(A,\Phi)$ of the unperturbed Seiberg-Witten solution on $W_{m,m+6}$ satisfying the following conditions:
\begin{enumerate}
\item $d^{*}(A-A_{0})=0;$
\item $\langle A-A_{0},\vec{v}_{m}\rangle=0$, where $\vec{v}_{m}$ is the outward normal vector field on $\partial W_{m,m+6}$;
\item $\mathcal{E}^{\text{top}}(A,\Phi)\leq \epsilon_{0}$,
\end{enumerate}
One has the inequality $$\|(A-A_{0},\Phi)\|_{L^{2}_{k+1}(W_{m+2,m+4})}\leq D\cdot \sqrt{\mathcal{E}^{\text{top}}(A,\Phi)}.$$
\end{lem}  
\begin{proof}
This is a modification of an argument in the proof Theorem 5.1.1 of the book. For completeness, we give a detailed argument here. To simplify the notation, we use $E$ to denote $\mathcal{E}^{\text{top}}(A,\Phi)$ and use $a$ to denote $A-A_{0}$. In the proof, we will write $C_{j}$ for certain universal constants independent of $(A,\Phi)$ and $\epsilon_{0}$. Without loss of generality, we can set $m=1$. Since $(A,\Phi)$ is a solution, the topolocical energy $E$ equals the analytical energy:
$$
\mathcal{E}^{an}(A,\Phi)=\frac{1}{4}\int_{W_{1,7}}(4|da|^{2}+4|\nabla_{A}\Phi|^{2}+|\Phi|^{4}+\operatorname{scal}(x)|\Phi|^{2})
$$ 
(See formula (4.16) of the book.) 
Because  $\operatorname{scal}(x)\geq s_{0}>0$, we get
\begin{enumerate}[(i)]
\item $\|\Phi\|_{L^{4}(W_{1,7})}\leq C_{1} E^{\frac{1}{4}}$;
\item $\|\Phi\|_{L^{2}(W_{1,7})}\leq C_{1} E^{\frac{1}{2}}$;
\item $\|da\|_{L^{2}(W_{1,7})}\leq C_{1} E^{\frac{1}{2}}$;
\item $\|\nabla_{A_{0}}\Phi+\rho(a)\Phi\|_{L^{2}(W_{1,7})}\leq C_{1} E^{\frac{1}{2}}$.
\end{enumerate}
Note that $b_{1}(W)=0$, by condition (1), (2), (iii) and \cite[Lemma 2]{manolescu2003seiberg}, we get $$\|a\|_{L^{2}_{1}(W_{1,7})}\leq C_{2} E^{\frac{1}{2}}.$$ By the continuity of the embedding $L^{2}_{1}\hookrightarrow L^{4}$ and the multiplication $L^{4}\times L^{4}\rightarrow L^{2}$, we get $$
\|\rho(a)\Phi\|_{L^{2}(W_{1,7})}\leq C_{3}E^{\frac{3}{2}}\leq \epsilon_{0} C_{3}E^{\frac{1}{2}}
.$$
By condition (ii) and (iv), this implies $$\|\Phi\|_{L^{2}_{1}(W_{1,7})}\leq (2C_{1}+\epsilon_{0} C_{3}) E^{\frac{1}{2}}.$$
Set $\gamma=(a,\Phi)$. Then we have $$
\|\gamma\|_{L^{2}_{1}(W_{1,7})}\leq (2C_{1}+C_{2}+\epsilon_{0} C_{3})E^{\frac{1}{2}}\leq (2C_{1}+C_{2}+\epsilon_{0} C_{3})\epsilon_{0} ^{\frac{1}{2}}
$$ and $\gamma$ satisfies the equation
\begin{equation}\label{abstract equation}
P\gamma+\gamma\#\gamma=0
\end{equation}
where $$P(a,\Phi):=(da,d^{*}a,\slashed{D}^{+}_{A_{0}}\Phi)$$ is a first-order elliptic operator and $\#$ is bilinear operator involving only pointwise multiplication. Choose a cutoff function $$\beta:W_{1,7}\rightarrow [0,1]$$ that equals $1$ on $W_{2,6}$ and equals $0$ near $\partial W_{1,7}$. Then (\ref{abstract equation}) implies
$$
P(\beta\gamma)=(-\beta\gamma)\#\gamma+\sigma(P,d\beta)\gamma
.$$ Here $\sigma(P,d\beta)$ denote the symbol of $P$ evaluated at $d\beta$. One gets
\begin{equation}\label{rearrangement argument}
\begin{split}
\|\beta\gamma\|_{L^{3}_{1}(W_{1,7})}&\leq C_{4}(\|P(\beta\gamma)\|_{L^{3}(W_{1,7})}+\|\beta\gamma\|_{L^{3}(W_{1,7})})\\&=C_{4}(\|(-\beta\gamma)\#\gamma+\sigma(P,d\beta)\gamma\|_{L^{3}(W_{1,7})}+\|\beta\gamma\|_{L^{3}(W_{1,7})})\\&\leq C_{5}(\|\gamma\|_{L^{2}_{1}(W_{1,7})}\|\beta\gamma\|_{L^{3}_{1}(W_{1,7})}+\|\gamma\|_{L^{3}(W_{1,7})})\\&\leq C_{5}(2C_{1}+C_{2}+\epsilon_{0} C_{3})\epsilon_{0}^{\frac{1}{2}}\|\beta\gamma\|_{L^{3}_{1}(W_{1,7})}+C_{5}\|\gamma\|_{L^{3}(W_{1,7})}
\end{split}.
\end{equation}
Here the first inequality is by G\r{a}rding inequality for interior domain (see Theorem 5.1.4 of the book) and the third inequality uses the continuity of multiplication $L^{2}_{1}\times L^{3}_{1}\rightarrow L^{3}$. By setting $\epsilon_{0}$ small enough such that  $C_{5}(2C_{1}+C_{2}+\epsilon_{0} C_{3})\epsilon^{\frac{1}{2}}\leq \frac{1}{2}$, we obtain from (\ref{rearrangement argument})
$$
\|\gamma\|_{L^{3}_{1}(W_{2,6})}\leq \|\beta\gamma\|_{L^{3}_{1}(W_{1,7})}\leq 2C_{5}\|\gamma\|_{L^{3}(W_{1,7})}\leq C_{6}\|\gamma\|_{L^{2}_{1}(W_{1,7})}\leq C_{7}E^{\frac{1}{2}}.
$$
One can irritate this rearrangement argument twice more: first time use the continuity of $L^{2}_{2}\times L^{3}_{1}\rightarrow L^{2}_{1}$ in place of $L^{2}_{1}\times L^{3}_{1}\rightarrow L^{3}$ to bound $L^{2}_{2}$-norm of $\gamma$ and second time use the continuity of $L^{2}_{3}\times L^{2}_{2}\rightarrow L^{2}_{2}$ to bound $L^{2}_{3}$-norm of $\gamma$. Each time we pass to a slightly smaller subdomain. Note that the $L^{2}_{p}$-Sobolev space is closed under multiplication for any $p\geq 3$. As a result, once we control the  $L^{2}_{3}$-norm, the elliptic bootstraping method can be proceeded directly (without rearrangement argument) to finish the proof of the lemma.
\end{proof}
Consider the common boundary of $W_{m}$ and $W_{m+1}$ (in $X_{+}$) and denote it by $Y_{m}$. We define the \emph{modified Chern-Simons-Dirac functional} on $Y_{m}$   
$$
\tilde{\mathcal{L}}: \mathcal{C}_{k-1/2}(Y_{m})\rightarrow \mathbb{R}
$$
by the formula
$$
\tilde{\mathcal{L}}(B,\Psi):=\mathcal{L}(B,\Psi)+\int_{Y_{m}}\frac{H}{2}|\Psi|^{2}d\text{vol}
$$
where $\mathcal{L}$ is the usual Chern-Simons-Dirac functional as defined in (\ref{CSD})(with $B_{0}=A_{0}|_{Y_{m}}$) and $H$ is the mean curvature of of the psc metric on $W_{1,+\infty}$. Note that  just like $\mathcal{L}$, the functional $\tilde{\mathcal{L}}$ is also gauge invariant. Moreover, applying the Stocks formula to (\ref{topolocial energy}), one has
\begin{equation}\label{CSD decay}
\tilde{\mathcal{L}}((A,\Phi)|_{Y_{m+1}})-\tilde{\mathcal{L}}((A,\Phi)|_{Y_{m}})=\frac{1}{2}\mathcal{E}^{\text{top}}(A,\Phi)=\frac{1}{2} \mathcal{E}^{\text{an}}(A,\Phi)\geq 0
\end{equation}
for any solution $(A,\Phi)$ on $W_{m}$ ($m\geq 1$). 
\begin{lem}\label{energy controls CSD} 
Let $\epsilon_{0}$ be the constant in Lemma \ref{energy controls norm}. Then there exists another constant $D'$ such that for any solution $(A,\Phi)$ over $W_{m,m+6}$ ($m\geq 1$) satisfying $\mathcal{E}^{\text{top}}(A,\Phi)\leq \epsilon_{0}$, one has 
\begin{equation}\label{equation:engergy bound csd}
    |\tilde{\mathcal{L}}((A,\Psi)|_{Y_{m+3}})|\leq D' \mathcal{E}^{\text{top}}(A,\Phi).
\end{equation}
\end{lem}
\begin{proof}
Note that both side of (\ref{equation:engergy bound csd}) are gauge invariant. Therefore, we can apply a suitable gauge transformation $u:W_{m,m+6}\rightarrow S^{1}$ and assume that condition (1) and (2) in Lemma \ref{energy controls norm} are also satisfied. Then by that lemma, we get
$$
\|(A-A_{0},\Phi)|_{Y_{m+3}}\|_{L^{2}_{k+1/2}(Y_{m+3})}\leq C_{1}D \sqrt{\mathcal{E}^{\text{top}}(A,\Phi)}.
$$
Then the proof is finished by Cauchy-Schwarz inequality.
\end{proof}
Now we can prove the exponential decay result for energy:
\begin{pro}\label{exponential decay for energy}
For any constant $C>0$, there exists a constants $\delta(C)$ the following significance: for any $\delta\in (0,\delta(C)]$ and any solution $(A,\Phi)\in L_{k,\delta}(W_{1,\infty})$ of the unperturbed Seiberg-Witten equations with $\mathcal{E}^{top}(A,\Phi)<C$, one has
$$
\mathcal{E}^{\text{top}}((A,\Phi)|_{W_{n}})\leq 2C e^{-\delta(C)n},\ \forall n\geq 1.
$$ 
\end{pro}
\begin{proof}
For $m_{2}>m_{1}\geq 1$, we let $b_{m_{1},m_{2}}=\mathcal{E}^{\text{top}}((A,\Phi)|_{W_{m_{1},m_{2}}})$ and $a_{m_{1}}=\tilde{\mathcal{L}}((A,\Phi)|_{Y_{m_{1}}})$. Then $$0\leq b_{m_{1},m_{2}}=a_{m_{1}}-a_{m_{2}}.$$
Since $\lim_{m\rightarrow \infty}a_{m}=0$, we get $a_{m}\geq 0$ for any $m$.
We set $$N=\max\{\lceil\frac{C}{\epsilon_{0}}\rceil,2D'\},$$ where $\epsilon_{0}$ and $D'$ are the constants in Lemma \ref{energy controls norm} and Lemma \ref{energy controls CSD} respectively. For any $n\geq 0$, consider the set 
$$\{b_{6l+1,6l+7}\mid nN\leq l<(n+1)N\}.$$ 
We assume that the minimum is achieved at $b_{6l(n)+1,6l(n)+7}$. Then
$$b_{6l(n)+1,6l(n)+7}\leq \frac{b_{6nN+1,6(n+1)N+1}}{N}\leq \frac{C}{N}\leq \epsilon_{0}.$$
By Lemma \ref{energy controls CSD}, we have 
$$
a_{6(n+1)N+1}\leq a_{6l(n)+4}\leq D'b_{6l(n)+1,6l(n)+7}\leq \frac{D'}{N}b_{6nN+1,6(n+1)N+1}
$$
This implies 
$$b_{6(n+1)N+1,6(n+2)N+1}\leq a_{6(n+1)N+1}\leq \frac{D'}{N}b_{6nN+1,6(n+1)N+1}\leq \frac{1}{2}b_{6nN+1,6(n+1)N+1}.$$
Hence we have
$$
b_{6nN+1,6(n+1)N+1}\leq (\frac{1}{2})^{n}\cdot b_{1,6N+1}\leq \frac{C}{2^{n}}. 
$$
By setting $\delta(C)=\frac{\log2}{6N}$, we finish the proof.
\end{proof}

\subsection{Compactness: local result}
After preparations in last subsection, we are ready to prove compactness result for solutions on the manifold $X_{+}$. The main result of this subsection is the following theorem:

\begin{thm}\label{local compactness}
There is a constant $\delta_{4}$ such that for any $\delta\in (0,\delta_{4})$, the following compactness result holds: Let $\gamma_{n}\in\mathcal{C}^{\sigma}_{k,\delta}(X_{+})$ $(n\geq 1)$ be a sequence solutions of the perturbed equation $\mathfrak{F}^{\sigma}_{\mathfrak{p}}(\gamma)=0$. Suppose that there is a uniform bound on the perturbed topological energy:
$$
\mathcal{E}^{\textnormal{top}}_{\mathfrak{q}}(\gamma_{n})\leq C_{1},
$$
and a uniform upper bound
$$
\Lambda_{\mathfrak{q}}(\gamma_{n}|_{\{\epsilon\}\times Y})\leq C_{2}.
$$
Then there is a sequence of gauge transformations $u_{n}\in \mathcal{G}_{k+1,\delta}(X_{+})$ such that the sequence $u_{n}(\gamma_{n})|_{X'}$ converge in the topology of $\mathcal{C}^{\sigma}_{k,\delta}(X')$ to a solution $\gamma_{\infty} \in \mathcal{C}^{\sigma}_{k,\delta}(X')$.
\end{thm}

The proof of this theorem uses Lemma \ref{exp decay} and the bootstraping argument. To make it easier to follow, we break it into several lemmas. First, we let $\gamma_{n}=(A_{n},s_{n},\phi_{n})$ and $\Phi_{n}=s_{n}\phi_{n}$. The topological energy of $\gamma_{n}$ can be broken into three parts (we treat the last two terms as one part)
\begin{equation}\label{energy decomposition}
\begin{split}\mathcal{E}^{\text{top}}_{\mathfrak{q}}(\gamma_{n})=&\mathcal{E}^{\text{top}}((A_{n},\Phi_{n})|_{X_{+}\setminus W_{0}})+\mathcal{E}^{\text{top}}((A_{n},\Phi_{n})|_{W_{0}\setminus ([0,3]\times Y)})\\&+(2\mathcal{L}_{\mathfrak{q}}((A_{n},\Phi_{n})|_{\{0\}\times Y})-2\mathcal{L}((A_{n},\Phi_{n})|_{\{3\}\times Y}),\end{split}\end{equation}
where $\mathcal{L}_{\mathfrak{q}}=\mathcal{L}-v$ (see (\ref{perturbation})). We denote the first, second and third part on the right hand side of (\ref{energy decomposition}) by $\mathcal{E}_{1,n}$, $\mathcal{E}_{2,n}$ and $\mathcal{E}_{3,n}$ respectively.

\begin{lem}\label{energy bound}
The energy terms $\mathcal{E}_{1,n}$, $\mathcal{E}_{2,n}$ and $\mathcal{E}_{3,n}$ are all uniformly bounded above by red some constant $E$.
\end{lem}
\begin{proof}
Since the restriction of $(A_{n},\Phi_{n})$ on $X_{+}\setminus ([0,3]\times Y)$ is a solution of the unperturbed Seiberg-Witten equations. By the relation between the topological energy and analytical energy (see Page 96 of the book) and Lemma 24.5.1 of the book, we have the following estimates
 $$\mathcal{E}_{1,n}=\frac{1}{4}\int_{X_{+}\setminus W_{0}}(|F_{A^{t}_{n}}|^{2}+4|\nabla_{A_{n}}\Phi_{n}|^{2}+|\Phi|^{4}+\operatorname{scal}(x)|\Phi|^{2})$$
$$\mathcal{E}_{2,n}=\frac{1}{4}\int_{W_{0}\setminus ([0,3]\times Y)}(|F_{A^{t}_{n}}|^{2}+4|\nabla_{A_{n}}\Phi_{n}|^{2}+(|\Phi|^{2}+\operatorname{scal}(x)/2)^{2})-\int_{W_{0}\setminus ([0,3]\times Y)}\frac{\operatorname{scal}(x)^{2}}{16}$$
$$
\mathcal{E}_{3,n}\geq
\frac{1}{2}\int_{[0,3]\times Y}(|F_{A^{t}_{n}}|^{2}+4|\nabla_{A_{n}}\Phi_{n}|^{2}+(|\Phi_{n}|^{2}-D_{1})^{2})-D_{2}
$$
where $D_{1},D_{2}$ are certain uniform constants. Note that $\operatorname{scal}(x)$ is positive on $X_{+}\setminus W_{0}$. It is easy to see that $\mathcal{E}_{1,n}$, $\mathcal{E}_{2,n}$ and $\mathcal{E}_{3,n}$ are all uniformly bounded below. Since the sum of these three terms is bounded above, each of them should also be bounded above.
\end{proof}

For each $m\geq 0,n\geq 1$, we let $u_{m,n}:W_{m,m+6}\rightarrow S^{1}$ be a gauge transformation with the following properties (recall that $A_{0}$ denotes the spin connection):
\begin{enumerate}[(i)]
\item $d^{*}(A_{n}-u^{-1}_{m,n}du_{m,n}-A_{0})=0;$
\item $\langle A_{n}-u^{-1}_{m,n}du_{m,n}-A_{0},\vec{v}_{m}\rangle=0;$
\item $u_{1,n}(o_{3})=u_{0,n}(o_{3}),\ u_{m+2}(o_{m+4})=u_{m}(o_{m+4}),\ \forall m\geq1;$
\end{enumerate}
where $\vec{v}_{m}$ is the normal vector on $\partial W_{m,m+6}$ and $o_{m}\in W_{m}$ corresponds to a fixed base point $o\in \operatorname{int}(W)$. Such $u_{m,n}$ can always be found by solving the Laplace equation (with Neumann boundary condition) on $W_{m,m+6}$ (see Page 101 of the book). We let $$A_{m,n}=A_{n}-u^{-1}_{m,n}du_{m,n},\ \Phi_{m,n}=u_{m,n}\Phi_{n},\ \phi_{m,n}=u_{m,n}\phi_{n}.$$
 We set the constant 
\begin{equation}\label{equation: delta4}
 \delta_{4}=\min\{\delta_{3},\frac{1}{2}\delta(E)\},   
\end{equation}
where the constants $\delta_{3},\ E$ and the fucntion $\delta(-)$ are defined in Lemma \ref{exp decay}, Lemma \ref{energy bound} and Proposition \ref{exponential decay for energy} respectively.

\begin{lem}\label{bounded results}
There exist a constant $C$ such that the following bounded results hold:
\begin{enumerate}[(i)]
\item $s_{n}\leq C,\ \forall n\geq 1$;
\item $\|(A_{m,n}-A_{0},\Phi_{m,n})\|_{L^{2}_{k+1}(W_{m+1,m+5})}\leq C,\ \forall n\geq 1, m\geq 1;$
\item $\|(A_{0,n}-A_{0},\Phi_{0,n})\|_{L^{2}_{k+1}(W_{0,5}\setminus([0,\epsilon]\times Y))}\leq C,\ \forall n\geq 1;$
\item $\|(A_{m,n}-A_{0},\phi_{m,n})\|_{L^{2}_{k+1}(W_{m+2,m+4})}\leq Ce^{-\delta_{4}m},\ \forall n\geq 1,m\geq 1;$
\item $\|(A_{0,n}-A_{0},\phi_{0,n})\|_{L^{2}_{k+1}(W_{0,4}\setminus([0,2\epsilon]\times Y))}\leq C,\ \forall n\geq 1.$
\end{enumerate}
\end{lem}

\begin{proof}
(i) By Lemma \ref{energy bound}, the energy terms $\mathcal{E}_{1,n},\mathcal{E}_{2,n},\mathcal{E}_{3,n}$ are all bounded above. They provide upper bounds on $\|\Phi_{n}\|_{L^{2}(X_{+}\setminus W_{0})},\  \|\Phi_{n}\|_{L^{4}(W_{0}\setminus ([0,3]\times Y))}$ and $\|\Phi_{n}\|_{L^{4}([0,3]\times Y)}$ respectively. Hence we get an upper bound $\|\Phi_{n}\|_{L^{2}(X_{+})}$, which is exactly $s_{n}$.

(ii) Suppose we do not have a uniform bound. Then there exists sequences of positive numbers  $\{m_{l}\},\{n_{l}\}$ such that
$$\|(A_{m_{l},n_{l}}-A_{0},\Phi_{m_{l},n_{l}})\|_{L^{2}_{k+1}(W_{m_{l}+1,m_{l}+5})}\rightarrow +\infty.$$
Notice that the upper bound on $\mathcal{E}_{1,n}$ gives a uniform upper bound on $$\mathcal{E}^{\operatorname{top}}((A_{n},\Phi_{n})|_{W_{m,m+6}})$$ for all $m,n\geq 1$. Suppose we treat the restriction of $A_{m_{l},n_{l}}$ (resp. $\Phi_{m_{l},n_{l}})$) on $W_{m_{l}+1,m_{l}+5}$ as a connection (resp. spinor) on a fixed manifold $W\cup_{Y}W\cup_{Y}...\cup_{Y}W$ for each $l$. Then by Theorem 5.1.1 of the book,
they converge in $C^{\infty}$ topology after passing to a subsequence, which is a contradiction.

(iii) The proof is similar with (ii): just use  $\mathcal{E}_{2,n}+\mathcal{E}_{3,n}$ to control the topological energy on $W_{0}$ then apply Theorem 5.1.1 of the book.

(iv) By Lemma \ref{exp decay}, the norm $\|\phi_{m,n}\|_{L^{2}(W_{m+1,m+5})}$ is bounded above by $e^{-\delta_{3}m}$ (up to a constant). Then we apply the elliptic bootstraping argument on the equation
$$
\slashed{D}^{+}_{A_{0}}\phi_{m,n}+\rho(A_{m,n}-A_{0})\phi_{m,n}=0
$$
to obtain the desired bound on $\|\phi_{m,n}\|_{L_{k+1}^{2}(W_{m+2,m+4})}$. Note that (ii) is used here to give a uniform bound on $\|\rho(A_{m,n}-A_{0})\|_{L^{2}_{k+1}(W_{m+1,m+5})}$.

Next, we control the norm $\|(A_{m,n}-A_{0})\|_{L^{2}_{k+1}(W_{m+2,m+4})}$: Using Proposition 4.6,  we get the estimate $$\mathcal{E}^{top}((A_{m,n},\Phi)|_{W_{m,m+6}})\leq 10Ee^{-\delta(E)m}.$$
There exists a constant $m_{0}$ such that $10Ee^{-\delta(E)m_{0}}<\epsilon_{0}$ (the constant in Lemme \ref{energy controls norm}). When $m\geq m_{0}$, we use can Lemma \ref{energy controls norm} to bound  $\|(A_{m,n}-A_{0})\|_{L^{2}_{k+1}(W_{m+2,m+4})}$ by $e^{-\frac{\delta(E)}{2}m}$ (up to a constant). For $m<m_{0}$, we simply use the uniform bound (ii).  

(v) The uniform bound on $\|A_{0,n}-A_{0}\|_{L^{2}_{k+1}(W_{0,4}\setminus([0,2\epsilon]\times Y))}$ is trivial from (iii). To get the uniform bound on $\|\phi_{0,n}\|_{L^{2}_{k+1}(W_{0,4}\setminus([0,2\epsilon]\times Y))}$, we use the fact $$\|\phi_{0,n}\|_{L^{2}(W_{0,4}\setminus([0,2\epsilon]\times Y))}\leq 1$$ (which follows from the definition) and apply the elliptic bootstraping argument on the Dirac equation
$$
\slashed{D}^{+}_{A_{0}}\phi_{0,n}+\rho(A_{0,n}-A_{0})\phi_{0,n}=\hat{\mathfrak{p}}^{1,\sigma}(A_{0,n},\phi_{0,n})
$$
Note that the perturbation term $\hat{\mathfrak{p}}^{1,\sigma}(A_{0,n},\phi_{0,n})$ does not affect the bootstraping argument because it is tame. (See the proof Theorem 10.9.2 in the book for a similar but more detailed argument.)\end{proof}

\begin{lem}\label{regularity for solutions} Let $\delta_{4}$ be the constant in (\ref{equation: delta4}) and let $\gamma_{n}\in \mathcal{C}^{\sigma}_{k,\delta}(X_{+})$ be the solutions in Theorem \ref{local compactness}. For any $\delta'\in (\delta,\delta_{4})$, we can find gauge transformations $u_{n}\in \mathcal{G}_{k+1,\delta}$ such that the restrictions $u_{n}\cdot\gamma_{n}|_{X'}=(A'_{n},s'_{n},\phi'_{n})$ is a bounded sequence in $\mathcal{C}^{\sigma}_{k+1,\delta'}(X')$. In other words, we have $$\mathop{\sup}\limits_{n} \|(A'_{n}-A_{0},\phi'_{n})\|_{L^{2}_{k+1,\delta'}}<\infty\text{ and } \mathop{\sup}\limits_{n} s'_{n}<\infty.$$
\end{lem}
\begin{proof}
The idea is to obtain $u_{n}$ by gluing $u_{0,n}|_{W_{0,3}}$ and $u_{m,n}|_{W_{m+2,m+4}}\ (m=1,3,5...)$ together using cutoff functions. Recall that $W$ is a manifold with boundary $(-Y)\cup Y$. We choose a cutoff function $\tau:W\rightarrow [0,1]$ that equals $0$ near the left boundary and equals $1$ near the right boundary. We also use $\tau$ for the induced cutoff function on $W_{m}$. For $m,n\geq 1$, we consider the function $\xi_{m,n}:W_{m+4}\rightarrow [0,1]$ with the property that $$\xi_{m,n}(o_{m+4})=0,\ e^{i\xi_{m,n}}\cdot(u_{m,n}|_{W_{m+4}})=u_{m+2,n}|_{W_{m+4}}.$$
We also define $\xi^{0}_{n}:W_{0,3}\rightarrow [0,1]$ by the condition
$$
\xi_{0,n}(o_{3})=0,\ e^{i\xi_{0,n}}(u_{0,n}|_{W_{3}})=u_{1,n}|_{W_{3}}.
$$
Then we have $d\xi_{0,n}=A_{0,n}|_{W_{3}}-A_{1,n}|_{W_{3}}$ and $d\xi_{m,n}=A_{m+2,n}|_{W_{m+4}}-A_{m,n}|_{W_{m+4}}$ for $m\geq 1$. By Lemma \ref{bounded results} (iv) and (v), there exists a uniform constant $C$ such that $$\|\xi_{m,n}\|_{L^{2}_{k+2}(W_{m+2})}\leq Ce^{-\delta_{4}m},\ \forall m\geq 0,n\geq1,$$ which implies similar bounds for $\tau \xi_{m,n}$ and $(1-\tau)\xi_{m,n}$. We consider the gauge transformations $$ \tilde{u}_{0,n}(x)=\left\{\begin{array} {ll}

 1 & x\in W_{0,2}\\
 e^{i\tau(x)\xi_{0,n}(x)}& x\in W_{3}
\end{array}\right.$$
and
$$\tilde{u}_{m,n}(x)=\left\{\begin{array} {ll}
 e^{i(\tau(x)-1)\xi_{m-1,n}(x)}& x\in W_{m+2}  \\
 1&x\in W_{m+3}\\
 e^{i\tau(x)\xi_{m,n}(x)}& x\in W_{m+4}
\end{array}\right.\ (m\geq1)$$
By the Sobolev multiplication theorem, there exists a uniform constant $C'$ such that
\begin{itemize}
\item $
\|\tilde{u}_{0,n}^{-1}d\tilde{u}_{0,n}\|_{L^{2}_{k+1}(W_{0,3})}\leq C'$;
\item $\|1-\tilde{u}_{0,n}\|_{L^{2}_{k+1}(W_{0,3})}\leq C'
$;
\item $\|\tilde{u}_{m,n}^{-1}d\tilde{u}_{m,n}\|_{L^{2}_{k+1}(W_{m+2,m+4})}\leq C'e^{-\delta_{4}m}$ for any $m\geq 1$;
\item $\|1-\tilde{u}_{m,n}\|_{L^{2}_{k+1}(W_{m+2,m+4})}\leq C'e^{-\delta_{4}m}$ for any $m\geq 1$.
\end{itemize}
Since it is easy to check that $$u_{0,n}\cdot \tilde{u}_{0,n}=u_{1,n}\cdot \tilde{u}_{1,n} \text{ on }W_{3};$$ $$u_{m,n}\cdot \tilde{u}_{m,n}=u_{m+2,n}\cdot \tilde{u}_{m+2,n}\text{ on } W_{m+4} \text{ for } m=1,3,5...,$$ we can glue $$\{(u_{0,n}\cdot \tilde{u}_{0,n})|_{W_{0,3}}\}\cup \{(u_{m,n}\cdot\tilde{u}_{m,n})|_{W_{m+2,m+4}}|\ m=1,3,5...\}$$ together to get a guage transformation $u_{n}:X_{+}\rightarrow S^{1}$. Denote $u_{n}(A_{n},\phi_{n})$ by $(\tilde{A}_{n},\tilde{\phi}_{n})$. Then we have
$$(\tilde{A}_{n},\tilde{\phi}_{n})=\left\{\begin{array} {cc}
 \tilde{u}_{0,n}(A_{0,n},\phi_{0,n})\ \text{ on } W_{0,3}\\
 \tilde{u}_{m,n}(A_{m,n},\phi_{m,n})\ \text{ on } W_{m+2,m+4} \text{ for }m=1,3,5,...
\end{array}\right.$$ By Lemma \ref{bounded results} (iv) and (v) and the above estimates on $\tilde{u}_{m,n}$, we can prove that there exists a uniform constant $C''$ such that
$$
\|(\tilde{A}_{n}-A_{0},\tilde{\phi}_{n})\|_{L^{2}_{k+1}(W_{0,4}\setminus ([0,2\epsilon]\times Y))}\leq C''
$$
and
$$
\|(\tilde{A}_{n}-A_{0},\tilde{\phi}_{n})\|_{L^{2}_{k+1}(W_{m+2,m+4})}\leq C''e^{-\delta_{4}m}\text{ for } m=1,3,5,...
$$
Since $\delta'<\delta_{4}$, we get $$\|(\tilde{A}_{n}-A_{0},\tilde{\phi}_{n})\|_{L^{2}_{k+1,\delta'}(X'_{+})}<C'''$$
for some constant $C'''$. The relation between $(A_{n}',s'_{n},\phi_{n}')$ and $(\tilde{A}_{n},s_{n},\tilde{\phi}_{n}')$ is given by
$$
A_{n}'=\tilde{A}_{n}|_{X'_{+}},\ s_{n}'=s_{n}\cdot \|\tilde{\phi}_{n}\|_{L^{2}(X'_{+})}\text{ and }\phi_{n}'=\frac{\tilde{\phi}_{n}|_{X'_{+}}}{\|\tilde{\phi}_{n}\|_{L^{2}(X'_{+})}}.
$$
As in the proof of Theorem 24.5.2 in the book, the condition $
\Lambda_{\mathfrak{q}}(\gamma_{n}|_{\{\epsilon\}\times Y})\leq C_{2}$ ensures that the norm $\|\tilde{\phi}_{n}\|_{L^{2}_{X_{+}'}}$ (which is always less than $1$) is bounded away from $0$. Therefore, we have proved the estimate in the lemma.

We are left to check that $u_{n}\in\mathcal{G}_{k+1,\delta}(X_{+})$. We write $u_{n}$ as $e^{i\xi_{n}}$. Then $d\xi_{n}=\tilde{A}_{n}-A_{n}\in L^{2}_{k,\delta}(X_{+};i\mathds{R})$. By Lemma \ref{Taubes's lemma}, we can find $\bar{\xi}_{n}\in i\mathds{R}$ such that $\xi_{n}-\bar{\xi}_{n}\in L^{2}_{k+1,\delta}(X_{+},i\mathds{R})$. Then we have
$$
u_{n}=e^{i\bar{\xi}_{n}}\cdot e^{i(\xi_{n}-\bar{\xi}_{n})}\in \mathcal{G}_{k+1,\delta}(X_{+}).$$
\end{proof}
\begin{proof}[Proof of Proposition \ref{local compactness}]
By Lemma \ref{regularity for solutions}, we can find  $u_{n}\in \mathcal{G}_{k+1,\delta}(X_{+})$ such that $u_{n}(\gamma_{n})|_{X'_{+}}$ is a bounded sequence in $C^{\sigma}_{k+1,\delta'}(X')$. Since $\delta'>\delta$, the natural inclusion $C^{\sigma}_{k+1,\delta'}(X')\rightarrow C^{\sigma}_{k,\delta}(X')$ maps a bounded closed set to a compact set. Therefore, we can find a subsequence that converges in $C^{\sigma}_{k,\delta}(X')$.\end{proof}
\subsection{Compactness: broken trajectories} With  Theorem \ref{local compactness} (compare Theorem 24.5.2 in the book) proved, the proof of Theorem \ref{compactness} is essentially the same with the proof of Theorem 24.6.4 in the book. For completeness, we sketch it as follows:
\begin{proof}[Proof of Theorem \ref{compactness}](Sketch)
We first consider a sequence $[\gamma_{n}]\in \breve{\mathcal{M}}([\mathfrak{b}_{0}],Z_{+})$ $(n\geq 1)$ represented by unbroken $Z_{+}$-trajectories $\gamma_{n}$. Using integration by part, it is easy to see that  $\mathcal{E}_{\mathfrak{p}}^{\text{top}}(\gamma_{n}|_{X_{+}})=\mathcal{L}_{\mathfrak{q}}(\gamma_{n}|_{\{0\}\times Y})$ for any $n$, which implies $\mathcal{E}_{\mathfrak{p}}^{\text{top}}(\gamma_{n}|_{X_{+}})<\mathcal{L}_{\mathfrak{q}}([\mathfrak{b}_{0}])$ (because $\gamma|_{Z}$ is a flow line with limit $[\mathfrak{b}_{0}]$). By similar decomposition as in the proof of Lemma \ref{energy bound}, we can prove that
$$
\mathcal{L}_{\mathfrak{q}}(\gamma_{n}|_{\{\epsilon\}\times Y})=\mathcal{E}_{\mathfrak{p}}^{\text{top}}(\gamma_{n}|_{X''_{+}})>C,\ \mathcal{L}_{\mathfrak{q}}(\gamma_{n}|_{\{2\epsilon\}\times Y})=\mathcal{E}_{\mathfrak{p}}^{\text{top}}(\gamma_{n}|_{X'_{+}})>C.
$$
for some uniform constant $C$. This implies both \begin{equation}\label{energy control1}\mathcal{E}_{\mathfrak{q}}^{\text{top}}(\gamma_{n}|_{(-\infty,\epsilon]\times Y})<\mathcal{L}_{\mathfrak{q}}([\mathfrak{b}_{0}])-C\end{equation} and \begin{equation}\label{energy control2}\mathcal{E}_{\mathfrak{q}}^{\text{top}}(\gamma_{n}|_{(-\infty,2\epsilon]\times Y})<\mathcal{L}_{\mathfrak{q}}([\mathfrak{b}_{0}])-C.\end{equation}
By the same argument as proof of Lemma 16.3.1 in the book, condition (\ref{energy control1}) actually implies $$\Lambda_{\mathfrak{q}}(\gamma_{n}|_{\{t\}\times Y})\leq C',\ \forall t\in (-\infty,\epsilon]$$ for some constant $C'$. Now we apply Theorem \ref{local compactness} to show that after applying suitable gauge transformations $u_{n}:X_{+}\rightarrow S^{1}$ and passing to a subsequence, the restriction $u_{n}(\gamma_{n}|_{X_{+}})|_{X'_{+}}$ has a limit $C^{\sigma}_{k,\delta}(X')$. Since $\Lambda_{\mathfrak{q}}(\cdot)$ is gauge invariant, we get $$\Lambda_{\mathfrak{q}}(\gamma_{n}|_{\{2\epsilon\}\times Y})=\Lambda_{\mathfrak{q}}(u_{n}(\gamma_{n})|_{\{2\epsilon\}\times Y})\geq C''$$ for some uniform constant $C''$. Another application of Lemma 16.3.1 in the book provides a uniform lower bound
$$
\Lambda_{\mathfrak{q}}(\gamma_{n}|_{\{t\}\times Y})\geq C''',\ \forall t\in (-\infty,2\epsilon].
$$
Now the proof proceed exactly as in the book: We can show that after passing to a further subsequence, $\gamma_{n}|_{(-\infty,2\epsilon]}$ converges to a (possibly broken) half trajectory. Putting the two pieces $\gamma_{n}|_{(-\infty,2\epsilon]\times Y}$ and $\gamma_{n}|_{X'_{+}}$ together, we see that after passing to a subsequence and composing with suitable gauge transformations, $\gamma_{n}$ converges to a (possibly broken) $Z_{+}$-trajectory $\gamma_{\infty}$. By our regularity assumption, $\gamma_{\infty}$ can have a most one breaking point, whose absolute grading must be $2w(X,g_{X},0)$. In other words, the limit $\gamma_{\infty}$ represents a point of $\mathcal{M}^{+}([\mathfrak{b}_{0}],Z_{+})$.

We have shown that any sequence $[\gamma_{n}]\in \breve{\mathcal{M}}([\mathfrak{b}_{0}],Z_{+})$ contains convergent subsequence in  $\mathcal{M}^{+}([\mathfrak{b}_{0}],Z_{+})$. By a similar argument, we see that $\breve{\mathcal{M}}([\mathfrak{b}],Z_{+})$ contains at most finitely many elements for any $[\mathfrak{b}]$ with $\operatorname{gr}^{\mathds{Q}}([\mathfrak{b}])=-2w(X,g_{X},0)$. Since there are only finitely many critical points $[\mathfrak{b}]$ with $\operatorname{gr}^{\mathds{Q}}([\mathfrak{b}])=-2w(X,g_{X},0)$ and $\breve{\mathcal{M}}([b_{0}],[\mathfrak{b}])$ is a finite set for each of them, we see that $$\mathcal{M}^{+}([\mathfrak{b}_{0}],Z_{+})\setminus\breve{\mathcal{M}}([\mathfrak{b}_{0}],Z_{+}) =(\mathop{\cup}\limits_{ \operatorname{gr}^{\mathds{Q}}([\mathfrak{b}])=-2w(X,g_{X},0)}\breve{\mathcal{M}}([\mathfrak{b}_{0}],[\mathfrak{b}])\times \mathcal{M}([\mathfrak{b}],Z_{+}))$$
is a finite set. This finishes the proof of the theorem.
\end{proof}

\section{Proof of the theorem \ref{new obstruction}}
 Suppose $g_{X}$ has positive scalar curvature everywhere. We first prove that $-2\operatorname{h}(Y,\mathfrak{s})\leq 2\lambda_{SW}(X)$. Suppose this is not the case. Recall that $\lambda_{SW}(X)=-\omega(X,g_{X},0)$ by Lemma \ref{casson for psc}. By Assumption \ref{3 dimensional perturbation}, the perturbation $\mathfrak{q}$ is chosen so that the condition of Lemma \ref{alternative defi of Froyshov} is satisfied. As a result, we can find nonzero integers
$ n,m_{1},...,m_{l}$ and irreducible critical points $[\mathfrak{b}_{1}],...,[\mathfrak{b}_{l}]\in \mathfrak{C}^{o}$ with $\operatorname{gr}^{\mathds{Q}}([\mathfrak{b}_{l}])=-2w(X,g_{X},0)+1$ such that
\begin{equation}\label{b0 killed}
\partial^{o}_{o}(m_{1}[\mathfrak{b}_{1}]+...+m_{j}[\mathfrak{b}_{l}])=0 \text{ and } \partial^{o}_{s}(m_{1}[\mathfrak{b}_{1}]+...+m_{j}[\mathfrak{b}_{l}])=n[\mathfrak{a}_{0}].\end{equation}
Now consider the manifold $$\mathcal{M}=(\mathop{\cup}\limits_{\{l|m_{l}>0\}}m_{l}\cdot\mathcal{M}^{+}([\mathfrak{b}_{l}],Z_{+}))\cup( \mathop{\cup}\limits_{\{l|m_{l}<0\}}m_{l}\cdot \bar{\mathcal{M}}^{+}([\mathfrak{b}_{l}],Z_{+})$$
where $m_{l}\cdot*$ means the disjoint union $m_{l}$ copies and $\bar{\mathcal{M}}^{+}([\mathfrak{b}_{l}],Z_{+}))$ denotes the orientation reversal of $\mathcal{M}^{+}([\mathfrak{b}_{l}],Z_{+})$. By Theorem \ref{gluing}, Theorem \ref{orientation}, Theorem \ref{compactness} and condition (\ref{b0 killed}), $\mathcal{M}$ is an oriented, compact $1$-dimensional manifold with $$\#\partial\mathcal{M}=n\cdot\#\breve{\mathcal{M}}([\mathfrak{a}_{0}],Z_{+}),$$
where as before, $\#*$ denotes the number of points, counted with sign, in an oriented $0$-dimensional manifold. By Assumption \ref{4 dim perturbation} and Proposition \ref{only one reducible}, we get
$$\#\partial\mathcal{M}=n\cdot \pm 1=\pm n\neq 0,$$
which is impossible because we know that the number, counted with sign, of boundary points in any compact $1$-manifold should be $0$.
This contradiction finishes the proof of the inequality $\operatorname{h}(Y,\mathfrak{s})\leq 2\lambda_{SW}(X)$.

By applying the same argument to the manifold $-X$, we also get $-2\operatorname{h}(-Y,\mathfrak{s})\leq 2\lambda_{\textnormal{SW}}(-X)$, which implies $-2\operatorname{h}(Y,\mathfrak{s})\geq 2\lambda_{\textnormal{SW}}(X)$ by Lemma \ref{orientation reversal} and Lemma \ref{orientation reversal 2}. Therefore,  we have $-2\operatorname{h}(Y,\mathfrak{s})= 2\lambda_{\textnormal{SW}}(X)$  and the theorem is proved.

\appendix
\section{Laplace equation on end-periodic manifolds}
This appendix is devoted to proving Proposition  \ref{laplace equation} using Fourier-Laplace transform defined in \cite{Taubes}. Our argument closely follows with \cite{MRS} (where the corresponding problem for the Dirac operator was studied).

To begin with, let us review the definition of Fourier-Laplace tranform. Let $T:\tilde{X}\rightarrow \tilde{X}$ be the covering transformation sending $W_{m}$ to $W_{m+1}$. For $x\in \tilde{X}$ and $n\in \mathds{Z}$, we denote $T^{n}(x)$ by $x+n$. Given a function $u\in C^{\infty}_{0}(\tilde{X};\mathds{C})$ and a complex number $\mu\in \mathds{C}$, the Fourier-Laplace transform of $u$ is defined as
\begin{equation}\label{Fourier transform}
\hat{u}^{*}_{\mu}(x)=e^{\mu\tilde{f}(x)}\mathop{\Sigma}\limits_{n=-\infty}^{\infty}e^{\mu n}u(x+n).\end{equation}
(Recall that $\tilde{f}$ is the harmonic function on $\tilde{X}$ satisfying $\tilde{f}(x+1)=\tilde{f}(x)+1$.)
 It is easy to check that $\hat{u}^{*}_{\mu}(x)=\hat{u}^{*}_{\mu}(x+1)$ for any $x\in \tilde{X}$. Therefore, $\hat{u}^{*}_{\mu}$ descends to a function on $X$, which we denote by $\hat{u}_{\mu}$. A simple observation is that
 \begin{equation}\label{periodicity of Fourier-Laplace transform}
 \hat{u}_{\mu+2\pi i}(x)=e^{2\pi i f(x)}\hat{u}_{\mu}(x).
 \end{equation}
(Note that $e^{2\pi if(x)}$ is a well defined function on $X$.)

In order to recover $u$, it suffices to know $\{\hat{u}^{*}_{\mu}|\mu \in I(\nu)\}$ for any complex number $\nu$, where $I(\nu)=[\nu-\pi i,\nu+\pi i]\subset \mathds{C}$. The formula is
\begin{equation}\label{inverse Fourier transform}
u(x)=\frac{1}{2\pi i}\int_{I(\nu)}e^{-\mu \tilde{f}(x)}\hat{u}^{*}_{\mu}(x)d\mu \text{ for any }x\in \tilde{X}.
\end{equation}
For $\delta\in \mathds{R},j\in \mathds{Z}$, we denote by $L^{2}_{j;\delta,\delta}(\tilde{X};\mathds{C})$ the Hilbert space obtained from completing $C_{0}^{\infty}(\tilde{X};\mathds{C})$ with respect to the norm
$$
\|u\|_{L^{2}_{j;\delta,\delta}}:=\|e^{\delta \tilde{f}}u\|_{L^{2}_{j}}.
$$
Note that this is different with the space $L^{2}_{j;-\delta,\delta}(\tilde{X};\mathds{C})$ we considered before. We have the following lemma, which was essentially proved in \cite[Lemma 4.3]{MRS}.
\begin{lem}\label{holomorphic family}
Let $u\in L^{2}_{j;\delta,\delta}(\tilde{X};\mathds{C})$ be a smooth function. Then we have the following results:
\begin{itemize}
\item Suppose $u$ is supported in $\mathop{\cup}\limits_{n\geq m}W_{n}$ for some integer $m$.
Then we can extend the definition of the Fourier-Laplace transform and define
$$
\hat{u}_{\mu}\in L^{2}_{j}(X;\mathds{C}),\ \forall \mu \text{ with } \operatorname{Re}\mu<\delta.
$$
The family $\hat{u}_{\mu}$ is harmonic with respect to $\mu$ in the half plane $\operatorname{Re}\mu<\delta$.
\item Suppose $u$ is supported in $\mathop{\cup}\limits_{n\leq m}W_{n}$ for some integer $m$.
Then we can extend the definition of the Fourier-Laplace transform and define
$$
\hat{u}_{\mu}\in L^{2}_{j}(X;\mathds{C}),\ \forall \mu \text{ with } \operatorname{Re}\mu>\delta.
$$
The family $\hat{u}_{\mu}$ is harmonic with respect to $\mu$ in the half plane $\operatorname{Re}\mu>\delta$.
\end{itemize}
(Here harmonic means that locally we can write $\hat{u}_{\mu}$ as a power series in $\mu$, with coefficients in $L^{2}_{j}(X;\mathds{C})$, that converges in $L^{2}_{j}$.)
\end{lem}

Now we discuss the Fourier-Laplace transform of the Laplace operator. Consider the following operator
\begin{equation}\label{Fourier-Laplace operator upstairs}
C^{\infty}(\tilde{X};\mathds{C})\rightarrow C^{\infty}(\tilde{X};\mathds{C}): u\mapsto e^{\mu \tilde{f}}\Delta(e^{-\mu \tilde{f}}u)
\end{equation}
which is invariant under the covering transformation $T$. We denote by $\Delta_{\mu}$ the induced operator on $X$. It is not hard to prove that
\begin{equation}\label{periodicity of Laplace}
\Delta_{\mu+2\pi i}u=e^{2\pi i f}\Delta_{\mu}(e^{-2\pi if }\cdot u)
\end{equation}
We call $\Delta_{\mu}$ the Fourier-Laplace transform of $\Delta$ because $$
\widehat{(\Delta u)}_{\mu}=\Delta_{\mu}(\hat{u}_{\mu})\text{ for any }u\in C^{\infty}_{0}(\tilde{X};\mathds{C}).
$$
Since $\tilde{f}$ is harmonic, we have a simple formula for $\Delta_{\mu}$:
\begin{equation}\label{twisted Laplace}
\Delta_{\mu}u=\Delta u-2\mu \langle du,f^{*}(d\theta)\rangle +\mu^{2} |f^{*}(d\theta)|^{2}\cdot u,
\end{equation}
where we use the metric $g_{X}$ to define the inner product $\langle\cdot, \cdot\rangle $ and the norm $|\cdot |$ on $T^{*}X$. We can extend $\Delta_{\mu}$ to a Fredholm operator
$$
\Delta_{\mu}:L^{2}_{j+2}(X;\mathds{C})\rightarrow L^{2}_{j}(X;\mathds{C}),
$$
for any non-negative integer $j$, which we fix from now on.
Just like $\hat{u}_{\mu}$, the operator $\Delta_{\mu}$ is holomorpic in $\mu$.
\begin{lem}\label{invertible on the image axis}
$\Delta_{\mu}$ is invertible for $\mu\in i\mathds{R}\setminus 2\pi i\mathds{Z}$.
\end{lem}
\begin{proof}
Suppose $\Delta_{\mu}u=0$ for some $\mu\in i\mathds{R}\setminus 2\pi i\mathds{Z}$ and $u\in L^{2}_{j+2}(X;\mathds{C})$. Then $w=e^{-\mu \tilde{f}}\tilde{u}$ is a harmonic function on $\tilde{X}$, where $\tilde{u}$ is a lift of $u$ to $\tilde{X}$. Notice that $w(x+1)=e^{\mu}w(x)$ for any $x$, which implies $|w(x+1)=|w(x)|$ because $\operatorname{Re}\mu=0$. By maximal principle, $w$ equals a constant $C$ satisfying $C=e^{\mu}C$. Since $e^{\mu}\neq 1$, we see that $w$ (and hence $u$) must be $0$. We have proved that $\Delta_{\mu}$ has trivial kernel for any  $\mu\in i\mathds{R}\setminus 2\pi i\mathds{Z}$, which implies the lemma because the index of $\Delta_{\mu}$ is always $0$.
\end{proof}
By \cite[Lemma 4.5]{Taubes}, we have the following corollary.
\begin{cor}\label{invertible near y-axis}
$\Delta_{\mu}$ is invertible for all $\mu \in \mathds{C}\setminus S$ where $S$ is a discrete set invariant under the translation $\mu\mapsto \mu+2\pi i$. $S$ does not intersect $i\mathds{R}\setminus 2\pi i\mathds{Z}$ and has no accumulation point (in $\mathds{C}$). In particular, there exists a constant $\delta_{4}>0$ such that $\Delta_{\mu}$ is invertible for any $\mu$ satisfying
$
-\delta_{4}\leq \operatorname{Re}\,\mu\leq \delta_{4},\  \mu\notin
2\pi i\mathds{Z}.
$
\end{cor}

Denote by $R_{\mu}$ the inverse of the holomorphic family $\Delta_{\mu}$. Then $R_{\mu}$ is holomorphic in $\mu$ for $\mu$ in the region $\mathds{C}\setminus S$. By (\ref{Fourier-Laplace operator upstairs}), we have
\begin{equation}\label{periodicity of inverse Laplace}
R_{\nu+2\pi i}(e^{2\pi if}u)=e^{2\pi i f}\cdot R_{\nu}(u).
\end{equation}

\begin{lem}\label{0 is second order}
$0$ is a pole of $R_{\mu}$ with order $2$.
\end{lem}
\begin{proof}
We can write down the explicit formula of $R_{\mu}$ for small $\mu$. To do this, we consider the Hilbert space
$$
V_{l}:=\{u\in L^{2}_{l}(X;\mathds{C})| \int u\,d\text{vol}=0\}
$$
for $l=j,j+2$ and fix the identification
$$
L^{2}_{l}(X;\mathds{C})\cong V_{l}\oplus \mathds{C}: u\leftrightarrow (u-\frac{\int u\,d\text{vol}}{\text{vol}(X)}, \frac{\int u\,d\text{vol}}{\text{vol}(X)}).
$$
By elementary calculation involving (\ref{twisted Laplace}), we see that under this identification, the operator $\Delta_{\mu}:V_{j+2}\oplus \mathds{C}\rightarrow V_{j}\oplus \mathds{C}$ can be represented by the matrix
$$\left(\left(\begin{array} {cc}
 \Delta|_{V_{j+2}} & D_{3} \\
 0 & C
\end{array}\right)+ \left(\begin{array} {cc}
 -2\mu D_{1}+ \mu^{2}D_{2} & 0 \\
 \mu^{2}D_{4} & 0
\end{array}\right)\right)\cdot \left(\begin{array} {cc}
 1 & 0 \\
 0 & \mu^{2}
\end{array}\right) $$
where $C= \tfrac{\int|f^{*}(d\theta)|^{2}\,d\text{vol}}{\text{vol}(X)}\in  \mathds{R}$ and $D_{i}$ $(i=1,2,3,4)$ are certain bounded operators (independent of $\mu$) whose specific forms are not important for us. Since $C>0$ and $\Delta|_{V_{j+2}}:V_{j+2}\rightarrow V_{j}$ is an isomorphism, we see that when $|\mu|$ is small, the operator
$$
\left(\begin{array} {cc}
 \Delta|_{V_{j+2}} & D_{3} \\
 0 & C
\end{array}\right)+ \left(\begin{array} {cc}
 -2\mu D_{1}+ \mu^{2}D_{2} & 0 \\
 \mu^{2}D_{4} & 0
\end{array}\right)
$$
is invertible and $R_{\mu}$ equals
$$\left(\begin{array} {cc}
 1 & 0 \\
 0 & \mu^{-2}
\end{array}\right)\cdot \left(\left(\begin{array} {cc}
 \Delta|_{V_{j+2}} & D_{3} \\
 0 & C
\end{array}\right)+ \left(\begin{array} {cc}
 -2\mu D_{1}+ \mu^{2}D_{2} & 0 \\
 \mu^{2}D_{4} & 0
\end{array}\right)\right)^{-1}.$$
This finishes the proof.
\end{proof}
Now we come to the key lemma in our argument.
\begin{lem}\label{Solving laplace equation on covering space}
There exists a (small) constant $\delta_{5}>0$ with the following significance: any function $u\in L^{2}_{\operatorname{loc}}(\tilde{X};\mathds{C})$ satisfying
\begin{equation}\label{exp bounded}\int e^{-\delta_{5}|\tilde{f}|}|u|^{2}d\operatorname{vol}<+\infty,\  \Delta u=0\end{equation} should equal  $a+b\tilde{f}$ for some $a,b\in \mathds{C}$.
\end{lem}
\begin{proof}
Let $\delta_{5}$ be any positive constant smaller than the constant $\delta_{4}$ in Corollary \ref{invertible near y-axis}. Suppose we have a harmonic function $u$ satisfying (\ref{exp bounded}). By standard elliptic bootstrapping argument, $u$ is smooth and belongs to $L^{2}_{j+2;\,\delta_{5},-\delta_{5}}(\tilde{X};\mathds{C})$. We choose a smooth cut-off function $\zeta: \tilde{X}\rightarrow [0,1]$ with the property that
$$\zeta\equiv1 \text{ on }\mathop{\cup}\limits_{n\leq -1}W_{n},\  \zeta\equiv 0\text{ on }\mathop{\cup}\limits_{n\geq 1}W_{n}.$$ Consider the functions $$v=(1-\zeta)u\in L^{2}_{j;-\delta_{5},-\delta_{5}}(\tilde{X};\mathds{C}),\ w=\zeta\cdot u\in L^{2}_{j;\,\delta_{5},\delta_{5}}(\tilde{X};\mathds{C}).$$
Then we have $$\Delta v=\kappa =-\Delta w$$ where $\kappa$ is a smooth function supported on $W_{0}$. Applying the Fourier-Laplace transform, by Lemma \ref{holomorphic family}, we get
$$
\Delta_{\mu}\hat{v}_{\mu}=\hat{\kappa}_{\mu} \text{ for }\text{Re}\,\mu< -\delta_{5}\text{ and } \Delta_{\mu}\hat{w}_{\mu}=-\hat{\kappa}_{\mu} \text{ for }\text{Re}\,\mu> \delta_{5}.
$$
By Lemma \ref{invertible near y-axis}, this implies
$$
\hat{v}_{\mu}=R_{\mu}\hat{\kappa}_{\mu} \text{ for }\mu\in I(-\delta_{4})\text{ and } \hat{w}_{\mu}=-R_{\mu}\hat{\kappa}_{\mu} \text{ for }\mu\in I(\delta_{4}).
$$
(Recall that $I(\nu)=[\nu-\pi i,\nu+\pi i]$.)
Now we use the Fourier inversion formula (\ref{inverse Fourier transform}) to get
\begin{equation}
u(x)=v(x)+w(x)=\frac{1}{2\pi i}(\int_{I(-\delta_{4})}e^{-\mu\tilde{f}(x)}(R_{\mu}\hat{\kappa}_{\mu})^{*}(x)d\mu-\int_{I(\delta_{4})}e^{-\mu\tilde{f}(x)}(R_{\mu}\hat{\kappa}_{\mu})^{*}(x)d\mu) \end{equation}
where $(R_{\mu}\hat{\kappa}_{\mu})^{*}$ denotes the lift of $R_{\mu}\hat{\kappa}_{\mu}$ to $\tilde{X}$.
Notice that the function $e^{-\mu\tilde{f}}(R_{\mu}\hat{\kappa}_{\mu})^{*}$ can actually be defined for any $\mu \in \mathds{C}\setminus S$ and is invariant under the translation $\mu\mapsto \mu+2\pi i$ by (\ref{periodicity of Fourier-Laplace transform}) and (\ref{periodicity of inverse Laplace}). As a result, we have
\begin{equation}\label{residue}
u(x)=\frac{1}{2\pi i}\int_{\partial \Gamma}e^{-\mu\tilde{f}(x)}(R_{\mu}\hat{\kappa}_{\mu})^{*}(x)d\mu=\text{Res}_{0}(e^{-\mu\tilde{f}(x)}(R_{\mu}\hat{\kappa}_{\mu})^{*}(x)),
\end{equation}
where $\partial\Gamma$ denotes the boundary of the domain $$\Gamma:=\{x+yi|(x,y)\in [-\delta_{4},\delta_{4}]\times [-\pi,\pi]\}$$ and $\text{Res}_{0}(e^{-\mu\tilde{f}(x)}(R_{\mu}\hat{\kappa}_{\mu})^{*}(x))$ denotes the residue at $\mu=0$ of $e^{-\mu\tilde{f}(x)}(R_{\mu}\hat{\kappa}_{\mu})^{*}(x)$ as a harmonic function on $\mu$ (with $x$ fixed). Here we use the fact that $0$ is the only pole in $\Gamma$ (see Corollary \ref{invertible near y-axis}).

To compute the residue, we consider the Laurent series of $R_{\mu}\hat{\kappa}_{\mu}$ near $\mu=0$. Since $\kappa$ is compactly supported and smooth, $\hat{\kappa}_{\mu}$ is holomorphic over the whole complex plane.  By Lemma \ref{0 is second order}, we can write
the Laurent series as
$$
R_{\mu}\hat{\kappa}_{\mu}=\mathop{\Sigma}^{\infty}_{l=-2}h_{l}\cdot \mu^{l}
$$
with $h_{l}\in L^{2}_{j+2}(X;\mathds{C})$. By (\ref{twisted Laplace}), we have
$$
\kappa_{\mu}=\Delta_{\mu}R_{\mu}\hat{\kappa}_{\mu}=\Delta h_{-2}\cdot \mu^{-2}+(\Delta h_{-1}-2\langle dh_{-2},f^{*}(d\theta)\rangle)\cdot \mu^{-1}+\mathop{\Sigma}^{\infty}_{l=0}c_{l}\cdot \mu^{l}.
$$
Since $\hat{\kappa}_{\mu}$ is holomorphic at $0$. We get
$$
\Delta h_{-2}=\Delta h_{-1}-2\langle dh_{-2},f^{*}(d\theta)\rangle=0.
$$
This implies that $h_{-1},h_{-2}$ are both constant functions. Since
$$
\text{Res}_{0}(e^{-\mu\tilde{f}(x)}(R_{\mu}\hat{\kappa}_{\mu})^{*}(x))=h_{-1}^{*}(x)-h^{*}_{-2}(x)\tilde{f}(x)
$$
where $h_{-1}^{*},h_{-2}^{*}$ are lifts of $h_{-1}$ and $h_{-2}$. The lemma is proved.
\end{proof}
\begin{proof}[Proof of Proposition \ref{laplace equation}]Let $\delta_{0}$ a positive number less than $\min(\delta_{4},\delta_{5},\delta_{6})$, where $\delta_{4},\delta_{5}$ are constants in Lemma \ref{invertible near y-axis} and Lemma \ref{Solving laplace equation on covering space} respectively and $\delta_{6}$ equals the smallest positive eigenvalue of $\Delta(Y)$ (the Laplace operator over $Y$). For any $\delta\in (0,\delta_{0})$, let us check the Fredholm properties of the operators one by one.

First consider $\Delta(\tilde{X};-\delta,\delta)$, since $\Delta_{\mu}$ is invertible for any $\mu$ with $\text{Re}\,\mu=\pm \delta$, by \cite[Lemma 4.3]{Taubes}, the operator $\Delta(\tilde{X};-\delta,\delta)$ is Fredholm. By the maximum principle,  the operator $\Delta(\tilde{X};-\delta,\delta)$ has trivial kernel. The cokernel of $\Delta(\tilde{X};-\delta,\delta)$ is isomorphic to the kernel of its adjoint operator, which is the Laplace operator from $L^{2}_{-j-2;\,\delta,-\delta}(\tilde{X},\mathds{R})$ to $L^{2}_{-j;\,\delta,-\delta}(\tilde{X},\mathds{R})$. By elliptic bootstrapping and Proposition \ref{Solving laplace equation on covering space}, we get
$$\text{dim}(\operatorname{coker}\,\Delta(\tilde{X};-\delta,\delta))=2.$$

Now consider the operator $\Delta(M_{+};\delta)$. By \cite[Lemma 4.3]{Taubes} again, this operator is Fredholm. To compute its index, we consider its adjoint operator $\Delta(M_{+};-\delta)$. We have
$$
2\operatorname{ind}\,\Delta(M_{+};\delta)=\operatorname{ind}\,\Delta(M_{+};\delta)-\operatorname{ind}\,\Delta(M_{+};-\delta)=\operatorname{ind}\,\Delta(\tilde{X};-\delta,\delta)=2.
$$
The second equality above uses the excision principle of index (see \cite[Proposition 6.1]{MRS}). By the maximum principle, the kernel of $\Delta(M_{+};\delta)$
is trivial. Therefore, the operator $\Delta(M_{+};\delta)$ has $1$-dimensional cokernel.

Now we consider the operator $\Delta(X_{+};\delta)$. By classical results on the Laplace equation with Neumann boundary condition, both the operator
$$
\Delta(\bar{M}):L^{2}_{j+2}(\bar{M};\mathds{R})\rightarrow L^{2}_{j}(\bar{M};\mathds{R})\oplus L^{2}_{j+1/2}(Y,\mathds{R})$$ $$u\mapsto (\Delta u,\langle du,\vec{v}\rangle)$$
and the operator
$$
\Delta(M\cup_{Y}\bar{M}):L^{2}_{j+2}(M\cup_{Y}\bar{M};\mathds{R})\rightarrow L^{2}_{j}(M\cup_{Y}\bar{M};\mathds{R})$$ $$u\mapsto \Delta u$$
are Fredholm with index $0$. By the excision principle relating
$\Delta(X_{+};\delta)\oplus \Delta(M\cup_{Y}\bar{M})$ with $\Delta(\bar{M})\oplus\Delta(M_{+};\delta)$, we see that $\Delta(X_{+};\delta)$ is Fredholm and
$$
\operatorname{ind}\Delta(X_{+};\mathds{R})=\operatorname{ind}\Delta(M\cup_{Y} \bar{M})+\operatorname{ind}\Delta(M_{+};\delta)-\operatorname{ind}\Delta(M\cup_{Y}\bar{M})=-1.
$$
Suppose $u\in \operatorname{ker}\Delta(X_{+};\mathds{R})$. Then $d*du=0$ and $i^{*}(*du)=0$, where $i:Y\rightarrow X_{+}$ is inclusion of the boundary. We have
$$
\int_{X_{+}} \langle du,du\rangle d\text{vol}=-\int_{X_{+}}  du\wedge *du=-\int _{X_{+}}u\wedge (d*du)+\int_{Y}u\cdot i_{*}(*du)=0,
$$
In other words, $u$ is a constant function. Because $u\in L^{2}_{j+2;\delta}(X_{+};\mathds{R})$, we have $u=0$. Thus $\Delta(X_{+};\mathds{R})$ has trivial kernel and $1$-dimensional cokernel.

We are left with the manifold $Z_{+}$. The argument is similar: Notice that we set $\delta_{0}$ to be less than the first positive eigenvalue of $\Delta(Y)$, which also equals the first positive eigenvalue of $\Delta(S^{1}\times Y)$. By \cite[Lemma 4.3]{Taubes} again,  the Laplace operator $\Delta(Z_{+};-\delta,\delta)$ is Fredholm. To conclude that its index equals $-2$, we apply the excision principle relating the manifold $Z_{+}\cup (M\cup_{Y} \bar{M})$ with the manifold $M_{+}\cup M_{-}$ (recall that $M_{-}=((-\infty,0]\times Y)\cup_{Y}\bar{M}$). Then we use the maximum principle (or integration by part) to prove that the kernel is trivial.
\end{proof}

\bibliographystyle{plain}
\bibliography{B-bib}
\end{document}